\documentclass[a4paper,reqno]{amsart}
\usepackage[english]{babel}
\usepackage{amsmath, amsthm, amsfonts,amssymb}
\usepackage[bookmarks,bookmarksopen,bookmarksnumbered,colorlinks,linkcolor=blue,citecolor=red,urlcolor=grey,breaklinks]{hyperref}
\usepackage{enumitem}
\usepackage{graphicx}
\usepackage{accents}
\usepackage{mathtools}
\usepackage{relsize}

\usepackage[mathscr]{euscript}
\usepackage{color}
\newcommand{\Fpred}[1]{#1_{\mathrm{pred}}}
\newcommand{\Fsucc}[1]{#1_{\mathrm{succ}}}
\newcommand{\hide}[1]{}

\newcommand{\half}{\mathbb{H}}

\newcommand{\sm}{\setminus}
\newcommand{\eps}{\varepsilon}

\newcommand{\Newpage}
{\newpage}

\theoremstyle{definition}
\newtheorem{Def}{Definition}[section]
\newtheorem{LemDef}[Def]{Lemma and Definition}

\newtheorem{CorDef}[Def]{Corollary and Definition}
\newtheorem*{notation}{Notation}
\newtheorem*{conv}{Convention}

\newtheorem{remarknumber}[Def]{Remark}

\theoremstyle{plain}

\newtheorem{Thm}[Def]{Theorem}
\newtheorem{Lem}[Def]{Lemma}
\newtheorem{Cor}[Def]{Corollary}
\newtheorem{Pro}[Def]{Proposition}

\newtheorem{maintheorem}{Theorem}

\theoremstyle{remark}
\newtheorem*{remark}{Remark}

\newcommand{\C}{\mathbb{C}} 
\newcommand{\CC}{\hat{\mathbb{C}}} 
\newcommand{\CT}[1]{\mathbb{C}_{#1}} 
\newcommand{\CTT}[1]{\mathbb{C}_{{#1}^{\mathsmaller{\mathsmaller{\infty}}}}} 
\newcommand{\D}{\mathbb{D}} 
 
\newcommand{\N}{\mathbb{N}} 
\newcommand{\Z}{\mathbb{Z}}

\newcommand{\extended}[1]{\widehat{#1}}

\newcommand{\unbounded}[1]{\accentset{\infty}{#1}}

\DeclareMathOperator{\landing}{land}

\newcommand{\landeq}{\sim_{\landing}}

\makeatletter
\def\moverlay{\mathpalette\mov@rlay}
\def\mov@rlay#1#2{\leavevmode\vtop{
		\baselineskip\z@skip \lineskiplimit-\maxdimen
		\ialign{\hfil$\m@th#1##$\hfil\cr#2\crcr}}}
\newcommand{\charfusion}[3][\mathord]{
	#1{\ifx#1\mathop\vphantom{#2}\fi
		\mathpalette\mov@rlay{#2\cr#3}
	}
	\ifx#1\mathop\expandafter\displaylimits\fi}
\makeatother

\makeatletter
\let\sv@endpart\@endpart
\def\@endpart{\thispagestyle{empty}\sv@endpart}
\makeatother

\makeatletter
\newcommand*{\rom}[1]{\expandafter\@slowromancap\romannumeral #1@}
\makeatother

\newcommand{\cupdot}{\charfusion[\mathbin]{\cup}{\cdot}}
\newcommand{\bigcupdot}{\charfusion[\mathop]{\bigcup}{\cdot}}

\renewcommand{\Re}{\operatorname{Re}}

\DeclareMathOperator{\It}{It}    
\DeclareMathOperator{\cl}{cl}   
\DeclareMathOperator{\Ad}{Ad}   

\DeclareMathOperator{\Per}{PreP}

\newcommand{\ad}[1]{\underline{#1}} 
\newcommand{\AD}{\mathbf{S}} 
\newcommand{\ADPer}{\mathbf{S}^{\Per}} 
\newcommand{\Crit}[1]{\mathcal{C}(#1)} 
\newcommand{\itn}[1]{\mathtt{\underline{#1}}}

\newcommand{\itin}[2]{\It(#1~\vert~#2)}
 
\newcommand{\itinc}[1]{\It(#1~\vert~\mathcal{I})}

\newcommand{\itinl}[2]{\It^-(#1~\vert~#2)} 
\newcommand{\itinr}[2]{\It^+(#1~\vert~#2)}

\newcommand\restr[2]{\ensuremath{\left.#1\right|_{#2}}}

\newcommand{\EL}{\mathcal{B}} 
\newcommand{\internal}{\beta} 
\newcommand{\dread}[2]{G_{#1}[#2]}

\numberwithin{equation}{section}

\title[Combinatorics of filament pairs]{Filament Pairs, Dynamic Partitions, and Spiders for Post-Singularly Finite Entire Functions}

\author{David Pfrang}
\address{prognostica GmbH, Berliner Platz~6, 97080~W\"urzburg, Germany}
\email{david.pfrang@prognostica.de}
\author{S\"oren Petrat}
\address{Department of Mathematics and Logistics, Constructor University, Campus Ring~1, 28759 Bremen, Germany}
\email{\texttt{s.petrat@constructor.university}}
\author{Bernhard Reinke}
\address{Aix--Marseille Universit\'e, Institut de Math\'ematiques de Marseille, 163 Avenue de Luminy Case 901, 13009 Marseille, France
}
\email{bernhard.reinke@univ-amu.fr}
\author{Dierk Schleicher}
\address{Aix--Marseille Universit\'e, Institut de Math\'ematiques de Marseille, 163 Avenue de Luminy Case 901, 13009 Marseille, France}
\email{\texttt{dierk.schleicher@univ-amu.fr}}

\thanks{We gratefully acknowledge that this project was supported by a grant from the Deutsche Forschungsgemeinschaft (DFG Project number 316866235), as well as by the Advanced Grant ``Hologram'' of the European Research Council.}

\begin{document}
	
	\begin{abstract}
		Filaments are a natural generalization of the well-known concept of dynamic rays in complex dynamics. In this article we investigate which periodic or preperiodic filaments land together for arbitrary post-singularly finite transcendental entire functions. Our first main result is a combinatorial description of the landing relation of filaments in terms of the dynamic partitions of the space of external addresses. One of the main difficulties deals with taming the more complicated topology of filaments. In the end, filaments possess all the topological properties of dynamic rays that are essential for the construction of dynamic partitions. The results of this paper are the foundation for the development of combinatorial models, in particular homotopy Hubbard trees, for arbitrary post-singularly finite transcendental entire functions.

Our second mail result is that every postsingularly finite entire function has an iterated that possesses and invariant spider: spiders are, like homotopy Hubbard trees, an important tool for the combinatorial classification of postsingularly finite polynomials and presumably also for entire functions.
	\end{abstract}
	
	\maketitle
	
	\tableofcontents
	
	\section{Introduction}
	
	For the dynamics of iterated polynomials, the Julia sets tend to have very complicated topology, but they can often be described successfully in terms of symbolic dynamics. The underlying partition comes from pairs or groups of dynamic rays that ``land'' at a common point in the Julia set and thus decompose the Julia set into several parts. These ideas are the foundation for quite a lot of deep work, including the famous puzzle theory developed by Yoccoz and others.
	
	At the basis of this work are dynamic rays, invariant curves consisting of points that ``escape'', i.e.\ converge to $\infty$ under iteration. This is facilitated by the fact that, notably in the important case that the Julia set is connected, the set of escaping points is a disk around infinity (in the Riemann sphere) with very simple dynamics, coming from the fact that the point at $\infty$ is a superattracting fixed point.
	
	The situation is far more complicated for the dynamics of transcendental entire functions. The point at $\infty$ is an essential singularity, so even the function itself has wild behavior near $\infty$, let alone its dynamics. Yet the goal remains to decompose the Julia set (which may well be the entire complex plane) into natural pieces with respect to which one may introduce symbolic dynamics.
	
	For a large class of entire functions, it has been shown in \cite{RRRS} that dynamic rays exist (sometimes also called ``hairs''): these are maximal curves in $\C$ consisting of escaping points. When two such rays land at a common point, they decompose $\C$ and thus the Julia set as desired. However, there are entire functions that do not have any curves in the escaping set \cite[Theorem~8.4]{RRRS}, so rays do not exist in these cases. However, it has been shown in \cite{BR} that a more general structure called \emph{filaments} does exist in many cases: these are invariant continua consisting of escaping points, possibly not containing any curves, but they may still ``land'' in pairs or groups. Despite their possibly complicated topology, they can still partition $\C$ in exactly the same way as before, and thus lead to symbolic dynamics that makes it possible to describe the Julia set in combinatorial terms.
	
{A good number of results in transcendental dynamics depends on the assumption that maps have finite order of growth, which is a condition in \cite{RRRS} for the fact that all escaping points are organized in the form of rays. However, in many cases the rays themselves are not required for various properties, but only the structure that these rays give to the dynamical plane. A very similar structure can be obtained by filaments, so many known results should hold beyond finite order, or at least have natural analogues. }

	One of our conceptual goals is to introduce \emph{Hubbard trees} for post-singularly finite transcendental entire functions: these are invariant trees known from polynomial dynamics that are instrumental for distinguishing and classifying polynomial mappings. The present paper is the first in a sequence, based on the first author's PhD thesis \cite{DavidThesis}, that will accomplish this goal (up to homotopy); here, we prepare the way for symbolic dynamics that is of interest in its own right, and it lays the foundations for the development of homotopy Hubbard trees.

Another concept for symbolic dynamics of polynomials are \emph{spiders} as introduced in \cite{Spiders}: in a way, Hubbard trees describe polynomials ``from the inside'', while spiders describe them ``from the outside'', and both concepts are (at least in principle) equivalent. 

Here is our first main result; for detailed explanation, see Section~\ref{section:Spiders}.

\begin{maintheorem}[Existence of spiders]
\label{maintheorem:ExistenceSpiders}
Every postsingularly finite transcendental entire function has a spider, and it has an iterate that has an invariant spider. 
\end{maintheorem}

Our second main result is the following (stated in a simplified form).
	
	\begin{maintheorem}[Landing equivalence]
\label{maintheorem:landingequivalence}
		Let $f$ be an arbitrary post-singularly finite transcendental entire function. Then there exists an iterate $f^{\circ n}$ such that two periodic or preperiodic rays or filaments land together if and only if their external addresses $\ad s$ and $\ad t$ have equivalent itineraries with respect to a certain dynamic partition.
	\end{maintheorem}

A precise version of our main result will be proved below as Theorem~\ref{thm:landingEquivalence}. 

Let us make a few explanatory comments. First of all,
	our result holds for arbitrary post-singularly finite entire functions, without any of the usual restrictions such as on its growth (in terms of a ``finite order'' condition). In particular, there may or may not be rays among the escaping points; the escaping points are always organized in the form of rays or filaments \cite[Corollary 4.5]{BR}. It is known \cite[Theorem 8.1]{BR} that in the post-singularly finite case, every periodic ray, and every periodic filament, \emph{lands} in the sense that the closure in $\C$ of the ray or filament contains exactly one additional point, which is the landing point. 
	Moreover, we show in Section~\ref{section:TopologyOfFilaments} that if $k\ge 1$ rays or filaments land at a common point, then their union, together with the common landing point, decomposes $\C$ into exactly $k$ connected components, all of which are open. This provides a combinatorially controlled dynamically meaningful decomposition of $\C$ (and hence of the Julia set), which provides the foundation for further work, in particular our existence proof of homotopy Hubbard trees as developed in \cite{DavidThesis}. 
	
Technically, our result only describes the landing relation, and invariant spiders, for an iterate of $f$, it is useful for $f$
	itself because the filaments of $f$ and $f^{\circ n}$ are the same \cite[Observation 4.13]{BR}. There are two reasons why we replace the function $f$ by an iterate. The first is that a repelling fixed point may not be the landing point of a fixed ray or a fixed filament, but always of a periodic ray or filament (and analogous results are true for periodic points). For an appropriate iterate, the landing rays or filaments are fixed. The second reason occurs for superattracting fixed (or periodic) points: these are of course not landing points of rays or filaments, so we construct \emph{extended filaments} that consist of a filament that lands on the boundary of the Fatou component, and in internal ray within this Fatou component. There are only finitely many fixed internal rays around every superattracting periodic points, and if the landing points of these internal rays on the boundary of the Fatou components happen to be singular values as well, then the (extended) filaments to each singular value fail to be disjoint. This is remedied by choosing a periodic internal ray that lands at a boundary point that is not in the postsingular set. 

	Our paper is organized as follows. In Section~\ref{Sec:Background}, we introduce some general facts about the dynamics of
	post-singularly finite transcendental entire functions. More importantly, we introduce an extension of the complex plane by adding
	iterated preimages of infinite degree for asymptotic values as limit points of asymptotic tracts. This extension is necessary
	in order to deal with preperiodic rays and filaments: even for the simplest case of exponential maps, there are preperiodic dynamic rays that
	do not land anywhere in the finite plane, but that do land in a meaningful way ``at infinity'' in our extension of the plane.
	
	In Section~\ref{section:PsfDynamics}, we collect results about the dynamics of post-singularly finite entire functions, specifically about the Fatou set and the escaping set. We give an overview of results established in \cite{BR} regarding the canonical decomposition of the escaping set into filaments
	and the combinatorial description of the escaping dynamics via external address.
	
	In Section~\ref{section:LandingOfFilaments}, we introduce the concept of \emph{landing} of filaments as established in \cite[Definition 6.4]{BR} and adjust
	it in order to take into account preperiodic filaments landing at points in our extension of the complex plane. One might think that once the landing of all periodic filaments has been introduced in \cite{BR}, it follows automatically that all preperiodic filaments land as well. This is indeed true when the landing point is in $\C$, but we need to be more careful for those preperiodic filaments that land ``at infinity''. The resulting landing results for preperiodic filaments might be useful in their own right.
	
	Section~\ref{section:TopologyOfFilaments} is about establishing some topological facts about filaments. Most notably, we show that filaments that
	land together separate the plane in the same way as dynamic rays. We also show that every preperiodic point in the extended plane is the landing point
	of a preperiodic filament and vice versa.
	
	In Section~\ref{section:DynamicalPartitions}, we introduce dynamic partitions both in the plane and in the space of external addresses of external addresses and establish
	a natural bijection between partition sectors in the plane and in the space of external addresses, and investigate the topology of partition sectors.
	
In Section~\ref{section:Spiders} we establish the existence of {what we call} \emph{simple dynamic partitions} and show that
	every post-singularly finite entire function has an iterate that admits a simple dynamic partition. An immediate consequence is the existence of an invariant spider for this iterate, and thus the existence of a spider for the original function; this proves Main Theorem~\ref{maintheorem:ExistenceSpiders}.

	In Section~\ref{section:Itineraries}, we introduce itineraries with respect to dynamic partitions. Essentially, the itinerary of a point is the sequence of
	partition sectors into which the point is mapped under iteration of the function. However, there are some boundary cases that need to be discussed and certain
	itineraries need to be identified via an adjacency relation because they are in a certain sense realized by the same point.

Section~\ref{section:LandingEquivalence} is the conclusion of this paper. Using the simple dynamic partitions, we describe
	which (pre\nobreakdash)periodic filaments land together via an explicit equivalence relation on the space of external addresses. Essentially, two filaments
	land together if and only if their external addresses have the same itinerary, but the preperiodic case is a bit more complicated because of the existence of
	critical points that lie on the boundary of the dynamic partition. This yields a proof of Main Theorem~\ref{maintheorem:landingequivalence}.

\emph{Note.} The concept of \emph{filament} was originally introduced in \cite{BR} under the name of \emph{dreadlock}, and we had used that name also in a first preprint version of this paper. 

\emph{Acknowledgements}.
This project owes a lot to interesting discussions with many people. Anna Miriam Benini and Lasse Rempe developed the concept of filaments and generously shared their expertise with us. With Dima Dudko and Kostya Drach we have had many interesting discussions on this project (and more); Kostya also provided most of the pictures. Finally, we had valuable discussions with Michael Rothgang especially during the initial phase of the project, which is based on a previous joint paper with him \cite{PRS}. We are most grateful to all of them.

\section{Background and conventions}
	\label{Sec:Background}	

Let us start with the following convention.
	\begin{conv}
	When we speak about an entire function $f$ without further qualification, we mean a \emph{transcendental} entire function.
	\end{conv}
	
	For an entire function $f$, a \emph{critical point} is a point $w\in\C$ with $f'(w)=0$; the associated image $f(w)$ is called a \emph{critical value}. An \emph{asymptotic value} is a point $a\in\C$ such that there is a curve $\gamma\colon[0,\infty)\to\C$ for which, as $t\to\infty$, we have $\gamma(t)\to\infty$ and $f(\gamma(t))\to a$. More generally, a \emph{singular value} is a point $a\in\C$ for which there does not exist a radius $r>0$ so that $f^{-1}(D_r(a))$ is a union of disjoint topological disks so that $f$ maps each of them homeomorphically onto $D_r(a)$. The function is of \emph{finite type}
	if it has only finitely many singular values. In this case, every singular value is either a critical value or an asymptotic value \cite{S2}. We denote the \emph{set of singular values} of $f$ by $S(f)$. It is well-known that $f\colon\C\setminus f^{-1}(S(f))\to\C\setminus S(f)$ is a covering \cite[Lemma 1.1]{GK}. As we will often consider branched coverings over a punctured disk, we give here a classification of such branched covers up to conformal equivalence.
	\begin{Lem}[Coverings of punctured conformal disks]\label{lem:PuncturedCoverings}
	Let $f$ be an entire function, and let $U\subsetneq\C$ be a simply connected domain such that $U\cap S(f)\subset\{a\}$. Let $V$ be a connected component of $f^{-1}(U)$. Then $V$ is simply connected, and exactly one of the following cases is true:
	\begin{enumerate}
		\item There exist biholomorphic maps $\psi\colon V\to\D$ and $\phi\colon U\to\D$ and an integer $d\ge 1$ such that $\phi(a)=0$ and $\phi\circ f\circ \psi^{-1}(z)=z^d$ for all $z\in\D$.
		\item There exist biholomorphic maps $\psi\colon V\to \half$ and $\phi\colon U\to\D$, where $\half:=\{z\in\C\colon \Re(z)<0\}$ is the left half-plane, such that $\phi(a)=0$ and $\phi\circ f\circ\psi^{-1}(z)=\exp(z)$ for all $z\in H$.
	\end{enumerate}
	\end{Lem}
	\begin{proof}
		This is a classical fact from the covering theory of Riemann surfaces; see \cite[Theorems 5.10 and 5.11]{F}.
	\end{proof}
	
	The forward orbits of the singular values form the \emph{post-singular set}
	\[
		P(f):=\overline{\bigcup_{a\in S(f)}\bigcup_{n\geq 0} f^{\circ n}(a)}.
	\]
	\begin{Def}[Post-singulary finite (psf) entire functions]
	The function $f$ is called \emph{post-singularly finite (psf)} if $\vert P(f)\vert < \infty$.
	\end{Def}
	
In particular, a psf function is of finite type and all singular values have finite orbits (i.e., are periodic or preperiodic). Post-singulary finite entire functions are the focus of this paper.
	Within any parameter space of entire functions, they are important representatives that help to understand the general case, and provide structure to parameter space. Their dynamics is in many ways simpler than that of arbitrary functions. In particular, we have the following well-known result.
	
	\begin{Pro}[Periodic points and Fatou components of psf maps]
		Every post-singularly finite entire function has only finitely many superattracting periodic orbits, and all other periodic orbits are repelling. 		
		Every Fatou component is eventually mapped to a cycle of superattracting Fatou components.\label{Pro:PsfProperties}
	\end{Pro}
	\begin{proof}
	For an entire function of finite type, every Fatou component is eventually mapped to a superattracting, attracting, or parabolic component, or to a Siegel disk. This is well known; see for example  \cite[Theorem 2.1 and Theorem 3.4]{S2}. Every attracting or parabolic cycle of Fatou components contains a singular value with infinite forward orbit, and every boundary
		point of a Siegel disk is the limit point of post-singular points \cite[Theorem 2.3]{S2}. Therefore, a psf entire function can only have superattracting Fatou components. As the
		post-singular set is finite, there can only be finitely many of them.
		
		A periodic point in the Julia set is either repelling or it is a Cremer point. By \cite[Corollary 14.4]{M}, every Cremer point is a limit point of
		post-singular points (the proof of \cite[Corollary 14.4]{M} is for rational functions, but the same proof works in the transcendental case).
		Therefore, a psf entire function cannot have Cremer points and every periodic point in the Julia set is repelling.
	\end{proof}	
	
	\subsection{An extension of the complex plane}
	\label{Sub:PlaneExtension}
	
	For an entire function $f\colon\C\to\C$, we define two extensions $\CTT{f}\supset\CT{f}\supset\C$ of the complex plane. The extension $\CT{f}$ is obtained by adding all transcendental singularities of $f^{-1}$ to $\C$ (``{limits of} asymptotic tracts''), while $\CTT{f}$ is the dynamical version of $\CT{f}$ obtained by adding the transcendental singularities of the inverse function for every iterate of $f$. The extension $\CT{f}$ is classical and exists for every non-constant holomorphic map between Riemann surfaces; see \cite{E2}.
	
	This construction can be carried out for all entire functions (compare \cite[Section~2.4]{DavidThesis}), but it is simpler in our case of functions that have only finitely many singular values. Let $a\in\C$ be an asympotic value and $V$ be an associated asymptotic tract: this is a domain $V\subset \C$ so that $f\colon V\to D_r(a)\setminus\{a\}$ is a universal cover for some $r=r(a)>0$ (a more natural definition for a tract $V$ is if there is a topological disk neighborhood $U$ of $a$ such that $f\colon V\to U\sm\{a\}$ is a universal cover). Two asymptotic tracts over the same asymptotic value are called equivalent if they have a common restriction to another asymptotic tract. We will identify asymptotic tracts with their equivalence classes.
	
  In order to construct $\CT{f}$, we add an additional point at $\infty$ for every asymptotic tract. We turn this into a topological space as follows: consider a particular tract $V$ over the asymptotic value $a$, denote the additional point at $\infty$ corresponding to $V$ by $T$, and define as a neighborhood basis of $T$ the sets $f^{-1}(D_r(a))\cap V$ for all $r\in(0,r(a))$. Then $f$ extends continuously to a map $\extended f\colon V\cup\{T\}\to D_r(a)$ by setting $\extended f(T)=a$.

	We define $\C_f$ as the complex plane, extended by additional points at all asymptotic tracts over all asymptotic values; see Figure~\ref{Fig:ExtendedPlane}. We have a {a continuous extension $\extended f\colon \CT f\to\C$}.

The prototypical case is $f=\exp$ where we have a single asymptotic tract, and we have an additional point $T$ that is often denoted by $-\infty$ and that maps to $0$. This case is general in the following sense.
	
	\begin{Lem}[Logarithmic singularities]
		\label{Lem:TranscSingularities}
		For a map $f$ with asymptotic tract\/ $V$ over the asymptotic value $a\in\C$, write $\extended V:=V\cup\{T\}$, where $T$ denotes the additional point at $\infty$. If $V$ is such that $f\colon V\to D_r(a)\setminus\{a\}$ is a universal cover, then there exist Riemann maps $\phi\colon D_r(a)\to\D$ and $\psi\colon V\to \half$ so that $\psi$ extends to a homeomorphism $\psi\colon \extended V\to \half\cup\{-\infty\}$ and so that the following diagram commutes:
\centerline{\includegraphics[width=0.4\textwidth,trim=0 0 0 5,clip]{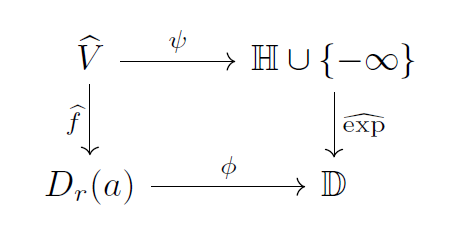}}
	\end{Lem}
	\begin{proof}
		This follows easily from Lemma~\ref{lem:PuncturedCoverings} and the definition of the topology at asymptotic tracts.
	\end{proof}

\begin{figure}[ht]
		\includegraphics[width=0.8\textwidth]{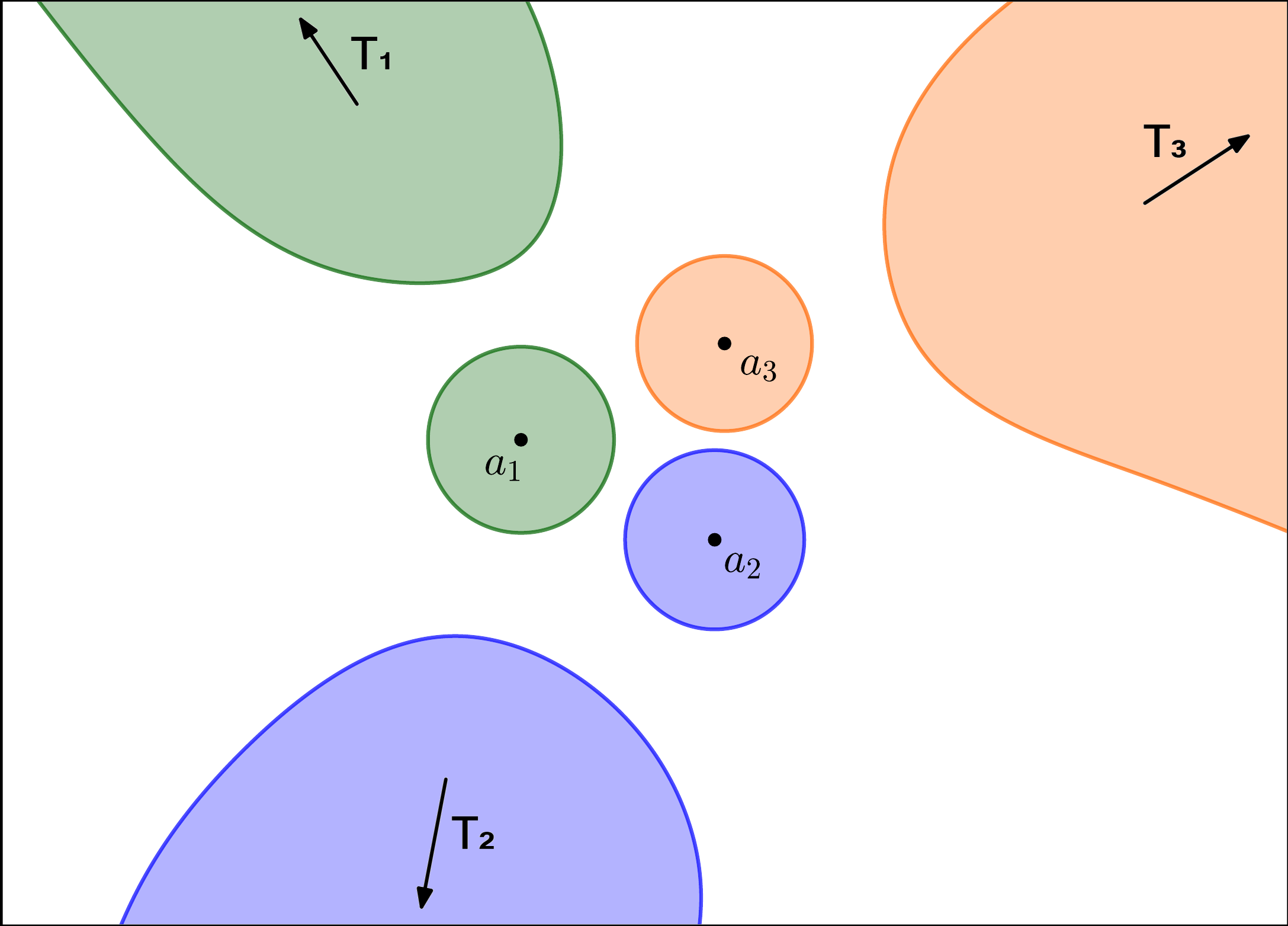}
		\caption{Sketch of a function with three transcendental singularities over three distinct asymptotic values.}
		\label{Fig:ExtendedPlane}
	\end{figure}

Now we proceed to define a dynamical version of $\CT{f}$ that we denote by $\CTT{f}$. In addition to the asymptotic tracts of $f$, we also
	add for every asymptotic value $a\in S(f)$ the asymptotic tracts of every iterate of $f$ to the complex plane (in other words, we add iterated preimages to each extended point in $\CT{f}$). Note that we also need to identify
	tracts for different iterates of $f$ as follows: a sufficiently small tract of $f^{\circ n}$ for an asymptotic value $a$ is also a tract for $f^{\circ (n+1)}$ for the asymptotic value $f(a)$ (strictly speaking, here we need the more natural definition that a tract covers a punctured neighborhood of $a$ that need not necessarily be a round disk). We then set 
	\[
	\CTT{f}:=\C\cup\left(\bigcup_{n\geq 1}\bigcup_{a\in S(f)}\bigcup_{\substack{\text{eq. classes of asympt.}\\\text{tracts of}f^{\circ n}\text{ over }a}}\{T\}\right)/\sim
	\;,
	\]
	where $\sim$ denotes the equivalence relation that identifies tracts for different iterates of $f$. Our function $f$ extends to a continuous map $\extended{f}\colon\CTT{f}\to\CTT{f}$. {In other words, $\CTT f$ extends $\CT f$ by in the minimal way by adding all backwards orbits of the points added in $\CT f$.  }

\begin{Pro}[Basic properties of the extension $\CTT{f}$]
		Let $f$ be an entire function for which the set $S(f)$ of singular values is finite. Then the extended map $\extended{f}\colon\CTT{f}\to\CTT{f}$ is a covering over $\CTT{f}\setminus S(f)$. If $g=f^{\circ n}$ is an iterate of $f$, we have $\CTT{f}=\CTT{g}$.\label{pro:ExtensionCovering}
	\end{Pro}
	\begin{proof}
		The restriction $\restr{f}{\C\setminus f^{-1}(S(f))}$ is a covering map over $\C\setminus S(f)$. It remains to show that every $T\in\CTT{f}\setminus\C$ has a neighborhood $V$ such that $\extended{f}|_{W}$ is a homeomorphism for every connected component $W$ of $\extended{f}^{-1}(V)$.
		
		Given a $T\in \CTT f\setminus\C$, there exists a smallest $n\geq 1$ and an $a\in S(f)$ such that $\extended f^{\circ n}(T)=a$. Choose $r>0$ such that $D_r(a)\cap P(f)=\{a\}$, and let $V$ be the component of $f^{-n}(D_r(a))$ in the tract represented by $T$. Let $V'$ be a connected component of $f^{-1}(V)$, and let $T'\in\CTT{f}$ denote the tract represented by $V'$. Then $h:=\restr{f}{V'}\colon V'\to V$ is biholomorphic because $V$ is simply connected and satisfies $V\cap P(f)=\emptyset$. The restriction extends to a homeomorphism
		\[
		\extended{h}\colon V'\cup\{T'\}\to V\cup\{T\}.
		\]
		This shows that $\extended{f}$ is a covering over $\CTT{f}\setminus S(f)$.
		
		To prove the second statement, just note that if $T$ is a tract of $f^{\circ n}$ over $a$, then $T$ is also a tract of $f^{\circ (n+m)}$ over $f^{\circ m}(a)$.
	\end{proof}
	
  \begin{remarknumber}\label{rem:VanKampen}
    Let us make some further remarks on the topology on $\CTT{f}$. If $p \in \CTT{f} \setminus \C$ and $V$ is a tract at $p$, then both
    $V$ and $V \setminus \{p\}$ are contractible. This will be useful to compare the topology of open sets $U \subset \CTT{f}$ with $U \cap \C$:

    For every $p \in U \cap (\CTT{f} \setminus \C)$, choose a tract $V_p \subset U$. We can arrange the $V_p$ to be disjoint. If we then apply the groupoid version of the Van Kampen theorem (see for example \cite[Theorem 2.5]{Kammeyer}) to the open covering $(U \cap \C) \cup \bigcup_p (V_p)$, we obtain that $U \cap \C$ and $U$ have equivalent fundamental groupoids. In particular, $U \cap \C$ and $U$ have the same number of connected components, and they have isomorphic fundamental groups.

Another useful property of the topology of $\CTT{f}$ is the following: suppose $\gamma : S^1 \rightarrow \CTT{f}$ is continuous and injective (so a ``Jordan curve" in the extended plane), such that $\gamma$ visits only $\CTT{f} \setminus \C$ only finitely many times. Then there is exactly one connected components $V$ of $\CTT{f} \setminus \gamma$ such that $V \cup \gamma$ is compact (and in fact, homeomorphic to a closed Jordan disk).
  \end{remarknumber}

	There is a natural way to define the Julia set $\mathcal{J}(\extended{f})$ and the Fatou set $\mathcal{F}(\extended{f})$ of the extended map $\extended{f}$: a point in $\CTT f\sm \C$ belongs to the Fatou resp.\ the Julia set if and only if the asymptotic value where the orbit first enters $\C$ belongs to the Fatou resp.\ Julia set of $f$. {Clearly, the Fatou set of $\extended f$ is still open and the Julia set is closed. }

Finally, we define the set of \emph{critical points} of $\extended{f}$ as
	\[
	C(\extended{f}):=C(f)\cup(\CT{f}\setminus\C)
	\;.
	\]
	This is precisely the set of points that do not have neighborhoods on which $\extended{f}$ is injective. In this sense, all singular values of $f$ are critical values of $\extended f$.
	
Let us introduce the following notation that we will often use.
	\begin{Def}[(Pre\nobreakdash-)periodic points]
		We denote by $\Per(\extended{f})\subset\CTT{f}$ the set of points $p\in\CTT{f}$ that are (pre\nobreakdash-)periodic under iteration of $\extended{f}$.\label{def:ExtendedPreperiodic}
	\end{Def}
	
	\section{Dynamics of post-singularly finite entire functions}\label{section:PsfDynamics}
	
	\subsection{Fatou components and internal rays}\label{subsection:FatouDynamics}
	
	Every periodic Fatou component, say $U$, is superattracting by Proposition~\ref{Pro:PsfProperties}. Therefore, $U$ contains a unique superattracting periodic point, called its \emph{center}. The component comes with a Riemann map $\Phi\colon U\to\D$ that sends the center to $0$, and so that it conjugates the first return map on $U$ to $z\mapsto z^d$ on $\D$, for a unique $d\ge 2$ \cite[Theorems 9.1 and 9.3]{M}; we call such a Riemann map a \emph{B\"ottcher map} in analogy to the situation for polynomials. There are exactly $d-1$ choices for $\Phi$ \cite[Theorem 9.1]{M}. An \emph{internal ray} of $U$ is the preimage of a radial line in $\D$ under $\Phi$. For our purposes, an internal ray \emph{lands} if its closure \emph{in the extended plane $\CTT{f}$} intersects $\partial_{\CTT{f}} U$ in a single point. This single point is then called the \emph{landing point} of the ray.
	
	It is well known that every periodic internal ray of a periodic Fatou component lands at a unique point $q\in\C$, and that $q$ is a repelling periodic point so that its period divides the period of the ray. (All this follows exactly as for polynomials, except one needs to show that the internal ray is bounded as a subset of $\C$; see \cite[Theorems B.1 and B.2]{R1}). Preperiodic internal rays of periodic Fatou components need not land in $\C$, but in any case they land at some point in $\CTT{f}$. For example, the function $z\mapsto(z-1)\exp(z)+1$ has a fixed Fatou component with superattracting fixed point at $z=0$, there is a fixed ray that lands at $z=1$, and this ray has a preperiodic preimage ray that lands at a point at $\infty$ (along the negative real axis). 
	
  Let us now extend centers and internal rays to a preperiodic Fatou component $V$. Let $n$ be minimal so that $U:=f^{\circ n}(V)$ is a periodic Fatou component. Then $f^{\circ n}\colon V\to U$ has finite or infinite mapping degree. In the case of finite mapping degree, it follows as for periodic Fatou components that $f^{\circ n}\colon V\to U$ can have only one critical point, and this point must map to the center of $U$. In this case, center and internal rays of $V$ can be defined naturally via pull-back, and all preperiodic internal rays must land in $\CTT{f}$. 

If $f^{\circ n}\colon V\to U$ has infinite mapping degree, similar reasons imply that as a subset of $\C$ it must factor through the exponential map as follows: there are Riemann maps $\psi\colon V\to \half\cup\{-\infty\}$ and $\phi\colon U\to\D$ so that $\phi$ maps the center of $U$ to $0$ and
	$f^{\circ n}=\phi^{-1}\circ\extended\exp\circ\psi\colon V\to U$, but here the additional point $-\infty\in\partial \half$ is part of the extended complex plane. In this case, the center of $V$ is at $\infty$.
	Internal rays in $U$ pull back to radial lines in $\D$ by $\phi$, then to horizontal lines in $\half$ by $\exp$, and finally give rise to uniquely defined internal rays in $V$. Every preperiodic internal ray in $V$ lands either at a repelling preperiodic point of $f$ in $\C$, or at a boundary point of $V$ in $\CTT f\setminus\C$; both are preperiodic points in $\mathcal{J}(\extended{f})$.
	
	In conclusion, every component $U$ of $\mathcal{F}(\extended{f})$ has a well-defined and unique center $p$. We denote the Fatou component with center $p$ by $U(p)$.

\begin{Lem}[Distinct landing points]
		For any Fatou component of $\extended f$, any two periodic or preperiodic internal rays land at distinct boundary points.
		\label{lem:DistinctLandingPoints}
	\end{Lem}
	\begin{proof}
Suppose the internal rays $\beta_1$ and $\beta_2$ of the Fatou component $V$ land at the same point $q$.
Let $m$ be large enough such that $\extended{f^m}(\beta_1)$ and $\extended{f^m}(\beta_2)$ are periodic rays. In particular, we get that $\extended{f^{m+k}}(\beta_1 \cup \beta_2)$ remains bounded for $k \geq 0$. Let $U$ be the connected component of $\CTT{f} \setminus (\beta_1 \cup \beta_2)$ such that $U \cup \beta_1 \cup \beta_2$ is compact (see Remark \ref{rem:VanKampen}). By the maximum modulus principle, we get $f^{m+k}$ remains bounded on $U$. So $U \subset \mathcal{F}(\extended{f})$, hence $\beta_1$ and $\beta_2$ are equivalent accesses to $q$. This is a contradiction.

\end{proof}
	
	\subsection{The escaping set}\label{sub:EscapingSet}
	
	One of the most important sets for studying the dynamics of entire functions is the set
	\[
	I(f):=\{z\in\C~\colon\lim_{n\to\infty}f^{\circ n}(z)=\infty\}
	\]
	of escaping points. The escaping set $I(f)$ is fully invariant and $I(f)\neq\emptyset$ \cite{E1}. In many cases, and in particular for psf entire maps, the escaping set carries a rich combinatorial structure. In the following, we will describe some of its aspects.
	
	In \cite{RRRS} it was shown that for a large class of transcendental entire maps the escaping set decomposes in a natural way into dynamic rays distinguished by external addresses so that the escaping dynamics of the map is semi-conjugate to the simple dynamics of the shift map on the space of external addresses. However, this class of maps does not include all post-singularly finite entire functions because components of the escaping set might have a more complicated topology than arcs, depending on the geometry of the tracts of the function \cite{RRRS}. Even in this more complicated case, the escaping set still decomposes in a natural way into unbounded connected components called \emph{filaments} \cite{BR}. These filaments share many similarities with rays: their connected components are again distinguished by external addresses so that the escaping dynamics is semi-conjugate to the dynamics of the shift map, and thus they can be viewed as natural generalizations of dynamic rays. In the following, we describe the natural decomposition of the escaping set into filaments and their combinatorial description via external addresses.
	
	The construction we describe below can be carried out more generally for functions with bounded post-singular set \cite{BR}, but in this paper we focus on post-singularly finite functions $f$. Let $D$ be an open disk containing $0$ and $f(0)$, large enough so as to contain $P(f)$. The connected components of $f^{-1}(\C\setminus\overline{D})$ are called \emph{tracts of $f$ over $\infty$}. It is easy to see that every tract $T$ of $f$ is simply connected, unbounded, that $\partial T$ is a Jordan arc tending to infinity in both directions, and that the restriction $\restr{f}{T}$ is a universal covering over $\C\setminus\overline{D}$. Let $\alpha:[0,1)\to\C$ be an arc connecting $\partial D$ to $\infty$ that satisfies $\alpha\cap D=\emptyset$ and $\alpha\cap f^{-1}(\C\setminus\overline{D})=\emptyset$, so $f(\alpha)\subset\overline D$; see Figure~\ref{fig:Tracts}.
	\begin{figure}[ht]
		\includegraphics[width=0.8\textwidth]{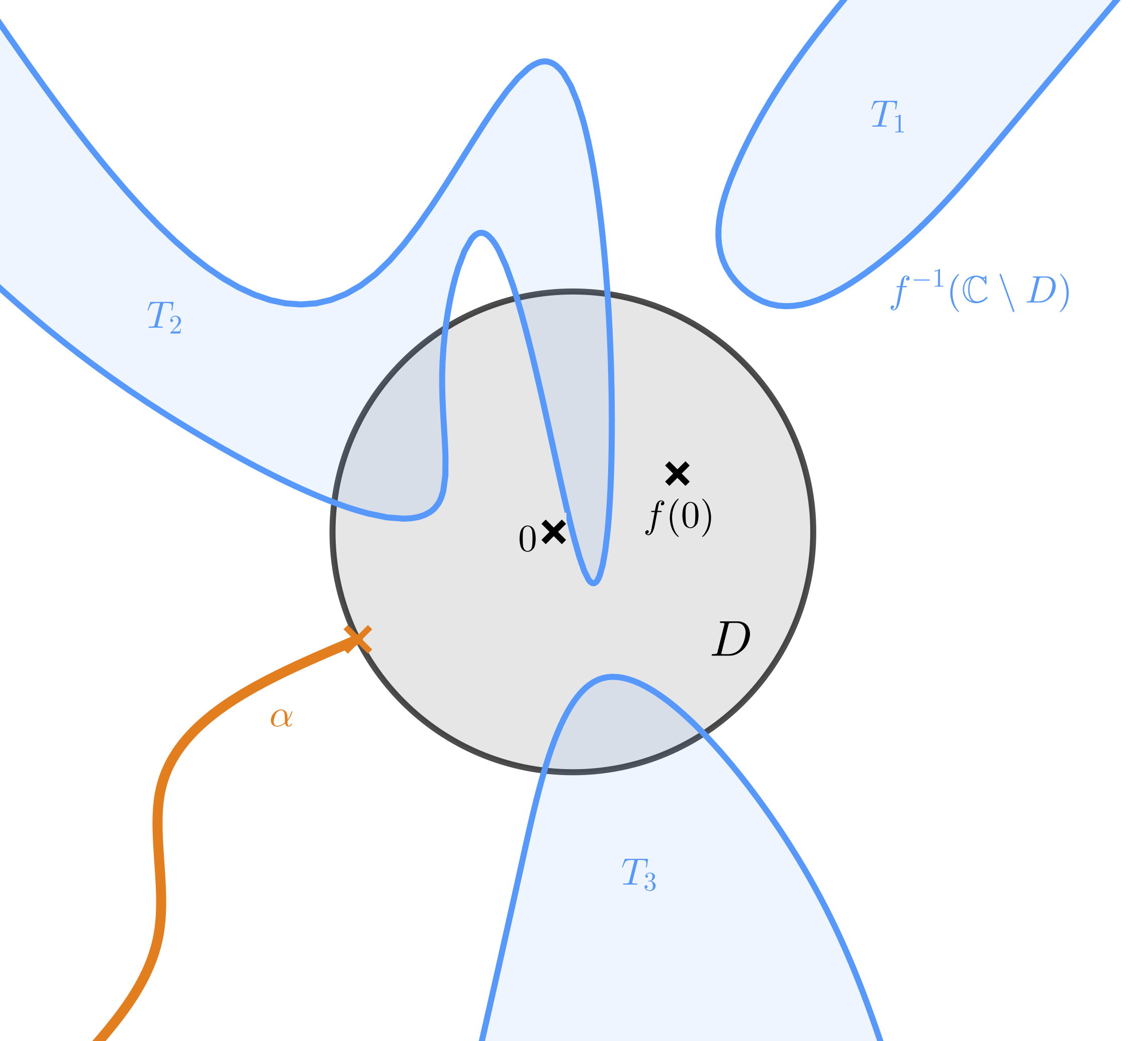}
		\caption{A sketch of a possible configuration of $D$ and $\alpha$.}
		\label{fig:Tracts}\end{figure}
	We set $W_0:=\C\setminus(\overline{D}\cup\alpha)$. The domain $W_0$ is simply connected with $P(f)\cap W_0=\emptyset$. Therefore, for every $n\ge 0$, every component of $f^{-n}(W_0)$ maps biholomorphically onto $W_0$; see Figure~\ref{fig:FundamentalDomain}.
	\begin{figure}[ht]
		\includegraphics[width=\textwidth]{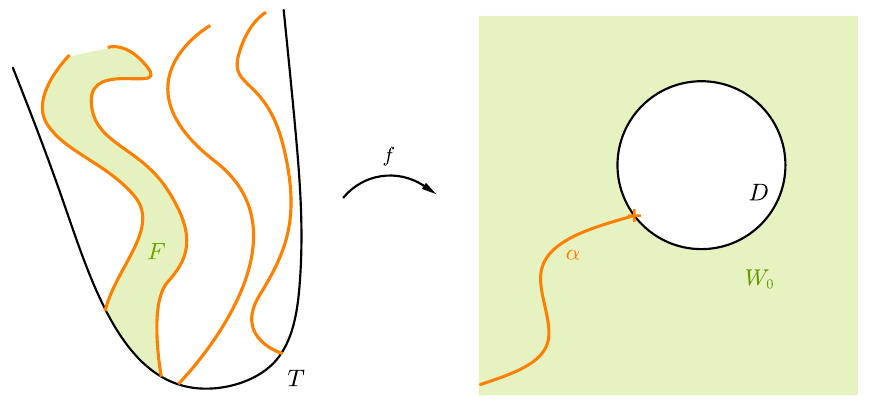}
		\caption{Sketch of $W_0$ and a fundamental domain $F$.}
		\label{fig:FundamentalDomain}\end{figure}
	
	\begin{Def}[Fundamental tails and external addresses]
		A connected component $\tau$ of $f^{-n}(W_0)$ is called a \emph{fundamental tail of level $n$}. A fundamental tail of level $1$ is called a \emph{fundamental domain} and is commonly denoted as $F$.
		We denote the set of all fundamental domains by $\mathscr{S}(D,\alpha)$ and call it a \emph{static partition} for $f$.
		
		An \emph{external address} $\ad{s}:=(F_i)_{i=0}^\infty=F_0F_1\ldots$ is a sequence of fundamental domains $F_i$. We denote the \emph{space of external addresses} by $\AD$ (that is, sequences over the set of tracts) and define the \emph{shift map} \[\sigma\colon\AD\to\AD,~\sigma(F_0F_1F_2\ldots  )=F_1F_2\ldots~~~\text{for all}~\ad{s}=(F_i)_{i=0}^\infty\in\AD.\] We call the external address $\ad{s}$ \emph{bounded} if it contains only finitely many distinct $F_i$. The address $\ad{s}$ is called \emph{periodic} if it is periodic under iteration of $\sigma$, and \emph{preperiodic} if it is preperiodic under iteration of $\sigma$.\label{def:fundamentalDomains}
	\end{Def}
	
	Let $\tau$ be a fundamental tail of level $n > 1$. Then $\tau$ is a Jordan domain on $\CC$ containing $\infty$ on its boundary, and we have $f^{\circ k}(z)\to\infty$ as $z\to\infty$ in $\tau$ for all $k\leq n$.
	It follows that $\tau$ tends to infinity through some fundamental domain. More precisely, there exists a unique fundamental domain $F$ such that $\overline{\tau}\setminus F$ is bounded, see \cite[Lemma 3.6]{BR}.
	Therefore, we can naturally associate a finite external address to each fundamental tail, see \cite[Definition 3.7]{BR}.
	
	\begin{Def}[Addresses of fundamental tails]
		Let $\tau$ be a fundamental tail of level $n$, and denote for $k<n$ by $F_k(\tau)$ the unique fundamental domain whose intersection with the fundamental tail $f^{\circ k}(\tau)$ is unbounded. We call the finite sequence $\ad{s}=F_0(\tau)F_1(\tau)\ldots   F_{n-1}(\tau)$ the \emph{(finite) external address} of $\tau$.
	\end{Def}
	
  Every finite external address is realized by one and only one fundamental tail, see \cite[Definition and Lemma 3.8]{BR}. We extend the shift map $\sigma$ to finite external addresses as follows: we let $\sigma(F_0F_1F_2\ldots F_n)=F_1F_2\ldots F_n.$
	
	\begin{LemDef}[Tails at a given finite address]\label{lemdef:TailsAtAddress}
		Let $n\ge 1$ be an integer and $\ad{s}=F_0F_1\ldots  $ be a finite or infinite sequence of fundamental domains that has length at least $n$. Then there exists a unique fundamental tail $\tau$ of level $n$ that has address $\tilde{\ad{s}}:=F_0F_1\ldots   F_{n-1}$. We denote this fundamental tail by $\tau_n(\ad{s})$; see Figure~\ref{fig:FundamentalTails} for a visualization. We also define the inverse branches
		\[
		f_{\ad{s}}^{-n}:=(\restr{f^{\circ n}}{\tau_n(\ad{s})})^{-1}\colon W_0\to\tau_n(\ad{s}).
		\]
	\end{LemDef}
	\begin{figure}[ht]
		\includegraphics[width=\textwidth]{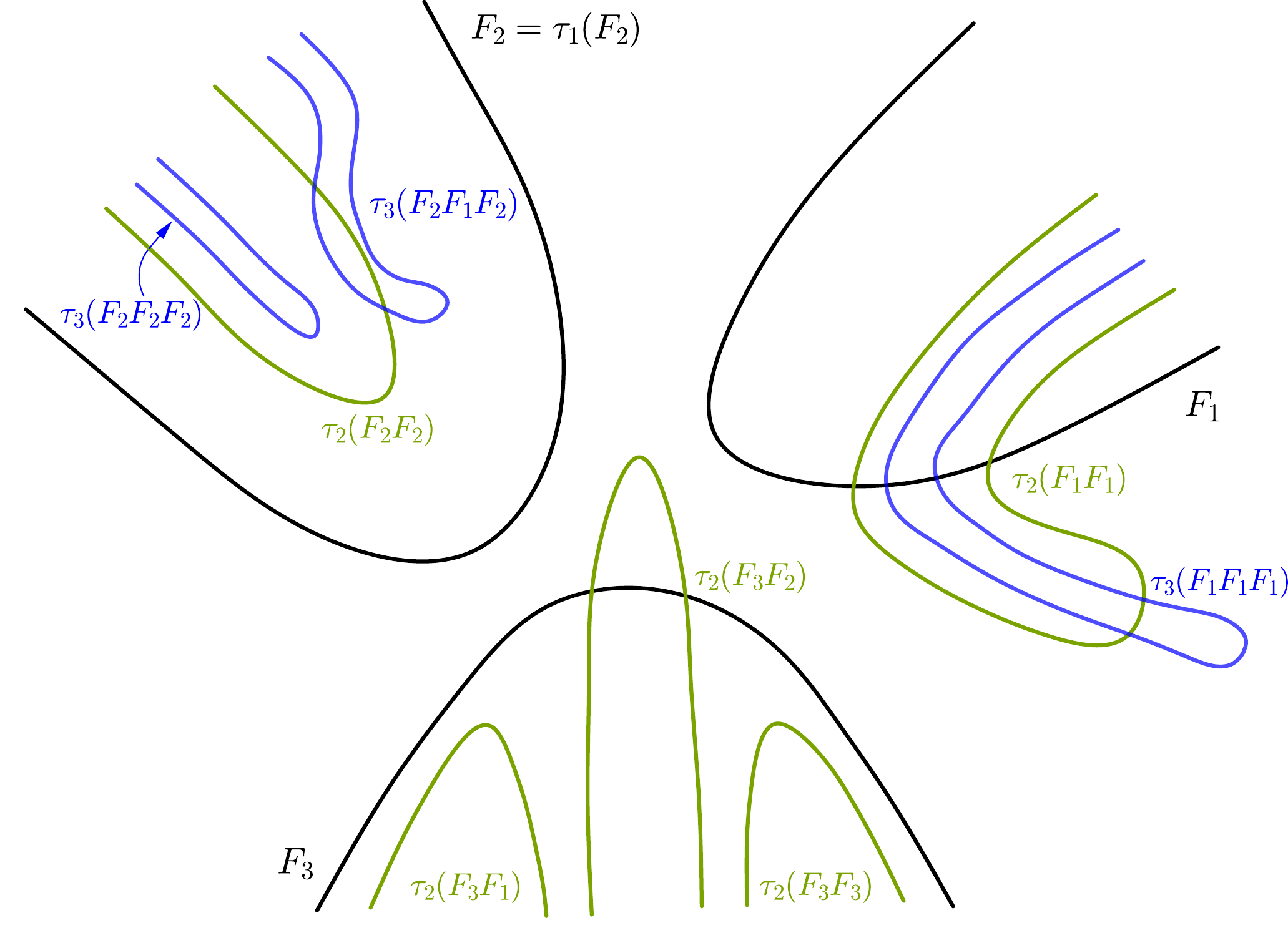}
		\caption{Sketch of some fundamental tails of various levels. For a given address $\ad{s}$, the fundamental tails $\tau_n(\ad{s})$ might be nested or not. Fundamental tails of different levels might intersect, even if one is not the prefix of the other.}
		\label{fig:FundamentalTails}\end{figure}
	Let us collect some more results about fundamental tails from \cite[Observation 3.9]{BR} that follow more or less directly from the preceding results and definitions.
	
	\begin{Lem}[Further facts about fundamental tails]
		Let $\tau$ be a fundamental tail of level $n$ and let $\ad{s}$ be the address of $\tau$. Then the address of $f(\tau)$ is $\sigma(\ad{s})$.
		
		Suppose that $\tau^1$ and $\tau^2$ are fundamental tails of levels $n_1$ and $n_2$ with $n_1\geq n_2$. Let $\ad{s}^1$ and $\ad{s}^2$ be the addresses of $\tau^1$ and $\tau^2$ respectively. Then $\tau^1\cap\tau^2$ is unbounded if and only if $\ad{s}^1$ is a prefix of $\ad{s}^2$, i.e., $\ad{s}^2 = \ad{s}^1 F_{n_1} \ldots F_{n_2-1}$. In this case, if additionally $n_1>n_2$, all sufficiently large points of {the closure} $\overline{\tau^1}$ are contained in $\tau^2$.\label{lem:FurtherFundamentalFacts}
	\end{Lem}

In order to define filaments, we introduce the unbounded component of fundamental tails.
	
	\begin{Def}[Unbounded component of fundamental tails] \label{def:unboundedComponents}
		Let $F$ be a fundamental domain. We denote by $\unbounded{F}$ the unique unbounded connected component of $F\setminus\overline{D}$; see Figure~\ref{fig:F_infinity}.
		
		More generally, let $\tau$ be a fundamental tail of level $n$. We define $\unbounded{\tau}$ to be the unbounded connected component of $\tau\setminus f^{-(n-1)}(\overline{D})$. In other words, if $F=f^{\circ(n-1)}(\tau)$, then $\unbounded{\tau}$ is the component of $f^{-(n-1)}(\unbounded{F})$ contained in $\tau$.
	\end{Def}
	
	\begin{figure}[ht]
		\includegraphics[width=0.6\textwidth]{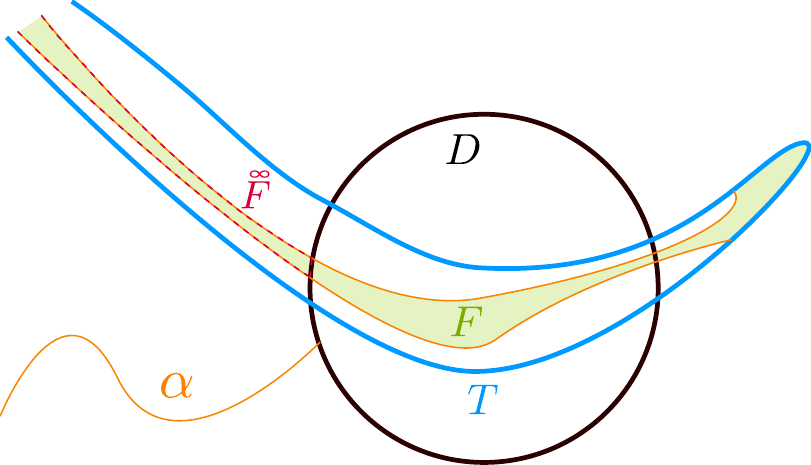}
		\caption{Sketch of the unbounded component of a fundamental domain $F$.}
		\label{fig:F_infinity}\end{figure}
	
	Note that $\unbounded{\tau}_n(\ad{s})$ is the unbounded connected component of $\tau_n(\ad{s})\cap\tau_{n-1}(\ad{s})$ for $n\geq 2$. We have finally gathered the necessary terminology to formally define filaments.
	
	\begin{Def}[Filaments] \label{def:Filaments}
		Let $\ad{s}$ be an external address. We say that a point $z\in\C$ \emph{has external address $\ad{s}$} if $z\in\unbounded{\tau}_n(\ad{s})$ for all sufficiently large $n$.
		
		The \emph{filament} $G_{\ad{s}}$ is defined to be the set of escaping points $z\in \C$ that have external address $\ad{s}$.
	\end{Def}
	
	It follows directly from Definition \ref{def:unboundedComponents} and Definition \ref{def:Filaments} that every escaping point has an external address, and thus every escaping point is contained in one and only one filament (see also \cite[Corollary~4.5]{BR}). 
	
	While filaments provide a decomposition of the escaping set, they depend a priori on the choice of the base domain $W_0$. However, it turns out that the decomposition of the escaping set into filaments is independent of the choice of the base domain, and external addresses w.r.t.\ any two given base domains are in natural bijection to each other; see \cite[Observation~4.12 and the discussion before]{BR} for an explanation. 
	
	Another important fact is that a function has the same filaments as any of its iterates; moreover, external addresses for the function are in natural bijection to external addresses for its iterate, see \cite[Observation~4.13 and the discussion before]{BR}. 
	
	\begin{Lem}[Filaments of iterates]
		Let $f$ be a psf entire function, and let $n\geq 1$. Then every filament of $f$ is a filament of $f^{\circ n}$ and vice versa.\label{lem:FilamentsOfIterates}
	\end{Lem}
	
	We summarize the results obtained on the escaping dynamics of psf entire functions in one theorem below, see \cite[Lemma 4.3, Corollary 4.5, and Proposition~4.10]{BR}.
	
	\begin{Thm}[The escaping set of a post-singularly finite entire function]
		Let $f$ be a post-singularly finite entire function. The escaping set $I(f)=\bigcupdot_{\ad{s}\in\AD} G_{\ad{s}}$ decomposes in a natural way into subsets indexed by external addresses called filaments. For every $\ad{s}\in\AD$, either $G_{\ad{s}}=\emptyset$ or $G_{\ad{s}}$ is unbounded and connected. We have $f(G_{\ad{s}})=G_{\sigma(\ad{s})}$.\label{thm:EscapingSet}
	\end{Thm}
	
	Not every external address is actually realized by points in $\C$: the growth of $f$ imposes a limit on the growth of realized external addresses \cite{SZ1}. However, every bounded address is realized, and especially every (pre-)periodic address. The latter are the ones we are interested in here.
	
	\subsubsection{Cyclic order of filaments.}\label{subsubsec:CyclicOrders} We denote a cyclic order on a set $X$ by $a\prec b\prec c$, where $a,b,c\in X$. For convenience, we write
	\[
	a_1\prec a_2\prec\ldots\prec a_n,
	\]
	where we use $\prec$ to stress that this is a cyclic order of $n\geq 3$ elements $a_j\in X$, not a linear order. This expression means that $a_{j-1}\prec a_j\prec a_{j+1}$ for all $j\in\{1,\ldots,n\}$, where indices are labeled mod $n$. Expressions such as
	\[
	a\preceq b\prec c
	\] 
	mean that either $a\prec b\prec c$ or $a=b\prec c$.
	
	There is a natural cyclic order on the set of fundamental domains. Given three fundamental domains $F_0$, $F_1$, and $F_2$, we choose, for $i \in \{0,1,2\}$, arcs $\gamma_i\colon[0,1]\to\CC$ that connect some point $\zeta_i=\gamma_i(0)\in F_i$ to $\infty=\gamma_i(1)$ and satisfy $\gamma_i((0,1))\subset F_i$. Let $R>0$ be large enough such that each of the arcs $\gamma_i$ contains a point of modulus $R$, and set $t_i:=\max_{t\in[0,1]}\{t\colon\vert\gamma_i(t)\vert=R\}$. The points $\gamma_i(t_i)$ have a counter-clockwise cyclic order on the circle $\partial D_R(0)$, and this is by definition the cyclic order of the arcs $\gamma_i$ at infinity. It is not hard to see that this cyclic order is well defined, i.e. independent of the choice of $R$. Indeed, if the cyclic order were different for $\tilde{R}\neq R$, the arcs $\gamma_i$ would not be pairwise disjoint.
	\begin{figure}[ht]
		\includegraphics[width=0.75\textwidth]{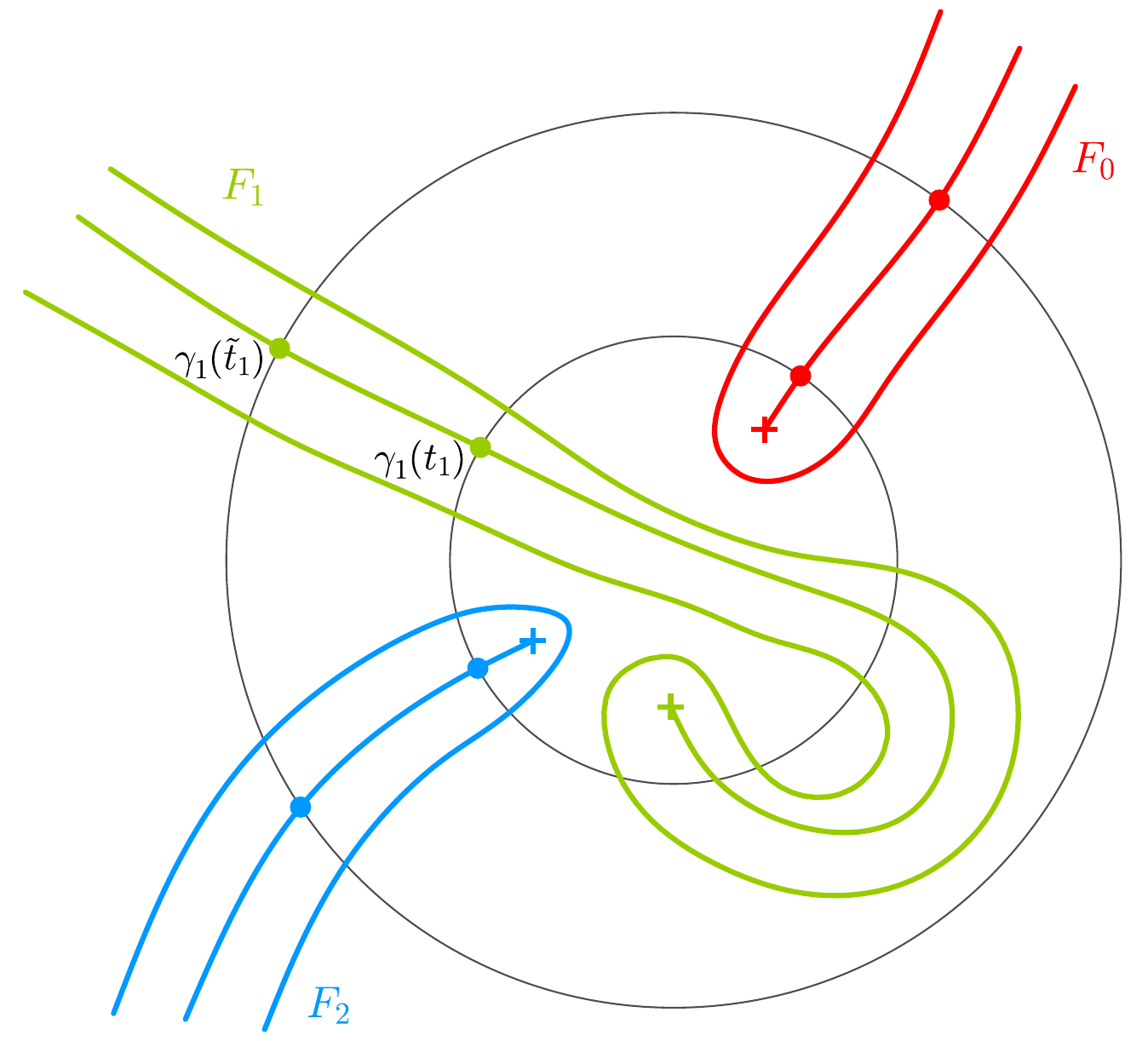}
		\caption{Sketch of the arcs $\gamma_i$ used to define the cyclic order of the fundamental domains $F_i$. The arc $\gamma_1$ shows that it is important to use the maximal potential $t_i$ at which $\gamma_i$ intersects $D_R(0)$.}
		\label{fig:CyclicOrderDomains}\end{figure}
	
	The cyclic order of the fundamental domains $F_i$ is, by definition, the same as the cyclic order of the corresponding curves $\gamma_i$. This order is well-defined, i.e., independent of the choice of the $\gamma_i$, because every fundamental domain has a unique access to $\infty$.
	
	\begin{LemDef}[Successors of fundamental domains]\label{def:Successor}
		Let $\mathscr{S}(D,\alpha)$ be a static partition, and let $F\in\mathscr{S}(D,\alpha)$ be a fundamental domain. Then there are unique fundamental domains $\Fpred{F},\Fsucc{F}\in\mathscr{S}(D,\alpha)$, called the \emph{predecessor} and \emph{successor} of $F$, such that every $F'\in\mathscr{S}(D,\alpha)\setminus\{\Fpred{F},F,\Fsucc{F}\}$ satisfies
		\[
		F'\prec \Fpred{F} \prec F\prec \Fsucc{F}  \;.
		\]
	\end{LemDef}
	\begin{proof}
		This follows directly from the topology of fundamental domains: each of them is contained in a tract, and the fundamental domains within each tract are ordered like $\Z$.
	\end{proof}
	
	For a given $N\geq 1$, the fundamental tails of level $N$ are also pairwise disjoint Jordan domains on $\CC,$ each with a unique access to $\infty$. Therefore, the same construction as above allows us to define a cyclic order on the set of fundamental tails of level $N$ for any given $N\geq 1$. These orders can be used to define a cyclic order on the space of external addresses.
	
	\begin{Def}[Cyclic order of external addresses]\label{def:CyclicOrder}
		Let $\ad{s}$, $\ad{t}$, and $\ad{u}$ be distinct external addresses. Choose $N\geq 1$ large enough so that the fundamental tails $\tau_N(\ad{s})$, $\tau_N(\ad{t})$, and $\tau_N(\ad{u})$ of level $N$ are distinct. Define the cyclic order of the addresses~via
		\[
		\ad{s}\prec\ad{t}\prec\ad{u} :\Leftrightarrow \tau_N(\ad{s})\prec\tau_N(\ad{t})\prec\tau_N(\ad{u}).
		\]
	\end{Def}
	
	By Lemma~\ref{lem:FurtherFundamentalFacts}, all points of $\tau_M(\ad{s})$ with sufficiently large absolute values are contained in $\tau_N(\ad{s})$ for each $M\geq N$. Therefore, the cyclic order on $\AD$ is well defined. It induces a topology on the space of external addresses (this works on the level of all external addresses, whether or not they are realized).
	
	\begin{Def}[Order topology on $\AD$]
		For distinct external addresses $\ad{s},\ad{t}$ we define intervals between $\ad{s},\ad{t}$ as
		\begin{gather*}
		(\ad{s},\ad{t}):=\{\ad{u}\in\AD\colon\ad{s}\prec\ad{u}\prec\ad{t}\},\\
		(\ad{s},\ad{t}]:=\{\ad{u}\in\AD\colon\ad{s}\prec\ad{u}\preceq\ad{t}\},\\
		[\ad{s},\ad{t}):=\{\ad{u}\in\AD\colon\ad{s}\preceq\ad{u}\prec\ad{t}\},\\
		[\ad{s},\ad{t}]:=\{\ad{u}\in\AD\colon\ad{s}\preceq\ad{u}\preceq\ad{t}\}.\\
		\end{gather*}
		
		The open intervals form $(\ad{s},\ad{t})$ the basis of the \emph{order topology} on $\AD$.\label{def:OrderTopology}
	\end{Def}
	
	When we do topological constructions in $\AD$ (as e.g.\ in Proposition~\ref{pro:RealizedItineraries}), the underlying topology is always the order topology.
	As filaments are  distinguished by external addresses, the cyclic order on the space of external addresses can be used to define a cyclic order on the set of filaments.
	
	\begin{Def}[Cyclic order of filaments]
		We define a cyclic order on the set of filaments via
		\[
		G_{\ad{s}}\prec G_{\ad{t}}\prec G_{\ad{u}}:\Leftrightarrow\ad{s}\prec\ad{t}\prec\ad{u}
		\]
		for distinct filaments $G_{\ad{s}}$, $G_{\ad{t}}$, and $G_{\ad{u}}$.\label{def:CyclicOrderFilaments}
	\end{Def}
	
	In \cite[Sections 4 and 13]{BR}, an equivalent but more direct definition of the cyclic order of filaments at infinity that does not need external addresses is given {using the cyclic order of tails at sufficiently large finite level that distinguish all three addresses $\ad s, \ad t, \ad u$}.
	
	\subsubsection{Intermediate addresses and linear order.}\label{subsubsec:CircleOfAddresses} Let $f\in\EL$ be of bounded type, and let $\mathscr{S}(D,\alpha)$ be a static partition for $f$. In \cite[Section 5]{BJR}, a dynamical compactification of the complex plane (depending on $f$ and $\mathscr{S}(D,\alpha)$) was obtained by adding a \emph{circle of addresses} at infinity. The authors defined an extension $\overline{\AD}\supset\AD$, where $\overline{\AD}$ is the completion of $\AD$ w.r.t.\ the cyclic order. The elements of $\overline{\AD}\setminus\AD$ are called \emph{intermediate addresses}, and the extension $\overline{\AD}$ is called the circle of addresses. The dynamical compactification is then defined to be $\C_{\overline{\AD}}:=\C\cupdot\overline{\AD}$ equipped with a suitable topology. In this topology, $\overline{\AD}$ is homeomorphic to the unit circle and $\C_{\overline{\AD}}$ is homeomorphic to the closed unit disk.
	The details of this constructions are described in \cite[Section 5]{BJR}; here, we will just describe a few special cases that we need, associated to the curve $\alpha$ and its immediate preimages.
	
	Given distinct external addresses $\ad{s}\neq \ad{t}\in\AD$, we choose $N\geq 1$ large enough such that the fundamental tails $\tau_N(\ad{s})$ and $\tau_N(\ad{t})$ of level $N$ are disjoint. Choose arcs $\gamma_{\ad{s}}\colon[0,1]\to\tau_N(\ad{s})\cup\{\infty\}$ and $\gamma_{\ad{t}}\colon[0,1]\to\tau_N(\ad{t})\cup\{\infty\}$ that connect arbitrary base points to $\infty$. Since $\alpha$ does not intersect the tracts of $f$, some unbounded part $\alpha'$ of $\alpha$ satisfies $\alpha'\cap\tau_N^i=\emptyset$ ($i\in\{1,2\}$). Hence, the arcs $\gamma_{\ad{s}}$, $\gamma_{\ad{t}}$, and $\alpha'$ are pairwise disjoint and have a well-defined cyclic order at infinity. We define
	\[
	\alpha\prec\ad{s}\prec\ad{t} :\Leftrightarrow \alpha'\prec\gamma_{\ad{s}}\prec\gamma_{\ad{t}}.
	\]
	It follows as before that this cyclic order is well defined. The first intermediate address that we define is $\alpha\in\overline{\AD}\setminus\AD$  (the mildly ambiguous usage of $\alpha$ should not cause any confusion).
	
	Now consider any immediate preimage $\tilde \alpha$ of $\alpha$. There is a unique fundamental domain $F\in\mathscr(D,\alpha)$ such that $\tilde{\alpha}\subset \partial \Fpred{F}\cap\partial F$ (see Definition~\ref{def:Successor}: $\tilde\alpha$ is between $F$ and its predecessor). We denote this lift by $\alpha_{\Fpred{F}}^F\in\overline{\AD}\setminus\AD$. Analogously to the case of $\alpha$, we can extend the cyclic order to all $\alpha_{\Fpred{F}}^F$.
	
	There are many more intermediate addresses, but we will not need them here.
	
	Finally, note that every cyclic order can be turned into a linear order by removing an arbitrary element. We do this with respect to the distinguished intermediate address $\alpha\in\overline{\AD}\setminus\AD$ (the simplest intermediate address).
	
	\begin{Def}[Linear and cyclic order of external addresses]
		Let $\ad{s}=F_0F_1\ldots$ and $\ad{t}=F_0'F_1'$ be distinct external addresses. We define a linear order on $\AD$ via
		\[
		\ad{s}<\ad{t}:\Leftrightarrow\alpha\prec\ad{s}\prec\ad{t}.
		\]\label{def:LinearOrder}
	\end{Def}

\Newpage

\section{Landing of filaments}
	\label{section:LandingOfFilaments}
	
	Our main focus in this paper are (pre\nobreakdash-)periodic filaments and their landing behavior. If the filament $G_{\ad{s}}$ happens to be a dynamic ray, i.e., if it can be parametrized via a homeomorphism $\gamma\colon(0,\infty)\to G_{\ad{s}}$ satisfying $\lim_{t\to+\infty}\gamma(t)=\infty$, then it makes sense to say that $G_{\ad{s}}$ lands if the limit $ \lim_{t\to 0} \gamma(t)$ exists.

	In general, the topology of $G_{\ad{s}}$ is more complicated, so we need a more abstract way to define whether a filament lands. An additional problem comes from the fact that there might be preperiodic filaments that do not accumulate anywhere in $\C$ but that do land, in a meaningful way, in the extended plane $\CTT{f}$. Indeed, this happens if and only if $f$ has asymptotic values, so it happens even for preperiodic dynamic rays of psf exponential maps; see Figure~\ref{fig:nonLandingRay} for an example.
	
	\begin{figure}[ht]
		\includegraphics[width=\textwidth,trim=0 30 0 30]{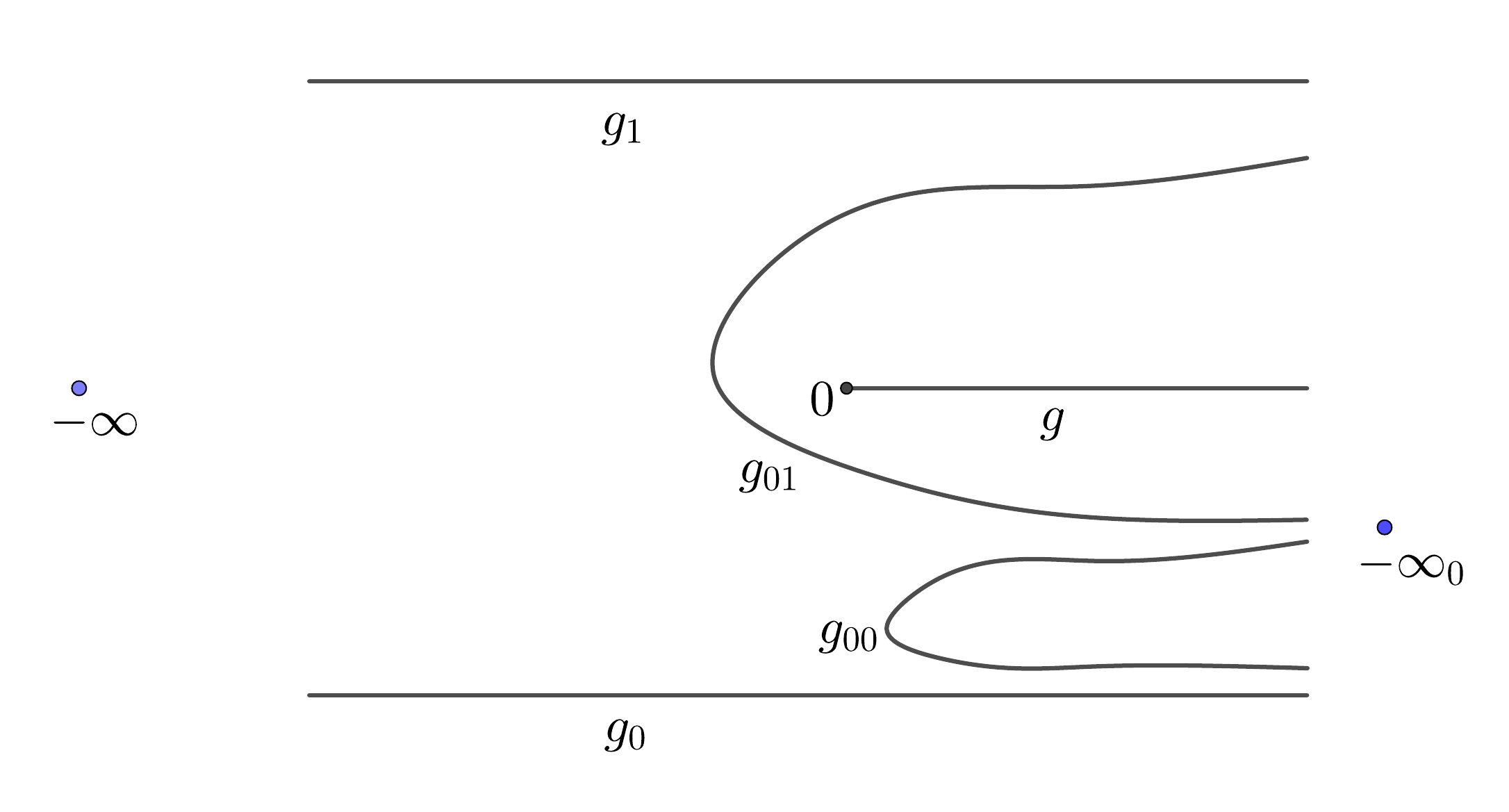}
		\caption{For every psf exponential map $\lambda\mapsto\lambda\exp(z)$, there exists a preperiodic dynamic ray $g$ that lands at $0$ \cite{SZ2}. The preimage rays $g_0$ and $g_1$ of $g$ do not land anywhere in $\C$, but they land at $-\infty\in\CT{\lambda\exp}$. The preimages $g_{01}$ of $g_1$ and $g_{00}$ of $g_0$ land together at a preimage $-\infty_0\in\CTT{\lambda\exp}$ of $-\infty$.}\label{fig:nonLandingRay}
	\end{figure}
	
	Recall the choice of inverse branches $f_{\ad{s}}^{-n}(\zeta)\in\tau_n(\ad{s})$ from Lemma and Definition~\ref{lemdef:TailsAtAddress}, as well as the domain $W_0:=\C\setminus(\overline{D}\cup\alpha)$ from the beginning of Section~\ref{sub:EscapingSet}.
	
	\begin{Def}[Landing of filaments] \label{def:extendedAccumulationSet}
		For a bounded external address $\ad{s}$,
		We say that the filament $G_{\ad{s}}$ \emph{lands} at a point $p\in\CTT{f}$ if and only if for every
    base point $\zeta\in W_0$, the sequence $\zeta_n:=\zeta_n(\ad{s}):=f_{\ad{s}}^{-n}(\zeta)\in\tau_n(\ad{s})$ converges in $\CTT{f}$ to $p$.
  In this case, we also denote the landing point of $G_{\ad{s}}$ by $L(\ad{s})\in\CTT{f}$.
	\end{Def}

  \begin{remark}
    
    In \cite[Definition 6.4]{BR}, the definition of landing of filaments in $\CC$ is defined via the accumulation set of the sequences $\zeta_n$ (see \cite[Definition 6.1]{BR}). As $\CC$ is compact, the accumulation set in $\CC$ is a singleton $\{p\}$ if and only if the sequence converges to $p$. Hence our definition of landing agrees with the original definition.
    Our space $\CTT{f}$ is not compact, so this equivalence fails in general, which is why we prefer to work directly with convergent sequences. It is shown in \cite[Proposition~6.5 (b)]{BR} that a filament lands in $\C$ if its closure (in $\C$) contains exactly one additional point, which is the landing point. This remark also applies for preperiodic filaments that land at infinity in our extended plane.
  \end{remark}

The main result of \cite{BR} is a generalization of the Douady--Hubbard landing theorem for post-singularly bounded polynomials to post-singularly bounded entire functions. We state a restricted version of \cite[Theorem~7.1]{BR} for post-singularly finite entire functions.
	
	\begin{Thm}[Landing theorem for {periodic filaments in $\C$}]
		Let $f$ be a psf entire function.
		Then every periodic filament of $f$ lands at a repelling periodic point $p\in\mathcal{J}(f)$.
		
		Conversely, every periodic point $p\in\mathcal{J}(f)$ is the landing point of at least one and at most finitely many filaments, all of which are periodic of the same period.\label{thm:LandingBR}
	\end{Thm}
	A point $p\in\CTT{f}\sm\C$ cannot be periodic because it is eventually mapped to a periodic post-singular value in $\C$ and is not part of its periodic orbit, so Theorem~\ref{thm:LandingBR} only makes a statement about points in $\C$ and filaments that land in $\C$.
	Using the extended plane $\CTT{f}$, Theorem~\ref{thm:LandingBR} extends neatly to the case of \emph{pre}periodic filaments.

\begin{Thm}[Landing of (pre\nobreakdash-)periodic filaments in $\CTT{f}$]
		Let $f$ be a psf entire function. Then every (pre\nobreakdash-)periodic filament of $f$ lands at a repelling (pre-)periodic point $p\in\mathcal{J}(\extended{f})$.
		
		Conversely, every (pre\nobreakdash-)periodic point $p\in\mathcal{J}(\extended{f})$ is the landing point of at least one and possibly infinitely many (pre\nobreakdash-)periodic filaments, all of which have the same period and preperiod.
\label{thm:PreperiodicLanding}
	\end{Thm}

One difference to the periodic case is that preperiodic rays and filaments can land ``at infinity'' at points in our extended complex plane. Another difference is that there are preperiodic points at which infinitely many filaments land together. This does not happen for preperiodic points in $\C$, but does happen precisely for all $p\in\mathcal{J}(\extended{f})\sm\mathcal{J}(f)$ (the point $-\infty$ in Figure~\ref{fig:nonLandingRay} is an example of such a point).	

	The proof of Theorem~\ref{thm:PreperiodicLanding} is given at the end of Section~\ref{section:TopologyOfFilaments} because we need to establish some topological properties of filaments beforehand.	
	
We are now ready to define the main object of this paper.
	
	\begin{Def}[Landing equivalence]
		We write $\ADPer$ for the set of (pre\nobreakdash-)periodic external addresses. Given two addresses $\ad{s},\ad{t}\in\ADPer$, we write $\ad{s}\landeq\ad{t}$ if the filaments $G_{\ad{s}}$ and $G_{\ad{t}}$ land together in $\CTT{f}$, i.e., if $L(\ad{s})=L(\ad{t})$. We call the equivalence relation $\landeq$ on $\ADPer$ the \emph{landing equivalence relation}.\label{def:landingEquivalence}
	\end{Def}
	
	Our main result is that for all psf entire functions and all (pre-)periodic filaments the landing equivalence relation can be described in terms of itineraries of filaments with respect to a dynamically meaningful partition of the plane. Roughly speaking, two periodic filaments land together if and only if they have the same itinerary with respect to this partition; for preperiodic filaments, the situation is a bit more complicated because preperiodic filaments might share a landing point that lies on the boundary of several partition sectors. The conceptual idea is not new, and so-called dynamic partitions have been defined in many other contexts in complex dynamics: for post-critically finite polynomials, one picks for every critical value an (extended) dynamic ray that lands at this value. The preimages of these (extended) rays naturally divide the complex plane into partition sectors, and periodic points can be distinguished in terms of their itineraries, see \cite{P1}. In \cite{SZ2}, dynamic partitions were defined for certain classes of exponential maps (attracting, parabolic, escaping, and psf parameters), and this is also where the term ``dynamic partition'' was introduced. In \cite{B}, dynamic partitions were defined for geometrically finite entire maps with dynamic rays. While many arguments given in this paper work in analogy to previous work on dynamic partitions, additional complications come from the fact that we are dealing with filaments instead of dynamic rays, and from having to deal with preperiodic filaments that do not land in $\C$ but in our extension of $\C$. We need to show some facts about the topology of filaments before we can define and use dynamic partitions in the general case.
	
	\section{Topology of filaments}
	\label{section:TopologyOfFilaments}
	
	\subsection{Landing of filaments}

	\begin{Lem}[Euclidean shrinking; \protect{\cite[Lemma 6.2]{BR}}]
		Suppose that $\Omega\subset\C\setminus P(f)$ is a bounded Jordan domain. Then for every $\varepsilon >0$ and every compact $K\subset\C$, there exists
		$N_\varepsilon\in\N$ with the following property: for every $n\geq N_\varepsilon$, every connected component of $f^{-n}(\Omega)$ that intersects $K$ has Euclidean
		diameter at most $\varepsilon$.\label{lem:EuclideanShrinking}	\end{Lem}
\begin{remark}
  A direct consequence of this lemma is the following: For every point $p \in \CTT{f}$ with neighbourhood $U$, there is a smaller neighbourhood $U'$ of $p$ with the following property: for $V\subset\C\setminus P(f)$ bounded Jordan domain, there is a $N \in \N$ such that for every $n \geq N$, every connected component of $f^{-n}(V)$ that intersects $U$ must be contained in $U'$: if $p \in \C$ it is enough to use the lemma with $\varepsilon$ small enough and $U' = D_{\varepsilon}(p), K = \overline{U'}$. If $p \in \CTT{f}\setminus \C$, the statement follows via pullback.

  This in particular implies that in order to check whether $G_{\ad{s}}$ lands at $p \in \CTT{f}$, it is enough to check it for one base point $\zeta \in W_0$. Indeed, for any other base point $\zeta' \in W_0$ we can choose a bounded Jordan domain $\Omega \subset W_0$ containing both. If the $\zeta_n$ tend to $p$, then by the previous paragraph, so do the $\zeta'_n$.
\end{remark}
We will use the following pair of lemmas from \cite{BR} in our construction. We will use the conclusion of the second lemma, but the statement of the construction refers to the first one, so while we cite the first one in completeness, we only cite the parts of the second lemma that we need.
\begin{Lem}[Domains for bounded-address filaments \protect{\cite[Lemma 6.9]{BR}}]
Let  $\zeta \in W_0$ belong to an unbounded connected component of $W_0 \cap f^{-1}(D)$, and let $R > 0$. Let $\mathcal{F}$ be any finite collection of fundamental domains of $f$. Then there is a Jordan domain $V \Subset \C \setminus P(f)$ with the following properties.
\begin{itemize}
    \item $\zeta \in V$.
    \item For all $F \in \mathcal{F}$, the unique preimage $\zeta_F$ of $\zeta$ in F also belongs to $V$.
    \item For all $F \in \mathcal{F}$, there is a connected component $A_F$ of $V \cap \overline{F}$ containing $\zeta_F$ as well as all points of $\overline{F}$ having modulus at most $R$.
    \item If $U$ is the connected component of $V \cap W_0$ containing $\zeta$, then $U \cap A_F$ intersects the unbounded connected component $\unbounded{F}$ of $F \setminus \overline{D}$ for all $F \in \mathcal{F}$.
\end{itemize}
\end{Lem}

	\begin{Lem}[Preimage domains \protect{\cite[Lemma 6.10]{BR}}]
Let $\mathcal{F}$ be a finite collection of fundamental domains of $f$, and assume that $\mathcal{F}$ contains every fundamental domain $F$ with $\overline{F} \cap \overline{D} \neq \emptyset$. Let $\zeta, R$ and $V$ be as in the previous lemma. If $R$ was chosen sufficiently large (depending only on $\mathcal{F}$), then the following holds.

Let $\ad{s} = F_0F_1 \dots$ be any external address. For $n \geq 1$, set $\zeta_n(\ad{s}) = f^{-n}_{\ad{s}}(\zeta)$. Also let $V_n(\ad{s})$ be the unique component of $f^{-n}(V)$ containing $\zeta_n(\ad{s})$. Then the following holds.
\begin{itemize}
\item The spherical diameter of $V_n(\ad{s})$ tends to $0$ as $n \rightarrow \infty$.
\item If $F_n \in \mathcal{F}$, then $\zeta_{n + 1}(\ad{s}) \in V_n(\ad{s})$, in particular $V_n(\ad{s}) \cap V_{n+1}(\ad{s}) \not= \emptyset$.
\end{itemize}
\label{lem:PreimageDomains}
	\end{Lem}
We will also need the following topological fact about filaments:
	\begin{Pro}[Topology of landing filaments; \protect{\cite[Theorem~7.1]{BR}}]\label{pro:filamentTopology}
		If $G_{\ad{s}}$ is a (pre-)periodic filament that lands in $\C$, then $\cl(G_{\ad{s}})$ does not separate $\C$.
	\end{Pro}
From this we immediately get the following corollary:
\begin{Cor}
    \label{cor:simplyconnectednbd}
    		If $G_{\ad{s}}$ is a (pre-)periodic filament that lands at $ p \in \C$, then there is a simply connected domain $U$ containing $G_{\ad{s}}$ such that $U \cap P(f) \subset \{p\}$.
\end{Cor}
\begin{proof}[Proof of Theorem~\ref{thm:PreperiodicLanding}]
		Let $\ad{t}$ be preperiodic of preperiod $k$, and let $\ad{s}=\sigma^{\circ n}(\ad{t})$ be the first periodic address on the forward orbit of $\ad{t}$.
		We want to show that $G_{\ad{t}}$ lands in $\CTT{f}$. By Theorem~\ref{thm:LandingBR}, the filament $G_{\ad{s}}$ lands at a periodic point $p\in\C$. Let us denote by $D^{sph}_\eps (p)$ the open spherical $\eps$ ball around $p$.
Let $\eps > 0$ small enough such that $D^{sph}_\eps (p) \cap P(f) \subset \{p\}$. Let $\mathcal{F}$ be the set of fundamental domains that appear in $\ad{t}$ or intersect $\overline{D}$. This is a finite set (this follows the fact that $\ad{t}$ is a bounded external address and there are only finitely many domains intersecting $\overline{D}$ (see \cite[Proposition 3.5]{BR})).
So we can choose $\zeta \in W_0$ and $V$ as in Lemma~\ref{lem:PreimageDomains} applied to $\mathcal{F}$. As $G_{\ad{s}}$ lands at $p$, there is an $N$ such that for all $n \geq N$, $V_n(\ad{s}) \subset D^{sph}_\eps (p)$. As $V_n(\ad{t}) \cap V_{n+1}(\ad{t}) \neq \emptyset$, we have that $V_n(\ad{t})$ must stay in the same connected component $U'$ of $f^{-k}(D^{sph}_\eps (p))$. As $D^{sph}_\eps (p) \cap P(f) \subset \{p\}$, it follows from the covering properties of $\extended{f}^{\circ n}$ that $U'$ contains a unique preimage $p'\in\extended{f}^{-n}(p)$. As $\zeta_n(\ad{s})$ tends to $p$, we get under this covering $\zeta_n(\ad{t})$ tends to $p'$, so $G_{\ad{t}}$ lands at $p'$.

Conversely, let $p'\in\CTT{f}$ be preperiodic, and let $p=\extended{f}^{\circ n}(p')$ be the first periodic point on the forward orbit of $p'$. There exists a filament $G_{\ad{s}}$ that lands at $p$. Choose $U$ in Corollary~\ref{cor:simplyconnectednbd}, and let $U'$ be the connected component of $\extended{f}^{-n}(U)$ containing $p'$. There exists a preimage filament $G_{\ad{t}}$ of $G_{\ad{s}}$ under $f^{\circ n}$ that satisfies $G_{\ad{t}}\subset U'$. By the first half of the lemma, the filament $G_{\ad{t}}$ lands at a point $L(\ad{t})\in\extended{f}^{-n}(p)$. As $p'$ is the only preimage of $p$ under $\extended{f}^{\circ n}$ in $U'$, we have $L(\ad{t})=p'$.
	\end{proof}

\subsection{Separation properties}
	
	In this subsection, we show that filaments that land together separate the plane in the same way as dynamic rays would. We have seen in Proposition~\ref{pro:filamentTopology} that a single filament that lands in the complex plane does not separate $\C$, and now we will show that $n$ filaments that land together in $\C$ separate the plane into precisely $n$ connected components. We also consider the separation properties of filaments that land together at points at infinity.

	The following result from \cite[Theorem 63.3 and Theorem 63.5]{Mu} will help us to deduce how several filaments that land together separate the plane.

	\begin{Thm}[A general separation theorem] \label{thm:PlaneSeparation}
		Let $C_1,C_2\subset\CC$ be non-separating continua such that $C_1\cap C_2=\{z\}$ for some $z\in\CC$. Then $\CC\setminus(C_1\cup C_2)$ is connected and simply connected. If, however, we have $C_1\cap C_2=\{z,w\}$ for distinct $z,w\in\CC$, then $\CC\setminus(C_1\cup C_2)$ consists of precisely two connected components that are both simply connected.
	\end{Thm}

The analogous result holds for a finite set of non-separating continua $C_i$.

\newpage
		
	\begin{Lem}[Filament separation, finite case]
		Let $G_{\ad{s}^1},\ldots ,G_{\ad{s}^n}$ be (pre\nobreakdash-)periodic filaments that land at a common point $p\in\mathcal{J}(f)$, and assume that they are indexed according to their cyclic order. Then $\CTT{f}\setminus\bigcup_i\cl_{\CTT{f}}(G_{\ad{s}^i})$ consists of precisely $n$ connected components, and all of these are all simply connected. Each of these components is bounded by exactly two filaments in $G_{\ad{s}^j}$ and $G_{\ad{s}^{j+1}}$ (with adjacent indices), together with their common landing point.

        Every filament $G_{\ad{t}}\notin\{G_{\ad{s}^1},\ldots ,G_{\ad{s}^n}\}$ is contained in exactly one connected component of $\CTT{f}\setminus\bigcup_i\cl_{\CTT{f}}(G_{\ad{s}^i})$, and it is that component for which there exists an index $j$ such that $\ad{s}^j\prec\ad{t}\prec\ad{s}^{j+1}$; see also Figure~\ref{fig:separationFinite}.

        Therefore, two filaments $G_{\ad{t}},G_{\ad{u}}\notin\{G_{\ad{s}^1},\ldots ,G_{\ad{s}^n}\}$ are contained in the same connected component of $\CTT{f}\setminus\bigcup_i\cl_{\CTT{f}}(G_{\ad{s}^i})$ if and only if there exists an index $j$ such that $\ad{s}^j\prec\ad{t}\prec\ad{s}^{j+1}$ and $\ad{s}^j\prec\ad{u}\prec\ad{s}^{j+1}$.
        \label{lem:SeparationFinite}
	\end{Lem}
	\begin{figure}[ht]
		\includegraphics[width=0.6\textwidth]{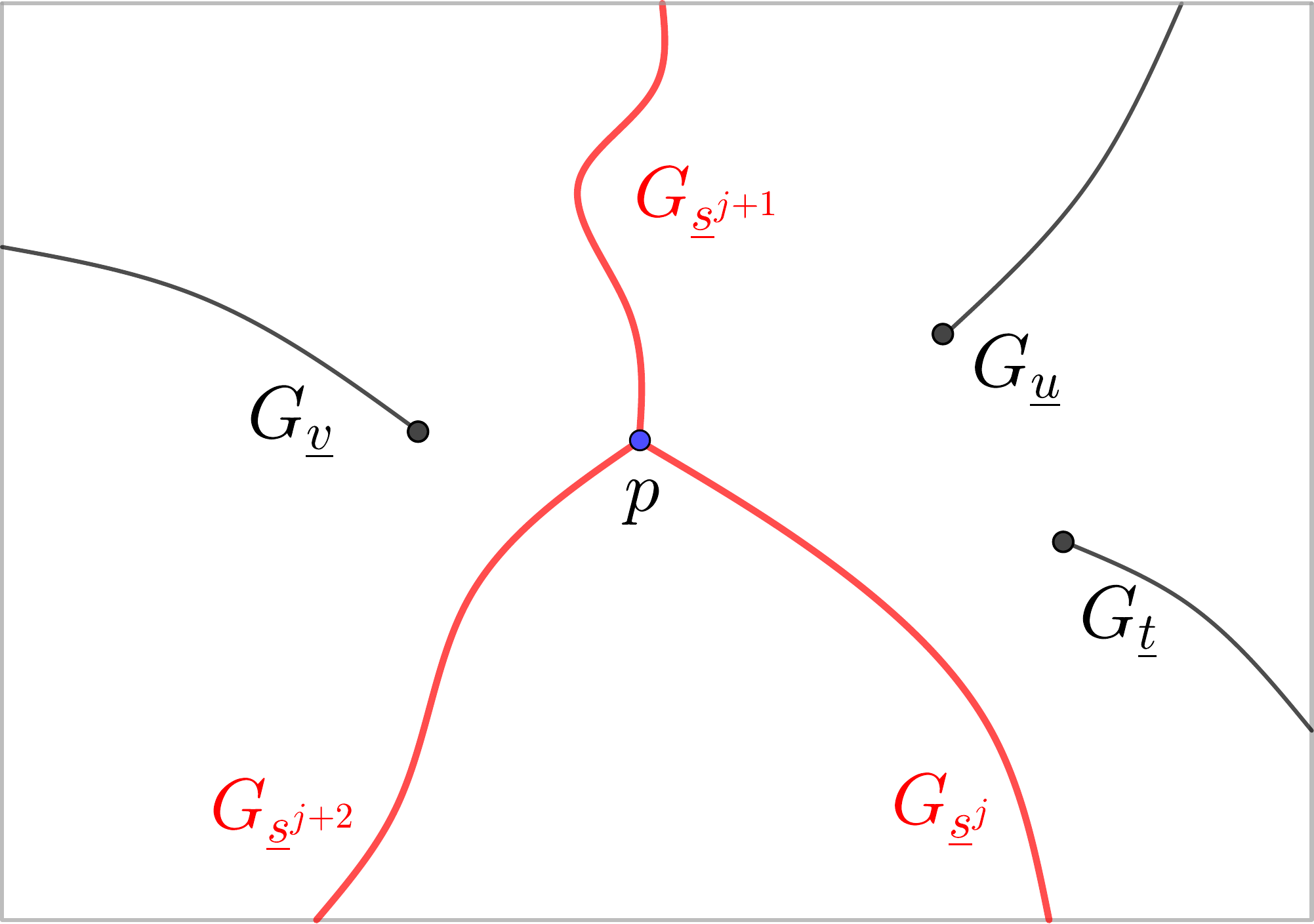}
		\caption{Sketch illustrating the statement of Lemma~\ref{lem:SeparationFinite}.}\label{fig:separationFinite}
	\end{figure}
	\begin{proof}
    We first prove the corresponding statement on the Riemann sphere (just replace $\CTT{f}$ by $\CC$ in the statement of the lemma). Note that the connected components of $X\setminus\bigcup_i\cl_{X}(G_{\ad{s}^i})$ are the same for $X=\C$ and $X=\CC$ because $\infty\in\cl_{\CC}(G_{\ad{s}^i})$. By Theorem~\ref{thm:PlaneSeparation}, the set $\CC\setminus(\cl_{\CC}(G_{\ad{s}^1})\cup\cl_{\CC}(G_{\ad{s}^2}))$ consists of precisely two connected components. Assume that $\CC\setminus\bigcup_{i\leq m}\cl_{\CC}(G_{\ad{s}^i})$ consists of precisely $m$ components $D_1,\ldots ,D_m$ for some $m<n$, all of which are simply connected. As $G_{\ad{s}^{m+1}}$ is connected by Theorem~\ref{thm:EscapingSet}, we have $G_{\ad{s}^{m+1}}\subset D_j$ for some $j$. The set $C:=\CC\setminus D_j$ is a non-separating continuum, and we have $C\cap \cl_{\CC}(G_{\ad{s}^{m+1}})=\{p,\infty\}$. By Theorem~\ref{thm:PlaneSeparation}, the complement $\CC\setminus(C\cup\cl_{\CC}(G_{\ad{s}^{m+1}}))$ consists of precisely two (simply) connected components; we call these $D_j'$ and $D_j''$. 
Therefore, 
\[
\CC\setminus\bigcup_{1\leq m+1}\cl_{\CC}(G_{\ad{s}^i})= \,D'_j \,\dot\cup\, D''_j  \, \dot\cup\,\bigcupdot_{i\neq j}  D_i 
\;,
\]
where all $m+1$ components are simply connected. 

It follows by induction that $\CC\setminus\bigcup_{i=1}^n\cl_{\CC}(G_{\ad{s}^i})$ consists of precisely $n$ simply connected components.
		
		For the second part of the statement, we choose closed unbounded connected sets  $B_i \subset G_{\ad{s}^i}$ and $C \subset G_{\ad{t}}$ such that $f$ escapes uniformly to infinity on these sets (this is possible by \cite[Prop. 4.10]{BR}). We then 
        choose $N$ large enough such that $C \subset \tau_N(\ad{t})$ and as well as the $B_i \subset \tau_N(\ad{s}^i)$ ($i\in\{1,\ldots,n\}$) and the fundamental tails have disjoint closures. It follows that the closed unbounded connected sets $B_i$ and $C$ have the same cyclic ordering at infinity in the sense of \cite[Appending 13]{BR} as the corresponding fundamental tails, so the second part follows.
		
    Finally, Remark~\ref{rem:VanKampen} shows that the lemma also holds in $\CTT{f}$.
	\end{proof}
	
	The content of Lemma~\ref{lem:SeparationFinite} is valid in a much more general context, namely for all post-singularly bounded entire functions and all bounded external addresses, and the reasoning is similar.
    
	We also need to consider filaments that land together at a transcendental singularity $p\in\mathcal{J}(\extended{f})\setminus\mathcal{J}(f)$. Already a single filament that lands at $p$ separates the complex plane (see Figure~\ref{fig:SeparationInfinite}). Yet, if we only consider all the infinitely many filaments that land at $p$, we have a result that is analogous to Lemma~\ref{lem:SeparationFinite}.

	\begin{remark}[Filaments landing at transcendental singularities]
		Given a transcendental singularity $p\in\mathcal{J}(\extended{f})\setminus\mathcal{J}(f)$, it follows from the fact that $p$ is a logarithmic singularity that the ordered set of the countably many filaments that land at $p$ is order-isomorphic to $\Z$.  Hence, we may denote them as $\{G_{\ad{s}^i}\}_{i\in\Z}$ where the $\ad{s}^i$ are indexed according to their cyclic order. Furthermore, it follows from the mapping properties of $f$ that for every $\ad{s}\notin\{\ad{s}^i\}_{i\in\Z}$ there exists an index $j\in\Z$ such that $\ad{s}^j\prec\ad{s}\prec\ad{s}^{j+1}$. This implies the following result, which extends Lemma~\ref{lem:SeparationFinite} to the case of singularities at $\infty$. We omit the proof.
	\end{remark}

	\begin{Lem}[Filament separation at $\infty$]
		Let $p\in\mathcal{J}(\extended{f})\setminus\mathcal{J}(f)$, and let $\{G_{\ad{s}^i}\}_{i\in\Z}$ be a set of distinct filaments that land at $p$ indexed according to their cyclic order. Two filaments $G_{\ad{t}}, G_{\ad{u}}\notin\{G_{\ad{s}^i}\}_{i\in\Z}$ are contained in the same connected component of $\CTT{f}\setminus\bigcup_{i\in\Z}\cl_{\CTT{f}}(G_{\ad{s}^i})$ if and only if there exists an index $j$ such that $\ad{s}^j\prec\ad{t}\prec\ad{s}^{j+1}$ and $\ad{s}^j\prec\ad{u}\prec\ad{s}^{j+1}$.\label{lem:SeparationInfinite}
	\end{Lem}
	\begin{figure}[ht]
		\includegraphics[width=0.6\textwidth]{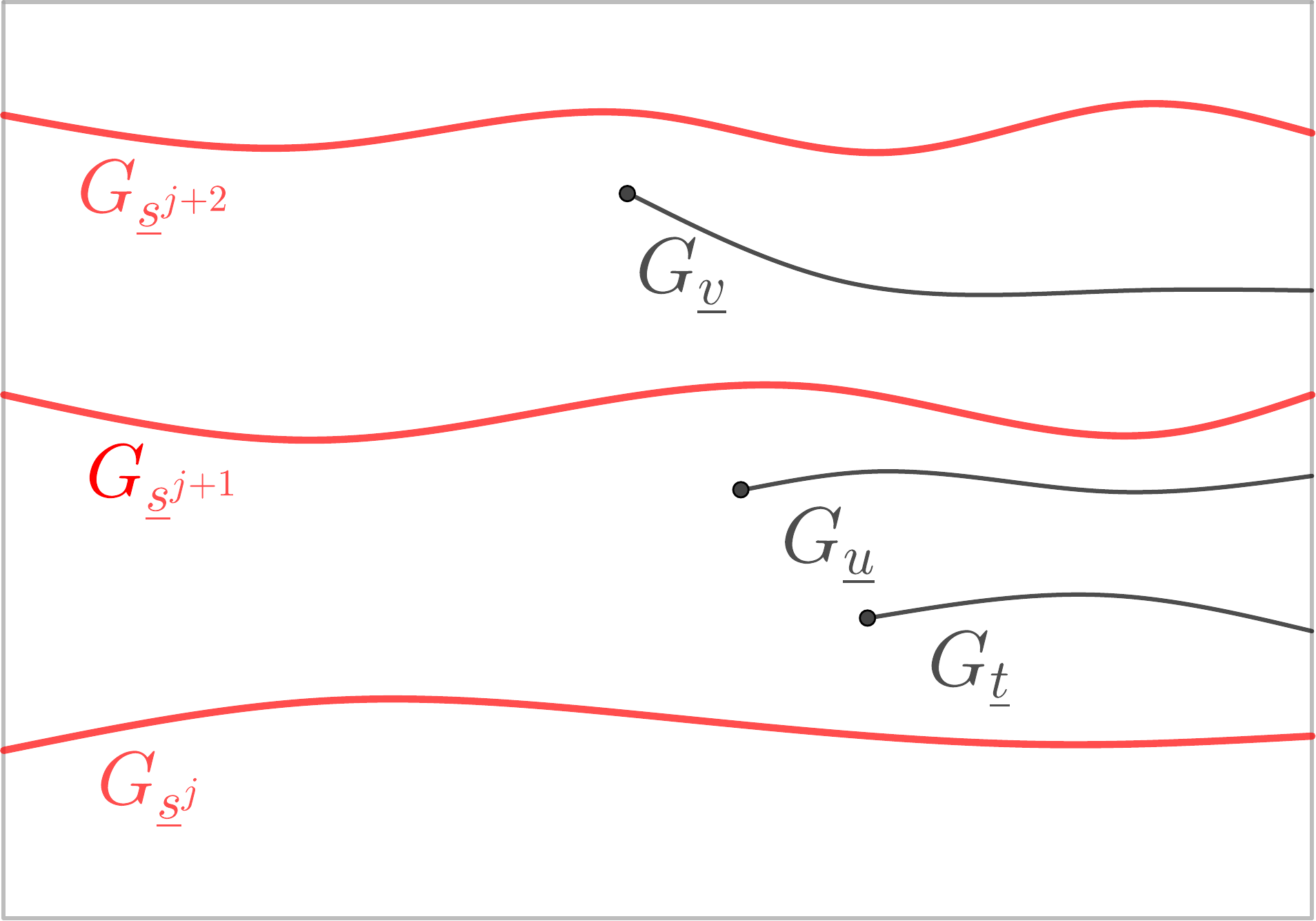}
		\caption{Sketch illustrating the statement of Lemma~\ref{lem:SeparationInfinite}. The filaments $G_{\ad u}$ and $G_{\ad t}$ are not separated by a $G_{\ad s^i}$, so they are in the same component, while $G_{\ad v}$ is in a different component.}
        \label{fig:SeparationInfinite}
	\end{figure}

	Let us summarize the content of Lemmas~\ref{lem:SeparationFinite} and~\ref{lem:SeparationInfinite} (using remark \ref{rem:VanKampen})
.
	
	\begin{Cor}[Separation by filaments]
		Let $p\in\mathcal{J}(\extended{f})$ be (pre\nobreakdash-)periodic, and let $\{G_{\ad{s}^i}\}_{i\in I}$ be the set of filaments that land at $p$ indexed according to their cyclic order, where we either have $I=\Z$ or $I=\Z/m\Z$ for some $m\geq 1$.
		
		Then every connected component of $W:=\CTT{f}\setminus\bigcup_{L(\ad{s})=p}\cl_{\CTT{f}}(G_{\ad{s}})$ is open and simply connected.
		
		Furthermore, two filaments $G_{\ad{t}},G_{\ad{u}}\notin\{G_{\ad{s}^i}\}_{i\in I}$ are
		contained in the same connected component of $W$ if and only if $\ad{s}^j\prec\ad{t}\prec\ad{s}^{j+1}$ and $\ad{s}^j\prec\ad{u}\prec\ad{s}^{j+1}$ for some $j\in I$.
        As before, if this is the case, then the component of $W$ that contains $G_{\ad{t}}$ and $G_{\ad{u}}$ is the one that is bounded by $G_{\ad{s}^j}$ and $G_{\ad{s}^{j+1}}$ and their common landing point. 
		\label{cor:SeparatingFilaments}
	\end{Cor}

\subsection{Extended filaments}

	Since filaments consist of escaping points and are thus in the Julia set, post-singular points in the Fatou set cannot be landing points of  filaments. In this subsection we extend filaments into Fatou components by internal rays of this component in order to obtain a dynamically meaningful continuum that connects a post-singular point in the Fatou set to $\infty$.

	\begin{notation}
		Let $U$ be a component of $\mathcal{F}(\extended{f})$, and let $p\in\partial_{\CTT{f}} U$. By Subsection~\ref{subsection:FatouDynamics},
		in particular Lemma~\ref{lem:DistinctLandingPoints}, there exists at most one internal ray of $U$ that lands at $p$. We denote this ray (if it exists) by $\internal_U[p]$.
	\end{notation}
	
	\begin{Def}[Extended filament]
		Let $q\in\CTT{f}$ be the center of a Fatou component $U(q)$. An \emph{extended filament that lands at $q$}  
		consists of a filament $G_{\ad{s}}$ that lands at a (pre\nobreakdash-)periodic boundary point $p=L(\ad{s})\in\partial_{\CTT{f}}U(q)$, extended by the internal ray $\beta_{U(q)}[p]$ of $U(q)$ that lands at $p$.	The \emph{external address of the extended filament $G_{\ad{s}}[q]$} is by definition $\ad{s}$. We denote this extended filament by $G_{\ad{s}}[q]:=\beta_{U(q)}[p]\cup\cl_{\CTT{f}}(G_{\ad{s}})$.
		\label{def:ExtendedFilament}
	\end{Def}
	
	In general, there are infinitely many extended filaments that land at any given Fatou center $q\in\mathcal{F}(\extended{f})$: the Fatou component $U(q)$ has infinitely many (pre\nobreakdash-)periodic boundary points, and all of them are landing points of (pre\nobreakdash-)periodic filaments.
	
	\begin{conv}
		Let $G_{\ad{s}}$ be a filament that lands at a (pre\nobreakdash-)periodic point $p=L(\ad{s})\in\CTT{f}$. In order to have a unified notation for filaments and extended filaments, we use $\dread{\ad{s}}{p}:=G_{\ad{s}}$ as an equivalent notion for the filament $G_{\ad{s}}$.
	\end{conv}

	Let $q\in\CTT{f}$ be the center of a Fatou component $U(q)$, and let $p\in\partial_{\CTT{f}}U(q)$ be (pre\nobreakdash-)periodic. As in the statement of Corollary~\ref{cor:SeparatingFilaments}, let $\{G_{\ad{s}^i}\}_{i\in I}$ be the set of filaments that land at $p$, indexed according to their cyclic order, and let
	$W:=\CTT{f}\setminus\bigcup_{i \in I}\cl_{\CTT{f}}(G_{\ad{s}^i})$.
	Let $V$ be the connected component of $W$ that contains $U(q)$. By Corollary~\ref{cor:SeparatingFilaments}, there exists an index $j\in I$ such that $G_{\ad{s}^j}\prec G_{\ad{t}}\prec G_{\ad{s}^{j+1}}$ for all filaments $G_{\ad{t}}\subset V$.
	
	\begin{Def}[Left and right supporting filaments]
		We call $G_{\ad{s}^j}$ the \emph{left supporting filament} of $U(q)$ at $p$, and $G_{\ad{s}^{j+1}}$ the \emph{right supporting filament} of $U(q)$ at $p$; see also Figure~\ref{fig:Kartoffel}.
\label{def:leftSupporting}
	\end{Def}

These are the largest, respectively smallest, of all filaments that land at $p$ with respect to the linear order of filaments in the complement of all those filaments that land at $\partial U\sm\{q\}$.	

	\begin{figure}[ht]
		\includegraphics[width=0.7\textwidth]{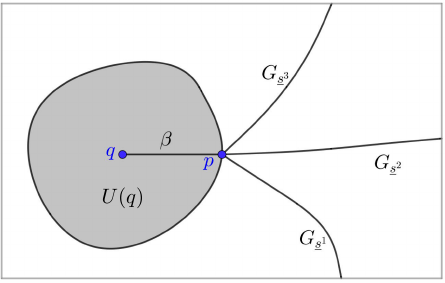
    }
		\caption{Sketch of the left and right supporting filaments $G_{\ad{s}^3}$ and $G_{\ad{s}^1}$ at a boundary point $p\in\partial U$ of the Fatou component $U=U(q)$. We obtain extended filaments by concatenating $G_{\ad{s}^i}$ with $\internal$, for $i=1,2,3$.} \label{fig:Kartoffel}
	\end{figure}
	
	\begin{Pro}[Topology of extended filaments]\label{pro:extendedTopology}
		Let $\dread{\ad{s}}{q}$ be an (extended) filament, and let $p:=L(\ad{s})$. If $p,q\in\C$, then $\cl_{\CC}(\dread{\ad{s}}{q})$ does not separate the Riemann sphere.
	\end{Pro}
	\begin{proof}
		For $q\in\mathcal{J}(f)$, this is the content of Proposition~\ref{pro:filamentTopology}. For $q\in\mathcal{F}(f)$, both $\cl_{\CC}(\internal_{U(q)}[p])$
		and $\cl_{\CC}(G_{\ad{s}})$ do not separate the Riemann sphere, and the intersection $\cl_{\CC}(\internal_{U(q)}[p])\cap\cl_{\CC}(G_{\ad{s}})=\{p\}$ is a singleton. It follows from
		Theorem~\ref{thm:PlaneSeparation} that $\dread{\ad{s}}{q}$ does not separate the Riemann sphere.
	\end{proof}
	A analogous result to Corollary~\ref{cor:SeparatingFilaments} is valid for extended filaments that support the same Fatou component.
	\begin{Lem}[Separation by extended filaments]
		Let $q\in\mathcal{F}(\extended{f})$ be the center of a Fatou component
		$U\subset\mathcal{F}(\extended{f})$. Let $\{G_{\ad{s}^i}\}_{i\in I}$ be filaments
		such that the landing points $L(\ad{s}^i)\in\partial U$ are distinct and satisfy $L(\ad{s}^i)\in\C$. If $q\in\C$, we require that the index set $I=\Z/m\Z$ for some $m\geq 1$. If $q\in\CTT{f}\setminus\C$, we require that $I=\Z$ and that the set $\{L(\ad{s}^i)\}$ of landing points is discrete. In both cases, we assume that the filaments are indexed according to their cyclic order.
		
		Let $G_{\ad{t}}[p]$ and $G_{\ad{u}}[r]$ be (extended) filaments such that \[G_{\ad{t}}[p]\cap G_{\ad{s}^i}[q]=G_{\ad{u}}[r]\cap G_{\ad{s}^i}[q]=\emptyset~~~\text{for all}~i\in I.\]
		Then the extended filaments $G_{\ad{t}}[p]$ and $G_{\ad{t}}[r]$ are contained in the same connected component of $\CTT{f}\setminus\bigcup_{i\in I}\cl_{\CTT{f}}(G_{\ad{s}^i}[q])$ if and only if there exists an index $j\in I$ such that $\ad{s}^j\prec\ad{t}\prec\ad{s}^{j+1}$ as well as $\ad{s}^j\prec\ad{u}\prec\ad{s}^{j+1}$.
		\label{lem:SeparationFatou}
	\end{Lem}
	\begin{proof}
		The proof works similarly to the proofs of Lemma~\ref{lem:SeparationFinite} and Lemma~\ref{lem:SeparationInfinite}.
	\end{proof}
	
	\section{Dynamic partitions}
\label{section:DynamicalPartitions}
We construct, for a given psf entire map $f$, a \emph{dynamic partition} by choosing for each singular value an (extended) filament that lands at this value, and taking the preimage of the complement: this way, the (extended) complex plane is partitioned into components called \emph{partition sectors} that map univalently onto the entire extended plane, minus the chosen (extended) filaments.

\begin{Def}[Dynamic partition]
	For a given psf entire map $f$, let $S(f)=\{a_1,\ldots ,a_n\}$ be the set of its singular values. Choose for every $a_i\in S(f)$ an (extended) filament $\dread{\ad{s}^i}{a_i}$ such that $\overline{\dread{\ad{s}^i}{a_i}}\cap\overline{\dread{\ad{s}^j}{a_j}}=\emptyset$ for $i\neq j$. For all $a_i\in S(f)\cap\mathcal{F}(f)$, we require that $L(\ad{s}^i)\in\C$ and that $\ad{s}^i$ is the address of the left supporting filament for $U(a_i)$ at $L(\ad{s}^i)$ (see Definition~\ref{def:leftSupporting}). 
	Define the \emph{base domain}
	\[
	B:=\C\setminus\bigcup_{a_i\in S(f)} \overline{\dread{\ad{s}^i}{a_i}}
	\]
	and define the \emph{dynamic partition} for $f$ (with respect to the chosen filaments) as the collection of connected components of $f^{-1}(B)$ and denote it by $\mathcal D$; see Figure~\ref{fig:dynamicPartition}.
	
	An element $D\in\mathcal{D}$ (i.e., a component of $f^{-1}(B)$) is called a \emph{partition sector}. The set 
\[
\partial\mathcal{D}:=f^{-1}\left(\bigcup_{a_i\in S(f)} \overline{\dread{\ad{s}^i}{a_i}}\right)\subset\C
\]
is called the \emph{boundary} of $\mathcal{D}$.
	
	We sometimes write $\mathcal{D}=\mathcal{D}(\{\dread{\ad{s}^i}{a_i}\})$ if we want to emphasize the (extended) filaments used to define the partition $\mathcal{D}$.\label{def:DynamicalPartitions}
\end{Def}

\begin{remark}
	Since we choose a single filament for each singular value, and in such a way that they are disjoint, the base domain $B$ is always simply connected and free of singular values. Therefore, the restriction $f\colon D\to B$ for any partition sector is biholomorphic. We are going to prove this in more detail in Lemma~\ref{lem:BasicProperties}.
	
	Taking left supporting filaments for singular values in the Fatou set is just a convention that allows for a more convenient description of which (pre\nobreakdash-)periodic filaments land together in terms of itineraries.
\end{remark}

Since we also want to assign itineraries to points at infinity, it is useful to extend the base domain to include $\CTT{f}\setminus\C$.

\begin{Def}[Extended dynamic partition]\label{def:extendedPartition}
	Let $\mathcal{D}=\mathcal{D}(\{\dread{\ad{s}^i}{a_i}\})$ be a dynamic partition for $f$. We call
	\[
	\extended{B}:=\CTT{f}\setminus\bigcup_{a_i\in S(f)} \cl_{\CTT{f}}(\dread{\ad{s}^i}{a_i})
	\]
	the \emph{extended base domain} of $\mathcal D$, and define the \emph{extended dynamic partition $\extended{\mathcal{D}}$} as the collection of connected components of $\extended{f}^{-1}(\extended{B})$.
	An element $\extended{D}\in\extended{\mathcal{D}}$ is called an \emph{extended partition sector}. We also write
	\begin{equation}
		\partial\extended{\mathcal{D}}:=\extended{f}^{-1}\left(\bigcup_{a_i\in S(f)} \cl_{\CTT{f}}\left(\dread{\ad{s}^i}{a_i}\right)\right)\label{Eq:BoundaryBaseDomain}
	\end{equation}
	for the boundary of the extended partition $\extended{\mathcal{D}}$.
\end{Def}
\begin{figure}
	\includegraphics[width=\textwidth]{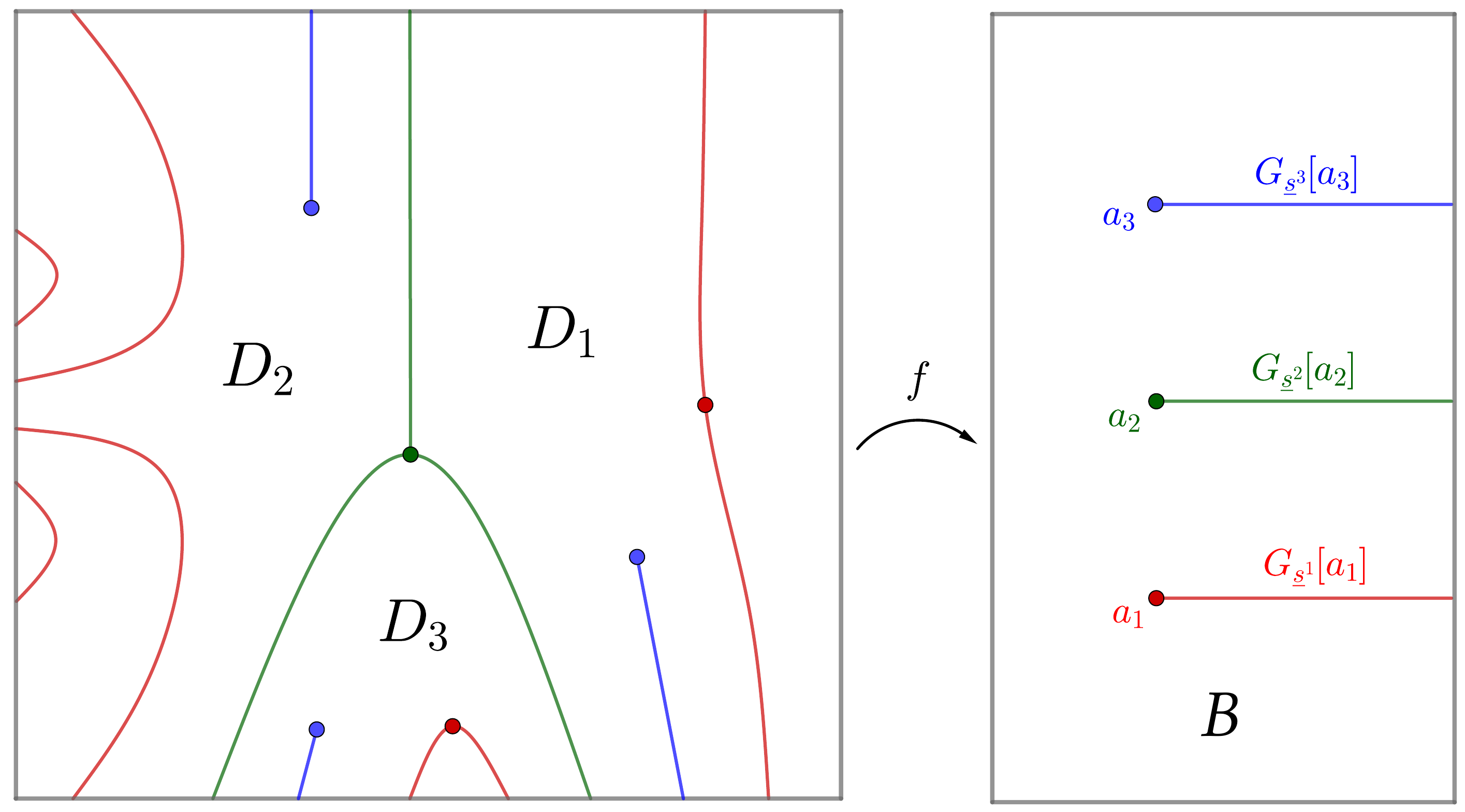}
	\caption{Sketch of some base domain (right), together with the corresponding dynamic partition (left). In the very left, we see (in red) parts of filaments that land together at some point at infinity.  We have drawn the filaments as straight lines in this image, although they are, in general, topologically far more complicated. However, regarding the features that are essential for dynamic partitions, a straight line is equivalent to a more complicated filament.\label{fig:dynamicPartition}}
\end{figure}
An essential property of dynamic partitions is that the base domain is evenly covered.

\begin{Lem}[Covering properties of dynamic partitions]
	Let $f$ be a psf entire function, and let $\mathcal{D}$ be a dynamic partition for $f$ with base domain $B$. Every partition sector $D\in\mathcal{D}$ is mapped biholomorphically onto $B$.
	Every extended partition sector $\extended{D}\in\extended{\mathcal{D}}$ is mapped homeomorphically onto $\extended B$.
\label{lem:BasicProperties}
\end{Lem}
\begin{proof}
	Since we required $L(\ad{s}^i)\in\C$ for all $i$, it follows from Proposition~\ref{pro:extendedTopology} that every individual $\cl_{\CC}(\dread{\ad{s}^i}{a_i})$ does not separate the sphere. Since the continua $\cl_{\CC}(\dread{\ad{s}^i}{a_i})$ only intersect each other at $\infty$, Theorem~\ref{thm:PlaneSeparation} implies that their union also does not separate the sphere, so their complement $B$ is simply connected. We have $B\cap S(f)=\emptyset$, so $f$ is a covering over $B$. This shows that every partition sector $D\in\mathcal{D}$ is mapped biholomorphically onto $B$.
	
	None of the (extended) filaments $\dread{\ad{s}^i}{a_i}$ accumulates at any of the points in $\CTT{f}\setminus\C$. Therefore, every $p\in\CTT{f}\setminus\C$ has a simply connected punctured neighborhood $V\subset B$. By the first paragraph, for a given partition sector $D$, there is a unique preimage component $U\subset D$ of $V$ under $f$. Therefore, the point $p$ has a unique preimage $q$ in the extended partition sector $\extended{D}$.
\end{proof}

The boundary of a dynamic partition consists of the preimages of the (extended) filaments $\dread{\ad{s}^i}{a_i}$, so every filament that is not one of these preimages is entirely contained in some partition sector. As filaments are distinguished by external addresses, a dynamic partition of the plane also partitions the space of external addresses: we call two external addresses \emph{equivalent} if the associated filaments are contained in the same partition sector. We want to describe this equivalence relation on $\AD\setminus\sigma^{-1}(\{\ad{s}^1,\ldots ,\ad{s}^n\})$. To this end, we need the concept of (un)linked sets of external addresses.

\begin{Def}[Unlinked addresses]
	We call two sets $T,U\subset\AD$ of external addresses \emph{unlinked} if $T\cap U=\emptyset$ and there do not exist addresses
	$\ad{t}^1,\ad{t}^2\in T$ and $\ad{u}^1,\ad{u}^2\in U$, such that
	\[
	\ad{t}^1\prec\ad{u}^1\prec\ad{t}^2\prec\ad{u}^2.
	\]
	Otherwise, $T$ and $U$ are called \emph{linked}.\label{def:unlinked}
\end{Def}
\begin{figure}[ht]
	\includegraphics[width=\textwidth]{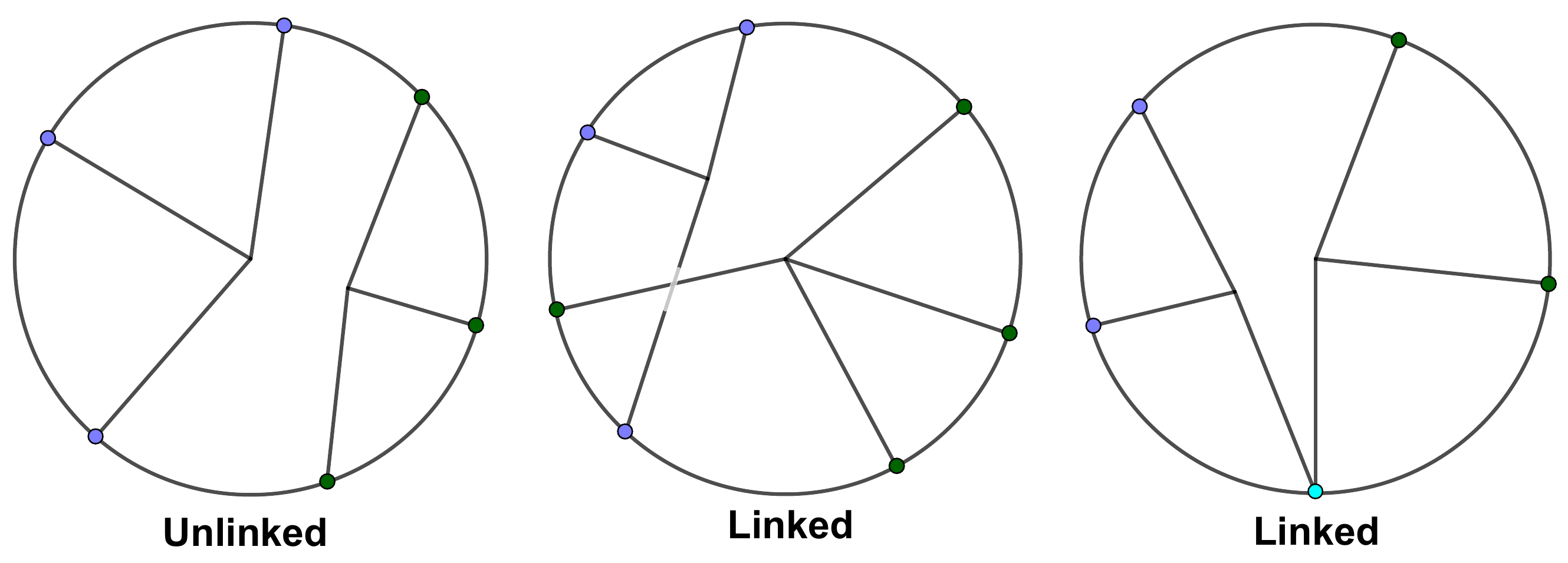}
	\caption{Illustration of linked and unlinked sets of addresses}\label{fig:Unlinked}
\end{figure}

The motivation for this concept comes from filaments that land together.  Using the definition of unlinked addresses, Corollary~\ref{cor:SeparatingFilaments} can be restated as follows: \emph{given a (pre-)periodic $p\in\mathcal{J}(\extended{f})$, two filaments $G_{\ad{t}}$ and $G_{\ad{u}}$ that do not land at $p$ are contained in the same component of $\CTT{f}\setminus\bigcup_{L(\ad{s})=p}\cl_{\CTT{f}}(G_{\ad{s}})$ (i.e.\ they are not separated by two filaments landing at $p$) if and only if the sets $\{\ad{s}\colon L(\ad{s})=p\}$ and $\{\ad{t},\ad{u}\}$ are unlinked}. A nice visualization of unlinked addresses is obtained by passing to the circle of addresses $\overline{\AD}$ that is the order completion of $\AD$ (see the paragraph after Definition~\ref{def:CyclicOrderFilaments}). The extended plane $\C\cup\overline{\AD}$ is homeomorphic to the closed unit disk $\overline{\D}$, and two subsets $T$ and $U$ of $\overline{\AD}$ are unlinked if and only if there are disjoint connected subsets $X,Y\subset\C\cup\overline{\AD}$ such that $T\subset X$ and $U\subset Y$, see \cite[Lemma~2.5]{BFH} and Figure~\ref{fig:Unlinked}.

Let $\ad{s}^1,\ldots ,\ad{s}^n$ be the external addresses of the (extended) filaments that land at the singular values of $f$ as in Definition~\ref{def:DynamicalPartitions}, ordered so that $\ad{s}^1<\ad{s}^2<\ldots <\ad{s}^n$ in the linear order defined in Section~\ref{subsubsec:CircleOfAddresses}. For a point $c\in\extended{f}^{-1}(S(f))$ with $\extended{f}(c)=a_j$, and such that $\ad{s}^j$ is the external address of the (extended) filament $G_{\ad{s}^j}[a_j]$ associated to $a_j$,  
we define 
\begin{align}
\label{def:CriticalSetNew}
\Crit{c}:= \{\ad{c}\in\AD\colon &\text{ $\sigma(\ad{c})=\ad{s}^j$ and the preimage component} \nonumber\\
&\text{of $G_{\ad{s}^j}[a_j]$ at external address $\ad c$ lands at $c$} \}
\;.
\end{align}
(Note that each filament, extended or not has countably many preimage components, and each of these contains a unique filament with a unique external address.)

Using this definition, we can now give a purely combinatorial definition of external addresses that belong to filaments in the same partition sector: we give the definition now and justify it in Proposition~\ref{pro:partition_correspondence}.

\begin{Def}[Dynamic Partitions of $\AD$]
	Two external addresses $\ad{t},\ad{u}\in\AD \sm\sigma^{-1}(\{\ad{s}^1,\ldots ,\ad{s}^n\})$ are called \emph{unlink equivalent} if $\Crit{c}$ and $\{\ad{t},\ad{u}\}$ are unlinked for all
	$c\in\extended{f}^{-1}(S(f))$. We call the resulting set $\mathcal{I}$ of equivalence classes the \emph{dynamic partition of $\AD$} (with respect to $\{\ad{s}^1,\ldots ,\ad{s}^n\}$). The elements $I\in\mathcal{I}$ are
	called \emph{partition sectors of $\AD$}.\label{def:CombinatorialPartition}
\end{Def}

This definition is such that there is a natural correspondence between the partition sectors of the dynamic partition $\mathcal{D}$ of the complex plane and the dynamic partition $\mathcal{I}$ of the space of external addresses. 

\begin{Lem}[Sector boundary]
	\label{lem:DTopology}
	\label{lem:DifferentSectors}
	Let $\mathcal{D}$ be a dynamic partition and $D\in\mathcal{D}$ a partition sector. For every singular value $a_i$, the intersection
	\[
	\partial_{\CTT{f}} D\cap\extended{f}^{-1}(\{a_i\})
	\]
	is a single point that we call $c_i$ (it depends on $D$). We have
	\begin{equation}
		\partial_{\CTT{f}} D\subset\bigcup_{i\in\{1,\ldots,n\}}\bigcup_{\ad{s}\in\Crit{c_i}}\cl_{\CTT{f}}(\dread{\ad{s}}{c_i})
		\label{Eq:BoundaryPartitionSector}
	\end{equation}
	and thus
	\begin{equation}
		D=\bigcap_{i\in\{1,\ldots,n\}} W_i,
		\label{Eq:PartitionSectorComplement}
	\end{equation}
	where $W_i$ is the complementary component of\/ $\bigcup_{\ad{s}\in\Crit{c_i}}\overline{\dread{\ad{s}}{c_i}}$ that contains $D$.
	
	If $p$ and $q$ are two points in different partition sectors, then there exists a $c_i$ such that $p$ and $q$ are contained in different components of \/ $\CTT{f}\setminus\bigcup_{\ad{s}\in\Crit{c}}\cl_{\CTT{f}}(\dread{\ad{s}}{c})$. 
\end{Lem}

\begin{proof}
	Choose pairwise disjoint domains $U_i\supset\dread{\ad{s}^i}{a_i}$ such that each $U_i$ is simply connected. Since $\extended f\colon \extended D\to \extended f(\extended D)=\extended{B}$ is a homeomorphism, $D$ intersects a single preimage component of each $U_i$. 
	
	Since $\partial\extended{\mathcal{D}}=\extended{f}^{-1}\left(\CTT{f}\setminus\bigcup_{a_i\in S(f)} \cl_{\CTT{f}}(\dread{\ad{s}^i}{a_i})\right)$ (see \eqref{Eq:BoundaryBaseDomain}), $\partial_{\CTT{f}} D$ is contained in the union of preimages of  those $\cl_{\CTT{f}}(\dread{\ad{s}^i}{a_i})$ that land at the $c_i$, and this is \eqref{Eq:BoundaryPartitionSector}.
	
	It follows that $D$ is separated from its complement by pairs of filaments that land at the various $c_i$, and this shows \eqref{Eq:PartitionSectorComplement} and the final claim.
\end{proof}

Using this lemma, we can establish the natural correspondence between partition sectors of the plane and of the space of external addresses.

\begin{Pro}[Sector correspondence]
	Consider two filaments $G_{\ad{t}}$ and $G_{\ad{u}}$ that are both not on the boundary of a partition sector. They are 
	contained in the same sector of $\mathcal{D}$ if and only if $\ad{t}$ and $\ad{u}$ are contained in the same sector of $\mathcal{I}$.\label{pro:partition_correspondence}
\end{Pro}
\begin{proof}
	First, assume that $G_{\ad{t}}$ and $G_{\ad{u}}$ are contained in the same sector of $\mathcal{D}$, and assume to the contrary that $\Crit{c}$ and
	$\{\ad{t},\ad{u}\}$ are linked for some $c\in\extended{f}^{-1}(S(f))$. By definition, there are external addresses $\ad{c}^1,\ad{c}^2\in\Crit{c}$ such that
	$\ad{c}^1\prec\ad{t}\prec\ad{c}^2\prec\ad{u}$. The (extended) filaments $\dread{\ad{c}^1}{c}$ and $\dread{\ad{c}^2}{c}$ land together at the point $c \in\CTT{f}$ and are contained in $\partial\mathcal{D}$. By Lemma~\ref{lem:SeparationFinite}, Lemma~\ref{lem:SeparationInfinite}, and Lemma~\ref{lem:SeparationFatou}, the filaments $G_{\ad{t}}$ and $G_{\ad{u}}$ are contained in distinct connected components of $\C\setminus\bigcup_{\ad{s}\in\Crit{c}}\overline{G_{\ad{s}}[c]}$ and thus in distinct partition sectors, a contradiction.
	
	Conversely, assume that $\ad{t}$ and $\ad{u}$ are unlink equivalent. If $G_{\ad{t}}$ and $G_{\ad{u}}$ were contained in distinct sectors of $\mathcal{D}$, then we could find a point $c\in\extended{f}^{-1}(S(f))$ such that $G_{\ad{t}}$ and $G_{\ad{u}}$ are contained in distinct connected components of $\C\setminus\bigcup_{\ad{s}\in\Crit{c}}\overline{G_{\ad{s}}[c]}$ by Lemma~\ref{lem:DifferentSectors}. By Lemma~\ref{lem:SeparationFinite}, Lemma~\ref{lem:SeparationInfinite}, and Lemma~\ref{lem:SeparationFatou}, this implies the existence of $\ad{c}^1,\ad{c}^2\in\Crit{c}$ such that $\ad{c}^1\prec\ad{t}\prec\ad{c}^2\prec\ad{u}$. Hence, $\{\ad{t},\ad{u}\}$ and $\Crit{c}$ would be linked, again a contradiction.
\end{proof}

Having thus established a natural bijection between the partition sectors of $\mathcal{D}$ and $\mathcal{I}$, we write $D(I)$ and $I(D)$ for the sector $D\in\mathcal{D}$ corresponding to $I\in\mathcal{I}$ and the sector $I$ corresponding to $D$ respectively.\label{def:CorrespondingSector}

\begin{Pro}[Topology of $I$]~\label{pro:SectorTopology}
	Let $I$ be a sector of the dynamic partition $\mathcal{I}$. The restriction
	\[
	\left.\sigma\right|_I\colon I\to\AD\setminus\{\ad{s}^1,\ldots ,\ad{s}^n\}
	\]
	is a bijection that preserves the cyclic order. We have
	\[
	I=(\ad{t}^1,\ad{t}^2)\cupdot(\ad{t}^3,\ad{t}^4)\cupdot\ldots \cupdot(\ad{t}^{2n-1},\ad{t}^{2n})
	\]
	with superscripts labeled modulo $2n$ and $\sigma(\ad{t}^{2i})=\sigma(\ad{t}^{2i+1})=\ad{s}^i$. There are distinct $c_i\in\extended{f}^{-1}(S(f))$, $i\in\{1,\ldots ,n\}$, such that $\ad{t}^{2i},\ad{t}^{2i+1}\in\Crit{c_i}$. \end{Pro}
\begin{proof}
	As shown in Lemma~\ref{lem:BasicProperties}, the restriction of $f$ to $D(I)$ is a conformal isomorphism onto $\C\setminus\bigcup_{i=1}^n \overline{\dread{\ad{s}^i}{a_i}}$. Therefore, every filament $G_{\ad{t}}$ with $\ad{t}\notin\{\ad{s}^1,\ldots ,\ad{s}^n\}$ has a unique preimage in $D(I)$. By Proposition \ref{pro:partition_correspondence}, every external address $\ad{u}\in\AD\setminus\{\ad{s}^1,\ldots ,\ad{s}^n\}$ that is realized by some filament has a unique preimage in $I$. Since realized addresses are dense in $\AD$, every external address $\ad{u}\in\AD\setminus\{\ad{s}^1,\ldots ,\ad{s}^n\}$ has a unique preimage in $I$. This shows that $\restr{\sigma}{I}$ is a bijection onto $\AD\setminus\{\ad{s}^1,\ldots ,\ad{s}^n\}$. Let $G_{\ad{t}},G_{\ad{u}},G_{\ad{v}}\subset D(I)$ be distinct filaments satisfying $G_{\ad{t}}\prec G_{\ad{u}}\prec G_{\ad{v}}$. As $\restr{f}{D_i}$ is a conformal map onto $B$ by Lemma~\ref{lem:BasicProperties}, thus an orientation-preserving homeomorphism, we have $G_{\sigma(\ad{t})}\prec G_{\sigma(\ad{u})}\prec G_{\sigma(\ad{v})}$. On the level of external addresses, this means that $\restr{\sigma}{I}$ preserves the cyclic order for every triple of realized external addresses. As realized addresses are dense in $\AD$, we conclude that $\restr{\sigma}{I}$ is order-preserving.
	
	Let $\ad{s}\in\AD\setminus\partial\mathcal{I}$ be an external address, and let $F$ be its initial entry. We have $\sigma(\ad{s})\in(\ad{s}^j,\ad{s}^{j+1})$ for some $j$. Recall that we required $\ad{s}^1<\ldots <\ad{s}^n$ in the linear order induced by $\alpha$ (see Definition~\ref{def:LinearOrder}). If $j\neq n$, then we have
	$\ad{s}\in(F\ad{s}^j,F\ad{s}^{j+1})$. If $j=n$, we further distinguish whether $\sigma(\ad{s})<\ad{s}^1$ or $\sigma(\ad{s})>\ad{s}^1$. In the former case, we have
	$\ad{s}\in(\Fpred{F}\ad{s}^n,F\ad{s}^1)$, while in the latter case we have $\ad{s}\in(F\ad{s}^n,\Fsucc{F}\ad{s}^1)$.
	Hence, $\AD\setminus\partial\mathcal{I}$ can be written as the disjoint union of intervals of the above form. Each of these intervals is fully contained in some sector of $\mathcal{I}$, and the restriction of $\sigma$ to any of these intervals preserves the cyclic order. By the above, $I$ is mapped bijectively onto $\AD\setminus\{\ad{s}^1,\ldots ,\ad{s}^n\}$, so there are fundamental domains $F_j$ ($j\in\{1,\ldots ,n\}$) such that
	\[
	I=((F_1)_p\ad{s}^n,F_1\ad{s}^1)\cupdot(F_2\ad{s}^1,F_2\ad{s}^2)\cupdot\ldots \cupdot(F_n\ad{s}^n,(F_n)_s\ad{s}^1).
	\]
	Hence, we see that $I$ is of the claimed form. The last statement follows from Lemma~\ref{lem:DTopology} where the $c_i$ from Lemma~\ref{lem:DTopology} coincide with the $c_i$ in the statement of this proposition.
\end{proof}

\begin{Cor}[Boundary correspondence]
	Let $I\in\mathcal{I}$ be a partition sector. In the notation of Proposition~\ref{pro:SectorTopology}, we have $G_{\ad{t}}\cap\overline{D(I)}\neq\emptyset$ if and only if $\ad{t}\in[\ad{t}^{2j-1},\ad{t}^{2j}]$ for some $j\in\{1,\ldots,n\}$. If $G_{\ad{t}}$ is (pre\nobreakdash-)periodic and $G_{\ad{t}}\cap\overline{D(I)}\neq\emptyset$, we have $L(\ad{t})\in\cl_{\CTT{f}}(D(I))$.\label{cor:BoundaryAddresses}
\end{Cor}
Note that $cl(G_{\ad{t}})\cap\overline{D(I)}$ might be also nonempty if $\ad{t} \in \Crit{c_i}$. 
\begin{proof}
	If $\ad{t}\notin\partial\mathcal{I}$, this follows directly from Proposition~\ref{pro:partition_correspondence}. In the boundary case, this follows easily from Proposition~\ref{pro:SectorTopology} and Lemma~\ref{lem:DTopology}.
\end{proof}

By adding either all left boundary points or all right boundary points to the sectors of $\mathcal{I}$, we obtain two full partitions of the space of external addresses.
\begin{CorDef}[Full partitions]
	For a partition sector $I\in\mathcal{I}$, we set
	\[
	I^-:=(\ad{t}^1,\ad{t}^2]\cup(\ad{t}^3,\ad{t}^4]\cup\ldots \cup(\ad{t}^{2n-1},\ad{t}^{2n}]
	\]
	and
	\[
	I^+:=[\ad{t}^1,\ad{t}^2)\cup[\ad{t}^3,\ad{t}^4)\cup\ldots \cup[\ad{t}^{2n-1},\ad{t}^{2n}),
	\]
	where the $\ad{t}^i$ are the external addresses introduced in Proposition~\ref{pro:SectorTopology}. In this way, we get two full partitions
	\[
	\AD=\bigcupdot_{I\in\mathcal{I}}I^-=\bigcupdot_{I\in\mathcal{I}}I^+
	\]
	of the space of external addresses.
	\label{cor:FullPartition}
\end{CorDef}
\begin{Lem}[Left and right sectors]
	Let $I\in\mathcal{I}$ be a partition sector. Then
	\begin{gather*}
		I^-=(\ad{u}^1,\ad{u}^2]\cupdot(\ad{u}^3,\ad{u}^4]\cupdot\ldots \cupdot(\ad{u}^{2m-1},\ad{u}^{2m}],\\
		I^+=[\ad{u}^1,\ad{u}^2)\cupdot[\ad{u}^3,\ad{u}^4)\cupdot\ldots \cupdot[\ad{u}^{2m-1},\ad{u}^{2m}),\\
		\overline{I}=[\ad{u}^1,\ad{u}^2]\cupdot[\ad{u}^3,\ad{u}^4]\cupdot\ldots \cupdot[\ad{u}^{2m-1},\ad{u}^{2m}]
	\end{gather*}
	for some $1\leq m\leq n$ and certain $\ad{u}^i\in\ADPer$. There are distinct $c_k\in C(\extended{f})$, $k\in\{1,\ldots ,m\}$, such that $\ad{u}^{2k},\ad{u}^{2k+1}\in\Crit{c_k}$, and we have $\ad{u}^{2k}\neq\ad{u}^{2k+1}$. The restrictions
	\[
	\left.\sigma\right|_{I^\pm}\colon I^\pm\to\AD
	\]
	are order-preserving bijections.\label{lem:LeftRight}
\end{Lem}
\begin{proof}
	The only statement that does not follow immediately from Proposition~\ref{pro:SectorTopology} is that each $c_k$ is a critical preimage of some $a_i\in S(f)$. The reason for this is that for a regular preimage $c\in\extended{f}^{-1}(S(f))$ the set $\Crit{c}=:\{\ad{c}\}$ is a singleton, so the left and right sectors of $\ad{c}$ agree, i.e., $\ad{c}\in I^-\cap I^+$ for some $I\in\mathcal{I}$. In this case, it follows that $\ad{c}$ is in the interior of some interval $(\ad{u}^{2k-1},\ad{u}^{2k})$, so $I^-$ would consist of less than $n$ intervals, a contradiction.
\end{proof}

It will sometimes be useful to talk about projections onto partition sectors in the plane as well as in the space of external addresses.
\begin{Def}[Projections]
	Let $D\in\mathcal{D}$ be a partition sector, and let $I=I(D)\in\mathcal{I}$ be the corresponding partition sector of the space of external addresses. We define the \emph{left projection map} $\pi_I^-\colon\AD\to I^-$ via
	\[
	\pi_I^-(\ad{s}):=
	\begin{cases}
		\ad{s}~~~\text{for}~\ad{s}\in I^-,\\
		\ad{u}^{2k}~~~\text{for}~\ad{s}\in(\ad{u}^{2k},\ad{u}^{2k+1}],
	\end{cases}
	\]
	and the \emph{right projection map} $\pi_I^+\colon\AD\to I^+$ via
	\[
	\pi_I^+(\ad{s}):=
	\begin{cases}
		\ad{s}~~~\text{for}~\ad{s}\in I^+,\\
		\ad{u}^{2k+1}~~~\text{for}~\ad{s}\in[\ad{u}^{2k},\ad{u}^{2k+1}).
	\end{cases}
	\]
	We define the \emph{projection map} $\pi_D\colon\Per(\extended{f})\to\Per(\extended{f})$ (see Definition~\ref{def:ExtendedPreperiodic}) in the following way: for a point $p\in\Per(\extended{f})$, let $\ad{s}$ be the address of an (extended) filament $\dread{\ad{s}}{p}$ that lands at $p$. We set
	\[
	\pi_D(p):=
	\begin{cases}
		p~~~\text{if}~\pi_I^-(\ad{s})=\ad{s},\\
		c_k~~~\text{if}~\pi_I^-(\ad{s})=\ad{u}^{2k}.
	\end{cases}
	\]
	It is easy to see that $\pi_D$ is well-defined, i.e., independent of the choice of address $\ad{s}$ and independent of the choice of $\pi_I^-$ versus $\pi_I^+$.\label{def:Projections}
\end{Def}

\section{Simple dynamic partitions and spiders}
\label{section:Spiders}

	In order to obtain a nice combinatorial description of the landing equivalence relation, we use dynamic partitions that satisfy some additional properties. One of these properties concerns the choice of extended filaments for singular values in the Fatou set.
	
	\begin{Def}[Minimal extended filament] \label{def:minimalLeftSupporting}
		Let $q\in\CTT{f}$ be the center of a Fatou component, let $m$ be the preperiod of $p$, and let $n$ be the period of $q$. An extended filament $\dread{\ad{s}}{q}$ is called \emph{minimal} if $\extended{f}^{\circ m+n}(\dread{\ad{s}}{q})=\extended{f}^{\circ m}(\dread{\ad{s}}{q})$.
	\end{Def}
	
Strictly speaking, it is not the filament itself that is minimal, but its period. This condition implies that filament and landing points have equal periods.

	We want to choose for every $p\in S(f)\cap\mathcal{J}(f)$ a filament that lands at $p$, and for every $q\in S(f)\cap\mathcal{F}(f)$ a minimal extended filament that lands at $q$, and in such a way that the chosen (extended) filaments are pairwise disjoint. This is not always
	possible, as for example the only fixed point on the boundary of a degree~$2$ fixed Fatou component might itself be a singular value. But if this is possible, and some additional properties hold, we are able to define a dynamic partition, called \emph{simple}, that has particularly nice properties. 	After the definition, we show that every psf entire function has an iterate that admits a simple dynamic partition.

\begin{Def}[Simple dynamic partition]\label{def:dynamicalPartitions}
		Let $f$ be a psf entire function, and let $\mathcal{D}=\mathcal{D}(\{G_{\ad{s}^i}[a_i]\})$ be a dynamic partition for $f$. We call $\mathcal{D}$ a \emph{simple dynamic partition} if the following properties are satisfied.
		\begin{enumerate}
			\item All periodic post-singular points are fixed. \label{en:PeriodicMeansFixed}
\item All (extended) filaments that land at periodic post-singular points are fixed. \label{en:Minimality}

\item
Every $a\in P(f)$ has an extended filament $\dread{\ad{s}(a)}{a}$ that lands at $a$ in such a way that 
$\bigcup_{a_i\in S(f)}\bigcup_{j\geq 0}f^{\circ j}(G_{\ad{s}^i}[a_i])=\bigcupdot_{a\in P(f)}\dread{\ad{s}(a)}{a}$. 
\label{en:ForwardInvariance}
		\end{enumerate}
	\end{Def}	
In condition \eqref{en:ForwardInvariance}, the union over $a\in P(f)$ on the right hand side means that $f$ has a spider in the sense of Definition~\ref{def:InvariantSpider} below, and the left hand side means that this spider is forward invariant.

	\begin{Pro}[Existence of simple dynamic partitions]
		Let $f$ be a psf entire function. There exists an $n\geq 1$ such that $f^{\circ n}$ admits a simple dynamic partition.\label{pro:SimpleExistence}
	\end{Pro}
	
	\begin{proof}
		By passing to a suitable iterate, we can make sure that Property~(\ref{en:PeriodicMeansFixed}) is satisfied. This property remains satisfied when passing to an iterate of this iterate. In particular, after passing to an iterate for a second time, we can make sure that all filaments that land at fixed post-singular points in $\mathcal{J}(f)$ are fixed, and choose one filament for each such point. Possibly after passing to an iterate for a third time, we are able to choose for every fixed $q\in P(f)\cap\mathcal{F}(f)$ a fixed internal ray $\internal_{U(q)}[p(q)]$ such that the landing points $p(q)$ are distinct, are not contained in $P(f)$, and all filaments that land at any of the $p(q)$ are fixed. It follows that there exists for every fixed $q\in P(f)\cap\mathcal{F}(f)$ a minimal left-supporting filament $\dread{\ad{s}(q)}{q}$ and for every fixed $p\in P(f)\cap\mathcal{J}(f)$ a filament $G_{\ad{s}(p)}$ that lands at $p$ such that the chosen (extended) filaments are pairwise disjoint.
		
		Assume that for a given $i\geq 0$ we have already chosen an (extended) filament for every $a\in P(f)$ such that $f^{\circ i}(a)$ is periodic (and hence fixed). Let $b\in P(f)$ be a point that is mapped to a periodic point after $i+1$ iterations, and set $c:=f(b)$. Let $\dread{\ad{s}(c)}{c}$ be the extended filament chosen for $c$. If $b$ has local mapping degree $d\geq 1$, then there are $d$ ways to lift $\dread{\ad{s}(c)}{c}$ to an (extended) filament that lands at $p$. It does not matter which lift we choose, so let $\dread{\ad{s}(b)}{b}$ be one of those lifts. By construction, the resulting filaments $\dread{\ad{s}(b)}{b}$ are still pairwise disjoint. We continue inductively until we have chosen an (extended) filament for every post-singular point. In particular, we have chosen an (extended) filament for every singular value. Denote these filaments by $\dread{\ad{s}^i}{a_i}$. Properties (\ref{en:Minimality}) and (\ref{en:ForwardInvariance}) are satisfied by construction.
	\end{proof}

	\begin{Def}[(Invariant) spider] \label{def:InvariantSpider}
A \emph{spider} for a psf entire function $f$ with postsingular set $P(f)=\{a_i\}$ is a choice, for each $a_i$, of a (possibly extended) filament $\dread{\ad{s}^i}{a_i}$ that lands at $a_i$, and so that all these filaments are disjoint. The spider is called \emph{invariant} if $\dread{\ad{s}^{i+1}}{a_{i+1}}=f(\dread{\ad{s}^i}{a_i})$ for all $i$.
	\end{Def}

\begin{Thm}[Invariant spider for iterates]
Every postsingularly finite entire function has a spider, and it has an  iterate that has an invariant spider.
\end{Thm}

\begin{proof}
The existence of an invariant spider for an iterate is an immediate consequence of the previous result. The existence of a spider for the original function follows.
\end{proof}

\begin{remark}
Spiders for postsingularly finite entire functions are also discussed in \cite[Section~6]{Bernhard}. In particular, the existence of periodic spiders (consisting of curves not filaments) is shown up to homotopy.
\end{remark}

\begin{remark}
Not every postsingularly finite entire function has an invariant spider; in facto not even every postcritically finite polynomial has one. Simple counterexamples are provided by quadratic polynomials where the critical orbit terminates at the interior fixed point (the ``$\alpha$-fixed point''): this fixed point is the landing point of several dynamic rays of period greater than $1$ that are permuted cyclically by the dynamics. There is an iterate of the polynomial for which all these rays are fixed, and this iterate has an invariant spider consisting of a single one of these fixed rays, plus appropriate preimages.
\end{remark}

Spiders for postcritically finite polynomials have been introduced in \cite{Spiders}. Postcritically finite polynomials can be classified in at least two fundamental ways: by invariant Hubbard trees and by invariant (or possibly periodic) spiders. There is every reason to believe that a similar classification should be possible for postsingularly finite entire functions, also in terms of spiders, and in terms of Hubbard trees. The existence of Hubbard trees for psf entire functions is a more subtle issue: in general, invariant Hubbard trees do not exist, not even for postsingularly finite exponential maps \cite{PRS}. The proper analog for transcendental maps are \emph{Homotopy Hubbard trees}. Their existence has been shown in \cite{DavidThesis}, and the present paper forms the first step towards this result.

\Newpage

\section{Itineraries and boundary symbols}
\label{section:Itineraries}

The main purpose of dynamic partitions is to distinguish points combinatorially in terms of itineraries. The itinerary of a point (or,  more precisely, of an orbit) is the sequence of partition sectors the point visits under iteration. This section discusses the case when the orbit of a  point lands on the partition boundary. This boundary consists of (possibly extended) filaments and their landing points. These landing points are (pre-)periodic and map under $f$ in the first step to a (possibly extended) filament that lands at a singular value.

There are several types of points on boundary filaments: every point on a filament escapes, so it has an external address; this case will be discussed in Section~\ref{Sub:BoundariesExtAddr}. If a boundary filament is non-extended, then its landing point is a (pre-)periodic point in the Julia set, and several filaments can land at the same point. This case is treated in Section~\ref{Sub:NonExtFilament}. If a boundary filament is extended, then we have the endpoint, which is in the Fatou set, and it can also be the landing point of several filaments; see Section~\ref{Sub:ExtFilament}. Finally, an extended filament contains the landing point of the non-extended filament, which is a (pre-)periodic point in the Julia set, but it is not on any other filament; this case is discussed in Section~\ref{Sub:IntermediateFilamentPoint} (the points in the extended filament along the internal ray in the Fatou set, other than the endpoint, have the same combinatorics).

Every boundary filament is a preimage of a (possibly extended) filament that lands at a singular value. Most interesting are those preimages that land either at a critical point, or at a tract in the extended complex plane. However, in general a singular value has regular preimages in $\C$ (so that a neighborhood in $\C$ of this preimage maps univalently to a neighborhood of the singular value), and there are preimage filaments that land at such regular preimages. Among all boundary filaments, they land alone at their respective landing points, and we call these \emph{trivial boundary filaments} because they contribute to the partition boundary only in a trivial way.

Now we discuss the various kinds of points on partition filaments in order.

\subsection{Landing point of a non-extended filament }
\label{Sub:NonExtFilament}

Let $p$ be the landing point of a non-extended filament; it is necessarily in the Julia set.
 In this case, $\extended{f}(p)=a_i\in S(f)$ is a singular value, and we say that $p$ is a \emph{Julia pre-singular boundary point} (here ``pre-singular'' is understood with respect to the first iterate, not higher iterates).

\begin{figure}[ht]
	\includegraphics[width=0.6\textwidth]{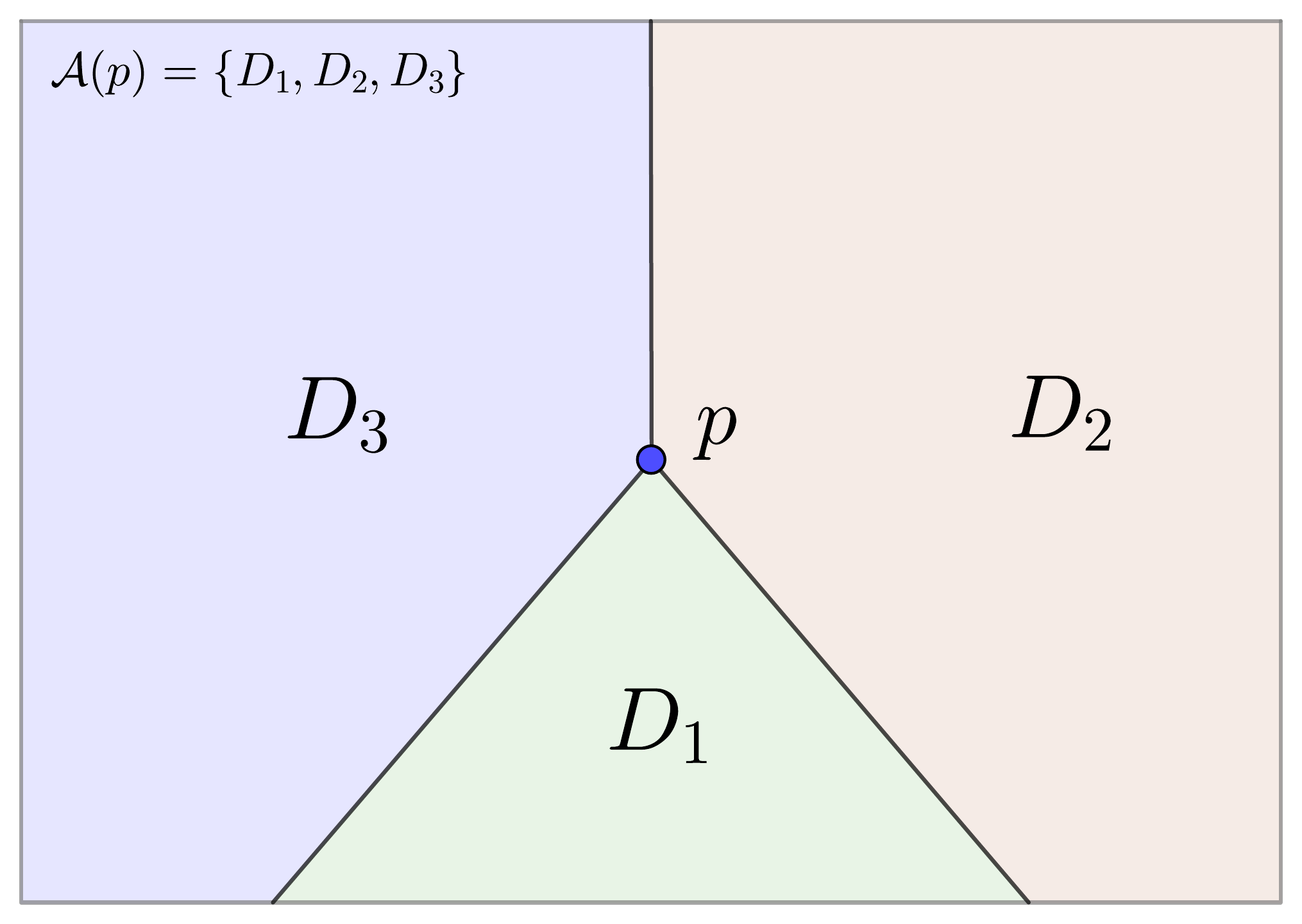}
	\caption{Sketch of a Julia pre-singular boundary point $p\in\C$, with $\mathcal A(f)$ finite}
	\label{fig:BoundarySymbols1}
\end{figure}
Let $\mathcal{A}(p)\subset\mathcal{D}$ consist of all partition sectors $D\in\mathcal{D}$ for which $p\in\partial_{\CTT{f}}D$; see Figure~\ref{fig:BoundarySymbols1}. The set $\mathcal{A}(p)$ is finite if $p\in\C$ is a critical point, and infinite if $p\in\CTT{f}\setminus\C$ is a transcendental singularity over $a_i$. The set $\mathcal A(p)$ consists of at least two elements except for trivial boundary filaments.

We introduce the boundary symbol $\star_{\mathcal{A}(p)}$\label{def:JuliaPreCritical} and call it a \emph{Julia pre-singular boundary symbol}.

\subsection{Non-escaping Julia point on an extended filament}
\label{Sub:IntermediateFilamentPoint}

Let $p$ be the landing point of a non-extended filament that is part of an extended filament on the boundary; see Figure~\ref{fig:BoundarySymbols2} for an illustration.. In this case, $w:=\extended{f}(p)\in\C$ is the landing point of the filament $G_{\ad{s}^i}$ used for the definition of the extended filament $\dread{\ad{s}^i}{a_i}$ that lands at the singular value $a_i\in S(f)\cap\mathcal{F}(f)$.

Recall that we have chosen the extended filaments so as to avoid all further singular values, so $w$ is a regular value and hence $p\in\C$. We say that $p$ is a \emph{Julia regular boundary point}.

Let $G_{\ad{t}}$ be the filament that lands at $p$ so that $f(G_{\ad{t}})=G_{\ad{s}^i}$. By Corollary~\ref{cor:FullPartition}, there are unique partition sectors $I_l, I_r\in\mathcal{I}$ such that $\ad{t}\in I_l^+\cap I_r^-$. We denote by $D_l(p), D_r(p)\in\mathcal{D}$ the corresponding partition sectors in the dynamic plane.
\begin{figure}[ht]
	\includegraphics[width=0.6\textwidth]{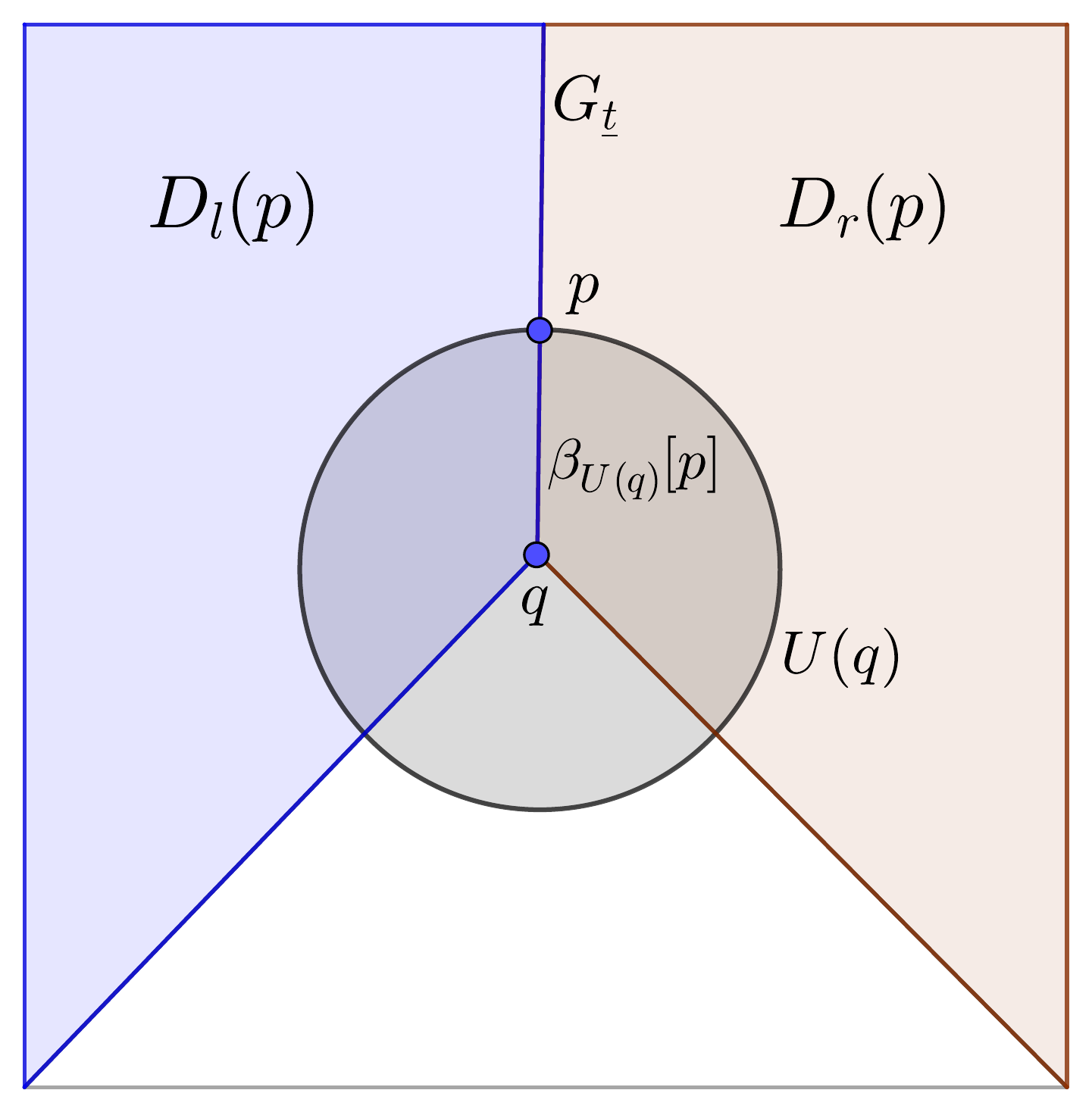}
	\caption{Sketch of a Julia regular boundary point $p$ (the shaded disk denotes the Fatou component that contains a point $q$ with $f(q)=a_i$).}\label{fig:BoundarySymbols2}
\end{figure}
We introduce the boundary symbol $\binom{D_l(z)}{D_r(z)}$ and call it a \emph{Julia regular boundary symbol}\label{def:JuliaRegular}, and we call the point $p$ a \emph{Julia regular boundary point}.

The two types of boundary symbols that we have introduced so far allow us to unambiguously define itineraries for all (pre\nobreakdash-)periodic points in the extended Julia set.

\Newpage

\begin{Def}[Itineraries of boundary points]\label{def:boundaryItineraries}
	Let $p\in\mathcal{J}(\extended{f})$ be (pre\nobreakdash-)periodic and write $p_i:=\extended{f}^{\circ i}(p)$. We define the \emph{itinerary}
	\[
	\It(p~\vert~\extended{\mathcal{D}}):=\itin{p}{\mathcal{D}}:=\itn{u}=\mathtt{u_0u_1\ldots }
	\]
	\emph{of\/ $p$ w.r.t.\ $\mathcal{D}$} as the sequence of partition sectors and boundary symbols defined via
	\[
	\mathtt{u_i}:=
	\begin{cases}
		D_i~~~\text{if}~p_i\in\extended{D_i},\\
		\star_{\mathcal{A}(p_i)}~~~\text{if}~p_i\in\partial\extended{\mathcal{D}}~\text{and}~p_i~\text{is Julia pre-singular},\\
		\binom{D_l(p_i)}{D_r(p_i)}~~~\text{if}~p_i\in\partial\extended{\mathcal{D}}~\text{and}~p_i~\text{is Julia regular}.
	\end{cases}
	\]
\end{Def}

\begin{remarknumber}
The dynamic partitions we are most interested in have the property that for all boundary points $p$ in the Julia set the itinerary $\itin{p}{\mathcal{D}}$ contains boundary symbols of only one of the two kinds. In order for this to be the case, we choose the extended filament that leads to a singular value in the Fatou set in such a way that it runs through a  (pre-)periodic point $p$ in the Julia set for which the forward orbit is disjoint from all postsingular points. This is easily possible since there are only finitely many postsingular points to avoid. This condition is not encoded in Definition~\ref{def:DynamicalPartitions} of dynamic partitions because the latter definition uses only function theoretic properties of the map $f$, not the dynamics. However, it will be satisfied in \emph{simple dynamic partitions} that we introduce in Definition~\ref{def:dynamicalPartitions}.
\end{remarknumber}

In a way, points on the partition boundary realize several (pre\nobreakdash-)periodic itineraries simultaneously. For example, if $a_i\in S(f)\cap\mathcal{F}(f)$ is a {superattracting} fixed point and $p=L(\ad{s}^i)$, where $\ad{s}^i$ is the address from Definition~\ref{def:DynamicalPartitions}, then there are sectors $D_r:=D_r(p)\in\mathcal{D}$ to the right of $\dread{\ad{s}^i}{a_i}$ and $D_l:=D_l(p)$ to the left of $\dread{\ad{s}^i}{a_i}$ (when standing at $a_i$ looking in the direction of the internal ray part of $\dread{\ad{s}^i}{a_i}$) as described above (see also Figure~\ref{fig:BoundarySymbols2}). It turns out that, in this case, there is no periodic point of itinerary $D_rD_rD_r\ldots$ and, likewise, no periodic point of itinerary $D_lD_lD_l\ldots$ (this is easy to prove using a standard hyperbolic contraction argument for the backwards iteration, see Proposition~\ref{pro:UniqueItineraries}). So one can say that the there is a point $p$ on the boundary of these sectors that realizes both of these itineraries at the same time. 

In the following, we define adjacency relations to describe which (pre\nobreakdash-)periodic sequences of partition sectors are realized by boundary points.

\begin{Def}[Adjacent itineraries]\label{def:adjacentItineraries}
	Let $\itn{u}=(D_i)_{i=0}^\infty$ be a sequence of partition sectors $D_i\in\mathcal{D}$. Let $p\in\mathcal{J}(\extended{f})$ be (pre\nobreakdash-)periodic. We call $\itin{p}{\mathcal{D}}$ \emph{adjacent to} $\itn{u}$ if one of the following is true:
	\begin{enumerate}[label=(\arabic*)]
		\item We have $\itin{p}{\mathcal{D}}=\itn{u}$. Then $\itin{p}{\mathcal{D}}$ is free of boundary symbols.
		\item The itinerary $\itin{p}{\mathcal{D}}=\itn{t}=\mathtt{t_1t_2\ldots }$ contains Julia pre-singular boundary symbols. For all boundary symbols $\mathtt{t}_i=\star_{\mathcal{A}_i}$, we have $D_i\in\mathcal{A}_i$. Otherwise, we have $\mathtt{t_i}=D_i$.\label{en:JuliaPreCritical}
		\item The itinerary $\itin{p}{\mathcal{D}}=\itn{t}=\mathtt{t_1t_2\ldots }$ contains Julia regular boundary symbols. For all boundary symbols $\mathtt{t}_i=\binom{D_l^i}{D_r^i}$, we have $D_l^i=D_i$. Otherwise, we have $\mathtt{t_i}=D_i$.\label{en:left}
		\item The itinerary $\itin{p}{\mathcal{D}}=\itn{t}=\mathtt{t_1t_2\ldots }$ contains Julia regular boundary \mbox{symbols}. For all boundary symbols $\mathtt{t}_i=\binom{D_l^i}{D_r^i}$, we have $D_r^i=D_i$. Otherwise, we have $\mathtt{t_i}=D_i$.\label{en:right}
	\end{enumerate}
\end{Def}

Every (pre\nobreakdash-)periodic point in the Julia set has an itinerary that is adjacent to an itinerary without boundary symbols.

\subsection{Landing point of extended filament}
\label{Sub:ExtFilament}

Let us now turn attention to the case of (pre\nobreakdash-)periodic points in the Fatou set and explain how they fit into the context of dynamic partitions. 

We may still define itineraries of (pre\nobreakdash-)periodic Fatou points as sequences that are, by definition, neither equal nor adjacent to any of the itineraries realized by points in the Julia set. To do this, we introduce another kind of boundary symbol. For a (pre\nobreakdash-)periodic point $p\in\mathcal{F}(\extended{f})$ that lies on the partition boundary, we denote by $\mathcal{A}(p)\subset\mathcal{D}$ --- just as in the Julia pre-crictial case --- the set of partition sectors $D\in\mathcal D$  for which $p\in\partial_{\CTT{f}}D$. We introduce the boundary symbol $\diamond_{\mathcal{A}(p)}$ and call it a \emph{Fatou boundary symbol}\label{def:ThirdKind}. It is analogous to $\star_{\mathcal{A}(p_i)}$ but for points in the Fatou set; see also Figure~\ref{fig:BoundarySymbols2}.

\begin{Def}[Itineraries of points in the Fatou set]
	Let $p\in\mathcal{F}(\extended{f})$ be (pre\nobreakdash-)periodic and write $p_i=\extended{f}^{\circ i}(p)$. We define the \emph{itinerary}
	\[
	\itin{p}{\extended{\mathcal{D}}}:=\itin{p}{\mathcal{D}}:=\itn{u}=\mathtt{u_0u_1}\ldots
	\]
	\emph{of $p$ w.r.t.\ $\mathcal{D}$} to be the sequence of partition sectors and Fatou boundary symbols defined via
	\[
	\mathtt{u_i}:=
	\begin{cases}
		D_i~~~\text{if}~p_i\in\extended{D_i},\\
		\diamond_{\mathcal{A}(p_i)}~~~\text{if}~p_i\in\partial_{\CTT{f}}\mathcal{D}.
	\end{cases}
	\]
\end{Def}

Extended boundary filaments also contain the internal ray in the Fatou set that connect the two (pre-)periodic points $q$ and $p$ in the Fatou set resp.\ the Julia set. These points all converge to a periodic orbit in the Fatou set, and they are on the boundary of the same partition sectors as the point $p$ in the Julia set, as described in Section~\ref{Sub:IntermediateFilamentPoint}.

\subsection{Boundaries and external addresses}
\label{Sub:BoundariesExtAddr}

In this subsection, we define itineraries and the adjacency relation on the level of external addresses. They also serve to provide symbolic dynamics to the escaping points on boundary filaments.

\begin{Def}[Combinatorial itineraries]
	Let $\mathcal{I}$ be a dynamic partition of $\AD$. For every external address $\ad{t}\in\partial\mathcal{I}$ there exist unique partition sectors $I_r(\ad{t}),I_l(\ad{t})\in\mathcal{I}$ such that $\ad{t}\in I_r(\ad{t})^-\cap I_l(\ad{t})^+$ by Corollary~\ref{cor:FullPartition}.
	
	We define the \emph{itinerary $\itinc{\ad{t}}=\itn{u}=\mathtt{u_0u_1\ldots }$ of $\ad{t}$ w.r.t.\ $\mathcal{I}$} as the sequence defined via 
	\[
	\mathtt{u_i}=
	\begin{cases}
		I~~~\text{if}~\sigma^{\circ i}(\ad{t})\in I,\\
		\binom{I_l(\ad{t})}{I_r(\ad{t})}~~~\text{if}~\sigma^{\circ i}(\ad{t})\in\partial\mathcal{I}.
	\end{cases}
	\]
\end{Def}

Adding consistently either the left- or the right-sided boundary addresses to the partition sectors, we obtain two full partitions of the space of external addresses, i.e., we have $\AD=\bigcupdot_{I\in\mathcal{I}}I^-$ and $\AD=\bigcupdot_{I\in\mathcal{I}}I^+$. Therefore, it makes sense to distinguish between left-sided and right-sided itineraries that contain boundary symbols.

\begin{Def}[Left- and right-sided itineraries]\label{def:LeftSidedItineraries}
	Let $\ad{s}$ be an external address and let $\mathcal{I}$ be a dynamic partition. For every $j\geq 0$, there exists a unique partition sector $I_j^l$ such that $\sigma^{\circ j}(\ad{s})\in(I_j^l)^-$. The sequence
	\[
	\itinl{\ad{s}}{\mathcal{I}}:=(I_j^l)_{j=0}^{\infty}=I_0^lI_1^l\ldots
	\]
	is called the \emph{left-sided itinerary of $\ad{s}$ w.r.t.\ $\mathcal{I}$}. In the same manner, we define the \emph{right-sided itinerary of $\ad{s}$ w.r.t.\ $\mathcal{I}$} as the sequence
	\[
	\itinr{\ad{s}}{\mathcal{I}}:=(I_j^r)_{j=0}^{\infty}=I_0^rI_1^r\ldots 
	\]
	where we have $\sigma^{\circ j}(\ad{s})\in (I_j^r)^+$.
	We call the itinerary $\itinc{\ad{s}}$ \emph{adjacent} to $\itn{u}\in\mathcal{I}^\N$ if $\itinl{\ad{s}}{\mathcal{I}}=\itn{u}$ or $\itinr{\ad{s}}{\mathcal{I}}=\itn{u}$.
\end{Def}

Let us describe the relationship between the adjacency relations on the space of external addresses and on the plane (at least in one direction).

\begin{Lem}[Adjacent itineraries]
	Let $\ad{s}\in\AD$ be (pre\nobreakdash-)periodic and let $(I_j)_{j=0}^\infty$ be a sequence of partition sectors. If $\itinc{\ad{s}}$ is adjacent to $(I_j)_{j=0}^\infty$, then $\itin{L(\ad{s})}{\mathcal{D}}$ is adjacent to $(D(I_j))_{j=0}^\infty$.\label{lem:AdjacentItineraries}
\end{Lem}
\begin{proof}
	Let $p:=L(\ad{s})$ be the landing point of $G_{\ad{s}}$, and set $p_j := \extended{f}^{\circ j}(p)$ for $j\geq 0$. If the forward orbit of $p$ does not intersect the partition boundary, the statement of the lemma follows from Proposition~\ref{pro:partition_correspondence}. Otherwise, we distinguish whether the forward orbit of $p$ contains Julia pre-singular or Julia regular boundary points. In both cases, we only prove the lemma for left-sided itineraries, i.e., we assume that $\itinl{\ad{s}}{\mathcal{I}}=(I_j)_{j=0}^\infty$. For right-sided itineraries, the proof works in complete analogy.
	
	Assume that $p$ contains Julia pre-singular points on its forward orbit, and choose $k\geq 0$ such that $p_k\in\partial\extended{\mathcal{D}}$. By Corollary~\ref{cor:BoundaryAddresses}, we have $p_k\in\cl_{\CTT{f}}(D(I_k))$. Therefore, $\itin{p}{\mathcal{D}}$ is adjacent to $(D(I_j))_{j=0}^\infty$ by Property~\ref{en:JuliaPreCritical} of Definition~\ref{def:adjacentItineraries}.
	
	Next, assume that $p$ contains Julia regular boundary points on its forward orbit, and choose $k\geq 0$ such that $p_k\in\partial\extended{\mathcal{D}}$. It follows from the definition of left and right sectors that $D(I_k)=D_r(p_k)$. Therefore, $\itin{p}{\mathcal{D}}$ is adjacent to $(D(I_j))_{j=0}^\infty$ by Property~\ref{en:left} of Definition~\ref{def:adjacentItineraries}.
\end{proof}
	
\Newpage

\section{The landing equivalence relation}\label{section:LandingEquivalence}

Our first important result on simple dynamic partitions is that for every\linebreak (pre\nobreakdash-)periodic sequence $(D_i)_{i=0}^{\infty}$ of partition sectors there is at most one\linebreak (pre\nobreakdash-)periodic point $p\in\mathcal{J}(\extended{f})$ for which $\itin{p}{\mathcal{D}}$ is adjacent to $(D_i)_{i=0}^{\infty}$. The following two lemmas are needed for the proof.
	
	\begin{Lem}[Fixed boundary points]
		Let $\mathcal{D}$ be a simple dynamic partition, and let $p\in\partial\mathcal{D}\cap\mathcal{J}(f)$ be periodic and hence fixed. Let $\itn{u}\in\mathcal{D}^{\N}$ be adjacent to $\itin{p}{\mathcal{D}}$. Then $\itn{u}=DDD\ldots $ for some $D\in\mathcal{D}$.
		
		Let $g:=f^{-1}\colon\C\setminus\bigcup_i\overline{\dread{\ad{s}^i}{a_i}}\to D$ be the inverse branch mapping the complement of the (extended) filaments used for the definition of $\mathcal{D}$ onto $D$. Then, for every neighborhood $U$ of $p$, there exists a point $b_0\in U\setminus\bigcup_i\overline{\dread{\ad{s}^i}{a_i}}$ such that the sequence $(b_i)_{i=0}^\infty$ defined via $b_i:=g^{\circ i}(b_0)$ is well-defined and satisfies $b_i\to p$.\label{lem:fixedPoints}
	\end{Lem}
	\begin{proof}
		We distinguish two cases. First, assume that $p=a_i\in S(f)$. Let $U_i\supset\dread{\ad{s}^i}{a_i}$ be simply connected such that $U_i\cap\dread{\ad{s}^j}{a_j}=\emptyset$ for $j\neq i$. Let $V_i$ be the connected component of $f^{-1}(U_i)$ containing $p$. Then $\partial\mathcal{D}\cap V_i=\dread{\ad{s}^i}{a_i}$, and the complement $V:=V_i\setminus\dread{\ad{s}^i}{a_i}$ is simply connected by Theorem~\ref{thm:PlaneSeparation}. It follows that $V\subset D$ for some partition sector $D\in\mathcal{D}$. This implies that $D$ is the only partition sector for which $p\in\partial_{\CTT{f}} D$. By Definition~\ref{def:adjacentItineraries}, the only itinerary adjacent to $\itin{p}{\mathcal{D}}$ is $\itn{u}=DDD\ldots $.
		
		Let $U\subset V_i$ be a linearizing neighborhood of $p$, and let $b_0\in U\cap V$. Let $f_D^{-1}\colon\C\setminus\bigcup_i\overline{\dread{\ad{s}^i}{a_i}}\to D$ be the unique inverse branch of $f$ with the prescribed domain and co-domain. Inductively, we define $b_i:=f_D^{-1}(b_{i-1})$. By the preceding paragraph, we have $b_i\in V\subset D$ for all $i\geq 0$. As $U$ is a linearizing neighborhood of $p$, it follows that $\lim_{i\to\infty} b_i=p$.
		
		The second case is that $p\notin S(f)$, so $p\in\dread{\ad{s}^i}{a_i}$ for some fixed $a_i\in S(f)\cap\mathcal{F}(f)$. It follows from Definition~\ref{def:adjacentItineraries} that either $\itn{u}=D_l(p)D_l(p)\ldots $ or $\itn{u}=D_r(p)D_r(p)\ldots $. Assume w.l.o.g.\ that $\itn{u}=D_l(p)D_l(p)\ldots $; the second case works analogously. Let $U$ be a linearizing neighborhood for $p$. There exists a point $b_0\in U\cap D_l(p)$ that can be connected to a point $w\in\internal_{U(a_i)}[p]$ via an arc $\gamma\colon[0,1]\to\C$ satisfying $\gamma([0,1))\subset D_l(p)\cap U$. Inductively, we define $b_i:=f_{D_l(p)}^{-1}(b_{i-1})$. Let $w'$ be the unique preimage of $w$ on $\internal_{U(a_i)}[p]$. The unique lift of $\gamma$ starting at $w'$ ends at $b_1$, as conformal maps are orientation-preserving. It follows inductively that $b_i\in U$ for all $i\geq 0$ and therefore $b_i\to p$.
	\end{proof}
	
	\begin{Lem}[Preimage itineraries]
		Let $\itn{u}=(D_i)_{i=0}^\infty$ be a sequence of itinerary domains, and let $p\in\mathcal{J}(\extended{f})$ be a (pre\nobreakdash-)periodic point whose itinerary $\itin{p}{\mathcal{D}}$ is adjacent to $\itn{u}$. Let $D\in\mathcal{D}$ be a partition sector. Then there exists one and only one $q\in\extended{f}^{-1}(p)$ such that $\itin{q}{\mathcal{D}}$ is adjacent to $D\itn{u}$.\label{lem:PreimageItineraries}
	\end{Lem}
	\begin{proof}
		First, assume that $p\in\CTT{f}\setminus\bigcup_{i\in\{1,\ldots ,n\}}\cl_{\CTT{f}}(\dread{\ad{s}^i}{a_i})$. Then every preimage of $p$ is contained in some partition sector and every partition sector contains precisely one preimage of $p$ by Lemma~\ref{lem:BasicProperties}. Therefore, the unique preimage $q\in\extended{f}^{-1}(p)\cap D$ is the only point for which $\itin{q}{\mathcal{D}}$ is adjacent to $D\itn{u}$.
		
		Else, we have $p\in\overline{\dread{\ad{s}^i}{a_i}}$ for some $i\in\{1,\ldots ,n\}$. If $p\in S(f)$, then there exists a unique preimage $q\in\extended{f}^{-1}(p)$ satisfying $q\in\partial_{\CTT{f}}D$ by Lemma~\ref{lem:DTopology}. By the definition of adjacency, this implies that $q$ is the only preimage of $p$ whose itinerary is adjacent to $D\itn{u}$.
		
		Otherwise, we have $p\notin S(f)$ and $p=L(\ad{s}^i)$. By Proposition~\ref{pro:SectorTopology} and the definition of right and left sectors, there are unique preimages $q_l,q_r\in f^{-1}(p)$ such that $D_l(q_l)=D$ and $D_r(q_r)=D$. If $\itin{p}{\mathcal{D}}$ is related to $\itn{u}$ via case~\ref{en:left} of Definition~\ref{def:adjacentItineraries}, then $q_l$ is the only preimage of $p$ whose itinerary is adjacent to $D\itn{u}$. If instead $\itin{p}{\mathcal{D}}$ is related to $\itn{u}$ via case~\ref{en:right} of Definition~\ref{def:adjacentItineraries}, then $q_r$ is the only preimage of $p$ whose itinerary is adjacent to $D\itn{u}$.
	\end{proof}
	
	We are now in the position to prove that a (pre\nobreakdash-)periodic itinerary is realized by at most one (pre\nobreakdash-)periodic point. The result is a generalization of \cite[Proposition 4.4]{SZ2}, and the underlying proof strategy is taken from there.
	
	\begin{Pro}[Unique itineraries]
		Let $\mathcal{D}$ be a simple dynamic partition, and let $\itn{u}=E_0\ldots  E_{k-1}\overline{D_0D_1\ldots  D_{m-1}}$ be a (pre\nobreakdash-)periodic sequence of partition sectors. Let $p,q\in\mathcal{J}(\extended{f})$ be (pre\nobreakdash-)periodic points, and assume that $\It(p~\vert~\mathcal{D})$ and $\It(q~\vert~\mathcal{D})$ are both adjacent to $\itn{u}$. Then $p=q$.\label{pro:UniqueItineraries}
	\end{Pro}
	\begin{proof}
		Let $\dread{\ad{s}^1}{a_1},\ldots ,\dread{\ad{s}^n}{a_n}$ be the (extended) filaments from Definition~\ref{def:dynamicalPartitions} that land at the singular values of $f$. We set
		\[W:=\C\setminus\bigcup_{i=1}^n\bigcup_{m\geq 0}f^{\circ m}(\overline{\dread{\ad{s}^i}{a_i}}).\]By Property~(\ref{en:ForwardInvariance}) of Definition~\ref{def:dynamicalPartitions}, the complement $\C\setminus W$ consists of finitely many pairwise disjoint (extended) filaments. By Proposition~\ref{pro:filamentTopology} and Proposition~\ref{pro:extendedTopology}, none of these (extended) filaments separates the plane. Hence, the complement $W$ is simply connected by Theorem~\ref{thm:PlaneSeparation}. The domain $W$ is backward invariant and satisfies $W\subset\C\setminus\bigcup_{i=1}^n \overline{\dread{\ad{s}^i}{a_i}}$. Hence, for every itinerary domain $D$ the unique branch $f^{-1}\colon\C\setminus\bigcup_{i=1}^n \overline{\dread{\ad{s}^i}{a_i}}\to D$ of the inverse of $f$ restricts to a branch $f_D^{-1}\colon W\to D$. We set $W_D:=f_D^{-1}(W)\subset W$.
		
		Assume for now that $\itn{u}=\overline{D_0\ldots  D_{m-1}}$ is periodic, and $p$ and $q$ are periodic. Assume by contradiction that $p\neq q$. In addition, assume that $p,q\in W$. Note that this implies $f^{\circ i}(p),f^{\circ i}(q)\in W$ for all $i\geq 0$, and therefore $\itin{p}{\mathcal{D}}=\itin{q}{\mathcal{D}}=\itn{u}$. Choose $m\in\N$ such that $f^{\circ m}(p)=p$ as well as $f^{\circ m}(q)=q$, and write $p_i:=f^{\circ i}(p)$ as well as $q_i:=f^{\circ i}(q)$. We have $p_i=f_{D_i}^{-1}(p_{i+1})$ as well as $q_i=f_{D_i}^{-1}(q_{i+1})$. Setting
		$g:=f_{D_0}^{-1}\circ f_{D_1}^{-1}\circ\ldots \circ f_{D_{m-1}}^{-1}\colon W\to W$, we obtain a univalent self-map of $W$ satisfying $g(p)=p$ and $g(q)=q$. If $p$ and
		$q$ were distinct, this would imply $g=\text{id}$. Hence, we have $p=q$.
		
		In general, both $p$ and $q$ might be contained in $\partial W$. Assume for now that $p\in\partial W$ and $q\in W$. Then Property~(\ref{en:PeriodicMeansFixed}) of
		Definition~\ref{def:dynamicalPartitions} implies that $p$ is fixed and so is every itinerary adjacent to $\itin{p}{\mathcal{D}}$ by Lemma~\ref{lem:fixedPoints}. Hence, we have $\itn{u}=DDD\ldots $ for some itinerary domain $D\in\mathcal{D}$. By Lemma~\ref{lem:fixedPoints}, there exists a point $b_0\in W$ such that the sequence $(b_i)_{i=0}^\infty$ defined via $b_{i+1}:=f_D^{-1}(b_i)$ converges to $p$. Assume that $q\in W$, and let $\alpha_0\colon[0,1]\to W$ be a smooth arc of finite length w.r.t.\ the hyperbolic metric on $W$ that connects $b_0=\alpha_0(0)$ to $q=\alpha_0(1)$. By the preceding paragraph, we have $f_D^{-1}(q)=q$. Hence, the lift $\alpha_1:=f_D^{-1}\circ\alpha_0$ connects $b_1$ to $q$, and it satisfies
		\[
		l_W(\alpha_1)< l_{f_D^{-1}(W)}(\alpha_1)=l_W(\alpha_0),
		\]
		where the first inequality follows from the Comparison Principle, and the second equality follows from the Schwarz Lemma. Inductively, we define $\alpha_{i+1}:=f_D^{-1}\circ\alpha_i$. As we have $l_W(\alpha_i)< l_W(\alpha_0)$ and $\alpha_i(0)=b_i\to p\in\partial W$ for $i\to\infty$, the Euclidean lengths of the $\alpha_i$ tend to $0$ contradicting our assumption that $p$ and $q$ are distinct.
		
		An analogous argument works if both $p$ and $q$ are contained in $\partial W$. In this case, we also have to take a point $\tilde{b}_0$ as in Lemma~\ref{lem:fixedPoints} for $q$ and connect $b_0$ to $\tilde{b}_0$ via an arc $\alpha_0$ of finite hyperbolic length.
		
		Let us now prove the full statement of the lemma. We set $p_i:=\extended{f}^{\circ i}(p)$ and $q_i:=\extended{f}^{\circ i}(q)$. Choose $n\geq 0$ such that $p_n$ and $q_n$ are both periodic. Then $\itn{\tilde{u}}:=\sigma^{\circ n}(\itn{u})$ is also periodic, and both $\itin{p_n}{\mathcal{D}}$ and $\itin{q_n}{\mathcal{D}}$ are adjacent to $\itn{\tilde{u}}$. By the periodic case proved above, we have $p_n=q_n$. Applying Lemma~\ref{lem:PreimageItineraries} inductively for $n$ times, we obtain $p=q$.
	\end{proof}
	We have thus shown that every (pre\nobreakdash-)periodic itinerary is realized by at most one (pre\nobreakdash-)periodic point. The next step is to show that every (pre\nobreakdash-)periodic itinerary is, in fact, realized.
	\begin{Lem}[Pullbacks of intervals]
		Let $I, I'\in\mathcal{I}$ be partition sectors, and let $J\subset I$ be an interval. If $\ad{s}^i\notin J$ for all $i$, then $J':=\sigma^{-1}(J)\cap I'$ is an interval of the form $J'=\{F\ad{t}~\vert~\ad{t}\in J\}$ for some fundamental domain $F$. \label{lem:IntervalPullbacks}
	\end{Lem}
	\begin{proof}
		For the proof, we are going to use concepts introduced in Sections~\ref{subsubsec:CyclicOrders} and \ref{subsubsec:CircleOfAddresses}. In particular, we are going to
		talk about the predecessor $\Fpred{F}$ of a fundamental domain, the circle of addresses $\overline{\AD}$, and the intermediate address $\alpha$ with its preimages $\alpha_{\Fpred{F}}^F\in\sigma^{-1}(\alpha)$.
		
		By the proof of Proposition~\ref{pro:SectorTopology}, there are fundamental domains $F_1,\ldots ,F_n$ such that
		\[
		I'=(\Fpred{(F_1)}\ad{s}^n,F_1\ad{s}^1)\cupdot(F_2\ad{s}^1,F_2\ad{s}^2)\cupdot\ldots \cupdot(F_n\ad{s}^n,\Fsucc{(F_n)}\ad{s}^1).
		\]
		The restriction $\restr{\sigma}{I'}$ is an order-preserving bijection onto $\AD\setminus\{\ad{s}^1,\ldots ,\ad{s}^n\}$. As $\ad{s}^i\notin J$ for all $i$ by hypothesis, the preimage $J'=\restr{\sigma}{I'}^{-1}(J)$ is contained in one of the intervals above.
		
		If $J'\subset(F_i\ad{s}^{i-1},F_i\ad{s}^i)$ for some $i\in\{2,\ldots,n-1\}$, the claim follows because all external addresses of the interval $(F_i\ad{s}^{i-1},F_i\ad{s}^i)$ have first entry $F_i$. If instead $J'\subset(\Fpred{(F_1)}\ad{s}^n,F_1\ad{s}^1)$, the claim follows from the fact that $J\subset I$ for some partition sector $I$: we have $\alpha\notin\cl_{\overline{\AD}}(J)$, so $J'\cap\sigma^{-1}(\alpha)=\emptyset$. Hence, either $J'\subset(\Fpred{(F_1)}\ad{s}^n,\alpha_{\Fpred{(F_1)}}^{F_1})\subset \Fpred{(F_1)}$ or $J'\subset(\alpha_{\Fpred{(F_1)}}^{F_1},F_1\ad{s}^1)\subset F_1$. The case $J'\subset(F_n\ad{s}^n,\Fsucc{(F_n)}\ad{s}^1)$ works analogously.
	\end{proof}

\begin{Pro}[Realized itineraries]
		Let $\mathcal{I}$ be a (not necessarily simple) dynamic partition. Then, for every periodic sequence $\itn{u}=\overline{I_0I_1\ldots  I_{m-1}}$ of itinerary domains, there exists a periodic external address $\ad{s}\in\mathcal{S}$ such that either $\It^+(\ad{s}~\vert~\mathcal{I})=\itn{u}$ or $\It^-(\ad{s}~\vert~\mathcal{I})=\itn{u}$.\label{pro:RealizedItineraries}
	\end{Pro}
	\begin{proof}
		The set
		\[
		T:=\bigcup_{i\in\{1,\ldots ,n\}}\bigcup_{j\geq 0}\{\sigma^{\circ j}(\ad{s}^i)\}
		\]
		is finite and forward invariant. By Proposition~\ref{pro:SectorTopology}, we have
		\[
		I_0\setminus T=J_1^{(0)}\cupdot\ldots \cupdot J_N^{(0)},
		\]
		where the $J_i^{(0)}$ are open intervals in $\AD$. Inductively, we define\[J_i^{(l)}:=\sigma^{-1}(J_i^{(l-1)})\cap I_{\mathtt{u}_{-l}}.\]As $\AD\setminus T$ is backward invariant, we have $J_i^{(l)}\cap T=\emptyset$ for all $i$ and $l$. Hence, Lemma~\ref{lem:IntervalPullbacks} implies that every $J_i^{(l)}$ is an open interval and the first $l$ entries are the same for all addresses in $J_i^{(l)}$.
		
		After $m$ pullbacks, we are back at subintervals of our initial partition sector $I_0$. We define $\rho\colon\{1,\ldots ,N\}\to\{1,\ldots ,N\}$ via $J_i^{(m)}\subset J_{\rho(i)}^{(0)}$.
		Choose $n_0$ and $i_0$ such that $\rho^{\circ n_0}(i_0)=i_0$, and set $J^{(l)}:=J_{i_0}^{(lmn_0)}$.
		The $J^{(l)}$ form a nested sequence of intervals. By Lemma~\ref{lem:IntervalPullbacks}, there exist fundamental domains $F_1,\ldots , F_{n_1}$
		(where $n_1:=mn_0$) such that every external address in $J^{(l)}$ begins with $l$ times the sequence $F_1\ldots  F_{n_1}$. Therefore, we have \[\ad{s}:=\overline{F_1\ldots  F_{n_1}}\in\bigcap_l \cl(J^{(l)}).\] Either the left- or the right-sided itinerary of $\ad{s}$ (or both) equals $\itn{u}$.
	\end{proof}
	
	\begin{Cor}[Itineraries are uniquely realized]
		Let $\itn{u}\in\mathcal{D}^\N$ be (pre\nobreakdash-)periodic. Then there exists a unique (pre\nobreakdash-)periodic point $p\in\mathcal{J}(\extended{f})$ such that $\itin{p}{\mathcal{D}}$ is adjacent to $\itn{u}$.\label{cor:UniquelyRealized}
	\end{Cor}
	\begin{proof}
		Lemma~\ref{lem:AdjacentItineraries} and Proposition~\ref{pro:RealizedItineraries} imply that there exists a (pre\nobreakdash-)periodic point $p\in\mathcal{J}(\extended{f})$ for which $\itin{p}{\mathcal{D}}$ is adjacent to $\itn{u}$. Proposition~\ref{pro:UniqueItineraries} says that this point is unique.
	\end{proof}
	
	This corollary does not imply that the map sending a (pre\nobreakdash-)periodic itinerary $\itn{u}$ to the unique (pre\nobreakdash-)periodic point $p\in\CTT{f}$ for which $\itin{p}{\mathcal{D}}$ is adjacent to $\itn{u}$ is a bijection. Indeed, this is never the case. As noted in the preceding section, we have $C(\extended{f})\neq\emptyset$ for all post-singularly finite maps $f$. The existence of critical points implies that several (pre\nobreakdash-)periodic itineraries correspond to the same (pre\nobreakdash-)periodic point.
	
	The remaining goal of this paper is to describe the landing equivalence relation in terms of itineraries w.r.t.\ some simple dynamic partition. The following lemma will be useful for the proof of the periodic case.
	\begin{Lem}
		Let $p\in\mathcal{J}(\extended{f})\setminus C(\extended{f})$ be (pre\nobreakdash-)periodic. Then there exists a partition sector $I\in\mathcal{I}$ such that the set
		\[
		\Ad(p):=\{\ad{s}\in\AD\colon L(\ad{s})=p\}
		\]
		satisfies $\Ad(p)\subset I^-$.\label{lem:SameSector}
	\end{Lem}
	\begin{proof}
		First, assume that $p\notin\partial_{\CTT{f}}\mathcal{D}$, so $p\in\extended{D}$ (see Definition~\ref{def:extendedPartition}) for some $D\in\mathcal{D}$. All filaments that land at $p$ are entirely contained in $\extended{D}$, so $\Ad(p)\subset I(D)$ by Proposition~\ref{pro:partition_correspondence}.
		
		If $p\in\partial_{\CTT{f}}\mathcal{D}$ is a Julia critical boundary point, it follows from $p\notin C(\extended{f})$ and Lemma~\ref{lem:LeftRight} that there is a partition sector $I\in\mathcal{I}$ so that $\Ad(p)\subset I^-$.
		
		If instead $p\in\partial_{\CTT{f}}\mathcal{D}$ is a Julia regular point, there is a Fatou center $c\in\extended{f}^{-1}(S(f))\cap\mathcal{F}(\extended{f})$ and a unique address $\ad{p}\in\Ad(p)$ such that $\ad{p}\in\Crit{c}$ (see the paragraph before Definition~\ref{def:CombinatorialPartition} for the definition of $\Crit{c}$). It follows from the fact that $G_{\ad{p}}[c]$ is the left supporting filament for $U(c)$ at $p$ that there exists a sector $I\in\mathcal{I}$ such that $\ad{\tilde{p}}\in I$ for all $\ad{\tilde{p}}\in\Ad(p)\setminus\{\ad{p}\}$ and $\ad{p}\in I^-\setminus I$.
	\end{proof}
	
	\begin{Pro}[Landing behavior of periodic filaments]~\label{pro:PeriodicLanding}
		Let $\mathcal{D}$ be a simple dynamic partition for the psf entire function $f$, and let $G_{\ad{t}}$ and $G_{\ad{u}}$ be periodic filaments. Then $G_{\ad{t}}$ and $G_{\ad{u}}$ land together if and only if $\itinl{\ad{t}}{\mathcal{I}}=\itinl{\ad{u}}{\mathcal{I}}$.
	\end{Pro}
	\begin{proof}
		First, assume that $G_{\ad{t}}$ and $G_{\ad{u}}$ have the same landing point $p\in\C$. As $p\notin C(\extended{f})$, Lemma~\ref{lem:SameSector} implies that $\ad{t},\ad{u}\in I^-$ for some $I\in\mathcal{I}$. The same reasoning applies to all points on the forward orbit of $p$, so we have $\itinl{\ad{t}}{\mathcal{I}}=\itinl{\ad{u}}{\mathcal{I}}$.
		
		Conversely, assume that $\itinl{\ad{t}}{\mathcal{I}}=\itinl{\ad{u}}{\mathcal{I}}$. By Lemma~\ref{lem:AdjacentItineraries}, there exists a periodic itinerary to which both $\itin{L(\ad{t})}{\mathcal{D}}$ and $\itin{L(\ad{u})}{\mathcal{D}}$ are adjacent. By Proposition~\ref{pro:UniqueItineraries}, this implies $L(\ad{t})=L(\ad{u})$.
	\end{proof}
	For the proof of the full statement, we need two additional lemmata.
	\begin{Lem}[Preimages and the landing equivalence]
		Let $\mathcal{I}$ be a simple dynamic partition, and let $I\in\mathcal{I}$ be a partition sector. Let $\ad{s},\ad{t}\in I^-$ be distinct (pre-)periodic addresses. Then $\ad{s}\landeq\ad{t}$ if and only if $\sigma(\ad{s})\landeq\sigma(\ad{t})$.\label{lem:PreimageLanding}
	\end{Lem}
	\begin{proof}
		Of course, $\ad{s}\landeq\ad{t}$ implies $\sigma(\ad{s})\landeq\sigma(\ad{t})$. For the other direction, note that both filaments $G_{\ad{s}}$ and $G_{\ad{t}}$ land at a preimage of $L(\sigma(\ad{s}))$. By Lemma~\ref{lem:SameSector}, there is only one preimage $p$ of $L(\sigma(\ad{s}))$ so that the addresses of the filaments that land at $p$ are contained in $I^-$. Hence, both $G_{\ad{s}}$ and $G_{\ad{t}}$ land at $p$.
	\end{proof}
	
	\begin{Lem}
		Let $\ad{s}\in\AD$ be (pre\nobreakdash-)periodic, and let $\mathcal{I}$ be a dynamic partition of $\AD$. Assume that $\itinl{\ad{s}}{\mathcal{I}}=\overline{I_0\ldots I_{n-1}}$ is periodic. Then $\ad{s}$ is periodic.\label{lem:NoPreperiodc}
	\end{Lem}
	\begin{proof}
		Assume to the contrary that $\ad{s}$ is preperiodic, and let $\ad{t}$ be the last preperiodic address contained in $\Omega^+(\ad{s})$. Then $\ad{u}:=\sigma(\ad{t})$ is periodic and has a unique periodic preimage $\ad{\tilde{t}}$. By hypothesis, we have $\ad{t},\ad{\tilde{t}}\in I_{j}^-$ for some $j\in\{0,\ldots,n-1\}$, $\ad{t}\neq\tilde{\ad{t}}$, and $\sigma(\ad{t})=\sigma(\ad{\tilde{t}})=\ad{u}$. This contradicts the fact that $\restr{\sigma}{I_j^-}$ is injective by Lemma~\ref{lem:LeftRight}.
	\end{proof}
	
	We are now ready to prove the main result of this paper.
	
	\begin{Thm}[The landing equivalence relation]
		Let $\mathcal{I}$ be a simple dynamic partition of $\AD$. Then the landing equivalence relation $\landeq$ on the set $\ADPer$ of (pre\nobreakdash-)periodic external addresses is the equivalence relation generated by the following relations:
		\begin{enumerate}
			\item We have $\ad{s}\landeq\ad{t}$ if $\itinl{\ad{s}}{\mathcal{I}}=\itinl{\ad{t}}{\mathcal{I}}$. \label{en:leftSided}
			\item We have $\ad{s}\landeq\ad{t}$ if there exists an $i_0\geq 0$ such that $I_i^-(\ad{s})=I_i^-(\ad{t})$ for all $i<i_0$ and $\sigma^{\circ i_0}(\ad{s}),\sigma^{\circ i_0}(\ad{t})\in\mathcal{C}(c)$ for some $c\in C(\extended{f})\cap\mathcal{J}(\extended{f})$. \label{en:JuliaCritical}
		\end{enumerate} \label{thm:landingEquivalence}
	\end{Thm}
	\begin{proof}
		It follows from Theorem~\ref{thm:PreperiodicLanding} that landing-related addresses have equal preperiod and period. Let $\sim$ denote the smallest equivalence relation on $\ADPer$ generated by the relations of type (\ref{en:leftSided}) and (\ref{en:JuliaCritical}). The statement of the theorem is that $\sim~=~\landeq$. By the above, one necessary condition for the theorem to hold is that addresses identified via $\sim$ have equal preperiod and period.
		\vspace{2mm}
		
		\textbf{Claim.} Assume that $\ad{s},\ad{t}\in\ADPer$ satisfy $\ad{s}\sim\ad{t}$. Then $\ad{s}$ and $\ad{t}$ have equal preperiod and period.
		\vspace{2mm}
		
		If $\ad{s}$ and $\ad{t}$ are related via (\ref{en:leftSided}), then $\itinl{\ad{s}}{\mathcal{I}}=\itinl{\ad{t}}{\mathcal{I}}$. It follows from Lemma~\ref{lem:NoPreperiodc} that $\ad{s}$ and $\ad{t}$ have equal preperiod and period.
		If $\ad{s}$ and $\ad{t}$ are related via (\ref{en:JuliaCritical}), then there exists a point $c\in C(\extended{f})\cap\mathcal{J}(\extended{f})$ and an $i_0\geq 0$ such that $\sigma^{\circ i_0}(\ad{s}),\sigma^{\circ i_0}(\ad{t})\in\Crit{c}$. As all addresses in $\Crit{c}$ are preperiodic of equal preperiod and period, the same is true for $\ad{s}$ and $\ad{t}$.\hfill$\triangle$
		\vspace{2mm}
		
		This allows us to prove the theorem by induction over the length of the preperiod. First, assume that $\ad{s}$ and $\ad{t}$ are periodic. Then Proposition~\ref{pro:PeriodicLanding} says that $\ad{s}\landeq\ad{t}$ if and only if $\itinl{\ad{s}}{\mathcal{I}}=\itinl{\ad{t}}{\mathcal{I}}$. The addresses $\ad{s}$ and $\ad{t}$ cannot be related by a relation of type (\ref{en:JuliaCritical}), because for a point $c\in C(\extended{f})\cap\mathcal{J}(\extended{f})$ all addresses contained in $\mathcal{C}(c)$ are preperiodic. Therefore, we have $\ad{s}\landeq\ad{t}$ if and only if $\ad{s}\sim\ad{t}$.
		
		Assume that there exists an $m\in\N_0$ such that the theorem is true for all external addresses $\ad{s}$ such that $\sigma^{\circ m}(\ad{s})$ is periodic. Let $\ad{s},\ad{t}$ be preperiodic addresses such that $\sigma^{\circ m+1}(\ad{s})$ and $\sigma^{\circ m+1}(\ad{t})$ are periodic, but $\sigma^{\circ m}(\ad{s})$ and $\sigma^{\circ m}(\ad{t})$ are not. We need to show that $\ad{s}\landeq\ad{t}$ if and only if $\ad{s}\sim\ad{t}$.
		
		First, assume that $\ad{s}\landeq\ad{t}$, and let $p=L(\ad{s})=L(\ad{t})\in\CTT{f}$ be the common landing point of $G_{\ad{s}}$ and $G_{\ad{t}}$. The filaments $G_{\sigma(\ad{s})}$ and $G_{\sigma(\ad{t})}$ are landing together at $\extended{f}(p)$, so we also have $\sigma(\ad{s})\landeq\sigma(\ad{t})$. By the inductive hypothesis, the addresses $\sigma(\ad{s})$ and $\sigma(\ad{t})$ are related by a finite chain of relations of type (\ref{en:leftSided}) or (\ref{en:JuliaCritical}), say \[\sigma(\ad{s})=:\tilde{\ad{u}}^0\sim\tilde{\ad{u}}^1\sim\ldots \sim\tilde{\ad{u}}^{m-1}\sim\tilde{\ad{u}}^m:=\sigma(\ad{t}),\]where each address in the chain is directly related to the next either by (\ref{en:leftSided}) or by (\ref{en:JuliaCritical}). All filaments $G_{\tilde{\ad{u}}^i}$ land at $\extended{f}(p)$ by the inductive hypothesis.
		
		If $p\notin C(\extended{f})$, then there exists a partition sector $I\in\mathcal{I}$ such that $\ad{s}\in I^-$ for all addresses satisfying $L(\ad{s})=p$ by Lemma~\ref{lem:SameSector}. Let $\ad{u}^i$ be the unique preimage of $\tilde{\ad{u}}^i$ in $I^-$. Then every $G_{\ad{u}^i}$ lands at $p$ by Lemma~\ref{lem:PreimageLanding}. No matter whether $\tilde{\ad{u}}^i$ and $\tilde{\ad{u}}^{i+1}$ are related by (\ref{en:leftSided}) or (\ref{en:JuliaCritical}), it follows from $\ad{u}^i,\ad{u}^{i+1}\in I^-$ that $\ad{u}^i\sim\ad{u}^{i+1}$ are related of the same type. As $\ad{u}^0=\ad{s}$ and $\ad{u}^{m}=\ad{t}$, this implies $\ad{s}\sim\ad{t}$.
		
		Otherwise, the point $p=c\in C(\extended{f})\cap\mathcal{J}(\extended{f})$ is critical. Let $I$ be the partition sector for which $\ad{s}\in I^-$. There exists an address $\ad{c}\in\Crit{c}\cap I^-$. By the inductive hypothesis, we have $\sigma(\ad{c})\sim\sigma(\ad{s})$. As in the preceding paragraph, it follows from $\ad{c},\ad{s}\in I^-$ that $\ad{s}\sim\ad{c}$. In the same way, we find an address $\tilde{\ad{c}}\in\Crit{c}$ such that $\ad{t}\sim\tilde{\ad{c}}$. The addresses $\ad{c},\tilde{\ad{c}}\in\Crit{c}$ are related via (\ref{en:JuliaCritical}), so we have $\ad{s}\sim\ad{c}\sim\tilde{\ad{c}}\sim\ad{t}$.
		
		Conversely, assume that $\ad{s}\sim\ad{t}$. We want to show that the filaments $G_{\ad{s}}$ and $G_{\ad{t}}$ are landing together. By transitivity, we can restrict to the case that
		$\ad{s}$ and $\ad{t}$ are directly related via (\ref{en:leftSided}) or (\ref{en:JuliaCritical}).
		
		If $\ad{s}$ and $\ad{t}$ are related via (\ref{en:leftSided}), the same is true for $\sigma(\ad{s})$ and $\sigma(\ad{t})$. By the inductive hypothesis, the shifts $\sigma(\ad{s})\landeq\sigma(\ad{t})$ are landing-related. As we have $I_0^-(\ad{s})=I_0^-(\ad{t})$, it follows from Lemma~\ref{lem:PreimageLanding} that $\ad{s}\landeq\ad{t}$.
		
		If $\ad{s}\sim\ad{t}$ via (\ref{en:JuliaCritical}), there exists an $i_0\geq 0$ such that $I_i^-(\ad{s})=I_i^-(\ad{t})$ for all $i<i_0$ and $\sigma^{\circ i_0}(\ad{s}),\sigma^{\circ i_0}(\ad{t})\in\mathcal{C}(c)$ for some $c\in C(\extended{f})\cap\mathcal{J}(\extended{f})$. If we have $i_0=0$, then $\ad{s},\ad{t}\in\Crit{c}$ and all addresses in $\Crit{c}$ are landing-related. Otherwise, the shifts $\sigma(\ad{s})$ and $\sigma(\ad{t})$ are also related via (\ref{en:JuliaCritical}), so $\sigma(\ad{s})\landeq\sigma(\ad{t})$ by the inductive hypothesis. As before, Lemma~\ref{lem:PreimageLanding} implies $\ad{s}\landeq\ad{t}$.
	\end{proof}

	\let\oldaddcontentsline\addcontentsline
	\renewcommand{\addcontentsline}[3]{}

	\let\addcontentsline\oldaddcontentsline

\begin{thebibliography}{McN1}
		
		\bibitem[BFH92]{BFH}
		Bielefeld, Ben; Fisher, Yuval; Hubbard, John. The classification of critically preperiodic polynomials as dynamical systems. J. Amer. Math. Soc. 5 (1992), no. 4, 721--762.
		
		\bibitem[BJR12]{BJR}
		Bara\'nski, Krzysztof; Jarque, Xavier; Rempe, Lasse. Brushing the hairs of transcendental entire functions. Topology Appl. 159 (2012), no. 8, 2102--2114.
		
		\bibitem[BRG20]{BR}
		Benini, Anna Miriam; Rempe-Gillen, Lasse. A landing theorem for entire functions with bounded post-singular sets. Geometric and Functional Analysis 30 (2020), 1465--1530.
		
		\bibitem[Ere89]{E1}
		Eremenko, Alexandre. On the iteration of entire functions. Dynamical systems and ergodic theory (Warsaw, 1986), 339--345, Banach Center Publ., 23, PWN, Warsaw, 1989.
		
		\bibitem[Ere13]{E2}
		Eremenko, Alexandre. Singularities of inverse functions. Presented at the ICMS conference ``The role of complex analysis in complex dynamics'', 2013.
		
		\bibitem[For81]{F}
		Forster, Otto. Lectures on Riemann surfaces. Translated from the German by Bruce Gilligan. Graduate Texts in Mathematics, 81. Springer-Verlag, New York-Berlin, 1981. {\rm viii}+254 pp.
		
		\bibitem[GK86]{GK}
		Goldberg, Lisa; Keen, Linda. A finiteness theorem for a dynamical class of entire functions. Ergodic Theory Dynam. Systems 6 (1986), no. 2, 183--192.

\bibitem[HS]{Spiders}
Hubbard, John H.; Schleicher, Dierk. The spider algorithm. In: AMS Symp. Appl. Math.~49 (1994), 155--180 (1994). 		

    \bibitem[Kam22]{Kammeyer}
    Kammeyer, Holger. Introduction to algebraic topology. Compact Textbooks in Mathematics. Birkhäuser/Springer, Cham, 2022. viii+182 pp. 

		\bibitem[Mil06]{M}
		Milnor, John. Dynamics in one complex variable. Third edition. Annals of Mathematics Studies, 160. Princeton University Press, Princeton, NJ, 2006. viii+304 pp.
		
		\bibitem[MB09]{B}
		Mihaljevi\'c-Brandt, Helena. Topological Dynamics of Transcendental Entire Functions. PhD Thesis, University of Liverpool, 2009.
		
		\bibitem[Mun00]{Mu}
		Munkres, James. Topology. Second edition. Prentice Hall, Incorporated, 2000. 537 pp.
		
		\bibitem[P19]{DavidThesis}
		Pfrang, David. Homotopy Hubbard Trees for Post-Singularly Finite Transcendental Entire Maps. PhD Thesis, Jacobs University Bremen, 2019.

		\bibitem[PRS18]{PRS}
		Pfrang, David; Rothgang, Michael; Schleicher, Dierk. Homotopy Hubbard trees for post-singularly finite exponential maps.	
Ergodic Theory and Dynamical Systems 43 1 (2023), 253--298. doi 10.1017/etds.2021.103.

\bibitem[Poi09]{P1}
		Poirier, Alfredo. Critical portraits for postcritically finite polynomials. Fund. Math. 203 (2009), no. 2, 107--163.

\bibitem[Rei22]{Bernhard}
Reinke, Bernhard. Iterated monodromy groups of entire maps and dendroid automata. Arxiv 2210.10385 (2022). 

\bibitem[Rem08]{R1}
		Rempe, Lasse. Siegel disks and periodic rays of entire functions. J. Reine Angew. Math. 624 (2008), 81--102.
		
		\bibitem[RRRS11]{RRRS}
		Rottenfu{\ss}er, G\"unter; R\"uckert, Johannes; Rempe, Lasse; Schleicher, Dierk. Dynamic rays of bounded-type entire functions. Ann. of Math. (2) 173 (2011), no. 1, 77--125.
		
		\bibitem[Sch10]{S2}
		Schleicher, Dierk. Dynamics of entire functions. Holomorphic dynamical systems, 295--339, Lecture Notes in Math., 1998, Springer, Berlin, 2010.
		
		\bibitem[SZ03a]{SZ1}
		Schleicher, Dierk; Zimmer, Johannes. Escaping points of exponential maps. J. London Math. Soc. (2) 67 (2003), no. 2, 380--400.
		
		\bibitem[SZ03b]{SZ2}
		Schleicher, Dierk; Zimmer, Johannes. Periodic points and dynamic rays of exponential maps. Ann. Acad. Sci. Fenn. Math. 28 (2003), no. 2, 327--354.
		
	\end{thebibliography}
\end{document}